\documentclass[10pt, leqno]{amsart}
\usepackage{amsthm,amsmath,amssymb,amscd,verbatim}
\usepackage[curve,matrix,arrow,frame,tips]{xy}
\usepackage{pdfsync}
\usepackage{hyperref}

\theoremstyle{plain}
\newtheorem{theorem}{Theorem}[section]
\newtheorem{lemma}[theorem]{Lemma}
\newtheorem{cor}[theorem]{Corollary}
\newtheorem{prop}[theorem]{Proposition}

\newtheorem{conj}[theorem]{Conjecture}

\theoremstyle{remark}

\theoremstyle{definition}

\newtheorem{defn}[theorem]{Definition}
\newtheorem{Remark}[theorem]{Remark}

\numberwithin{equation}{subsection}
\numberwithin{theorem}{subsection}

\newcommand\hfld[2]{\smash{\mathop{\hbox to 10mm{\rightarrowfill}}
     \limits^{\scriptstyle#1}_{\scriptstyle#2}}}
\newcommand\hflg[2]{\smash{\mathop{\hbox to 10mm{\leftarrowfill}}
     \limits^{\scriptstyle#1}_{\scriptstyle#2}}}

\title{The stable Bernstein center and \\ test functions for Shimura varieties}

\author{Thomas J. Haines}

\begin{document}

\thanks{Research partially supported by NSF DMS-0901723.}

\date{}

\maketitle

\begin{abstract} \begin{small}
We elaborate the theory of the stable Bernstein center of a $p$-adic group $G$, and apply it to state a general conjecture on test functions for Shimura varieties due to the author and R.~Kottwitz.  This conjecture provides a framework by which one might pursue the Langlands-Kottwitz method in a very general situation: not necessarily PEL Shimura varieties with arbitrary level structure at $p$. We give a concrete reinterpretation of the test function conjecture in the context of parahoric level structure. We also use the stable Bernstein center to formulate some of the transfer conjectures (the ``fundamental lemmas'')  that would be needed if one attempts to use the test function conjecture to express the local Hasse-Weil zeta function of a Shimura variety in terms of automorphic $L$-functions.
\end{small}
\end{abstract}

\markboth{T. Haines}
{Stable Bernstein center and test functions}

\tableofcontents

\section{Introduction}

The main purpose of this chapter is to give precise statements of some conjectures on test functions for Shimura varieties with bad reduction.  

In the Langlands-Kottwitz approach to studying the cohomology of a Shimura variety, one of the main steps is to identify a suitable test function that is ``plugged into'' the counting points formula that resembles the geometric side of the Arthur-Selberg trace formula.  To be more precise, let $({\bf G}, h^{-1}, K^pK_p)$ denote Shimura data where $p$ is a fixed rational prime such that the level-structure group factorizes as $K^p K_p \subset {\bf G}(\mathbb A^p_f)  {\bf G}(\mathbb Q_p)$. This data gives rise to a quasi-projective variety $Sh_{K_p} := Sh({\bf G}, h^{-1}, K^pK_p)$ over a number field ${\bf E} \subset \mathbb C$. Let $\Phi_p \in {\rm Gal}(\overline{\mathbb Q}_p/\mathbb Q_p)$ denote a geometric Frobenius element.  Then one seeks to prove a formula for the semi-simple Lefschetz number ${\rm Lef}^{\rm ss}(\Phi^r_p, Sh_{K_p})$ 
\begin{equation} \label{Lef^ss_intro}
{\rm tr}^{\rm ss}(\Phi^r_p, {\rm H}^\bullet_c(Sh_{K_p} \otimes_{\bf E} \overline{\mathbb Q}_p \, , \overline{\mathbb Q}_\ell)) = \sum_{(\gamma_0; \gamma, \delta)} c(\gamma_0; \gamma, \delta) ~ {\rm O}_\gamma(1_{K^p}) ~ {\rm TO}_{\delta \theta}(\phi_r),
\end{equation}
(see $\S\ref{test_fcn_subsec}$ for more details).

The {\em test function} $\phi_r$ appearing here is the most interesting part of the formula. Experience has shown that we may often find a test function belonging to the center $\mathcal Z({\bf G}(\mathbb Q_{p^r}), K_{p^r})$ of the Hecke algebra $\mathcal H({\bf G}(\mathbb Q_{p^r}), K_{p^r})$, in a way that is explicitly determined by the ${\bf E}$-rational conjugacy class $\{ \mu \}$ of 1-parameter subgroups of ${\bf G}$ associated to the Shimura data. In most PEL cases with good reduction, where $K_p \subset {\bf G}(\mathbb Q_p)$ is a hyperspecial maximal compact subgroup, this was done by Kottwitz (cf.~e.g.~\cite{Ko92a}). When $K_p$ is a parahoric subgroup of ${\bf G}(\mathbb Q_p)$ and when ${\bf G}_{\mathbb Q_p}$ is unramified, the {\em Kottwitz conjecture} predicts that we can take $\phi_r$ to be a power of $p$ times  the {\em Bernstein function} $z^{K_p}_{-\mu, j}$ arising from the Bernstein isomorphism for the center $\mathcal Z({\bf G}(\mathbb Q_{p^r}), K_{p^r})$ of the parahoric Hecke algebra $\mathcal H({\bf G}(\mathbb Q_{p^r}), K_{p^r})$ (see Conjecture \ref{Kottwitz_conj} and $\S \ref{appendix_sec}$).

In fact Kottwitz formulated (again, for unramified groups coming from Shimura data) a closely related conjecture concerning nearby cycles on Rapoport-Zink local models of Shimura varieties, which subsequently played an important role in the study of local models (Conjecture \ref{Kottwitz_conj_RPsi}). It also inspired important developments in the geometric Langlands program, e.g.~\cite{Ga}.  Both versions of Kottwitz' conjectures were later proved in several parahoric cases attached to linear or symplectic groups (see \cite{HN02a,H05}). In a recent breakthrough, Pappas and Zhu \cite{PZ} defined group-theoretic versions of Rapoport-Zink local models for quite general groups, and proved in the unramified situations the analogue of Kottwitz' nearby cycles conjecture for them. These matters are discussed in more detail in $\S\ref{parahoric_TFC_sec}$ and $\S\ref{evidence_TFC_sec}$.

Until around 2009 it was still not clear how one could describe the test functions $\phi_r$ in all deeper level situations. In the spring of 2009 the author and Kottwitz formulated a conjecture predicting test functions $\phi_r$ for general level structure $K_p$. This is the {\bf test function conjecture}, Conjecture \ref{TFC}.  It postulates that we may express $\phi_r$ in terms of a distribution $Z_{V^{E_{j0}}_{-\mu,j}}$ in the Bernstein center $\mathfrak Z({\bf G}(\mathbb Q_{p^r}))$ associated to a certain representation $V^{E_{j0}}_{-\mu, j}$ (defined in (\ref{test_rep_defn})) of the Langlands $L$-group $^L(G_{\mathbb Q_{p^r}})$. Let $d = {\rm dim}(Sh_{K_p})$. Then Conjecture \ref{TFC} asserts that we may take
$$
\phi_r = p^{rd/2} \big(Z_{V^{E_{j0}}_{-\mu, j}} * 1_{K_{p^r}}\big) \in \mathcal Z({\bf G}(\mathbb Q_{p^r}), K_{p^r})
$$
the convolution of the distribution $Z_{V^{E_{j0}}_{-\mu, j}}$ with the characteristic function $1_{K_{p^r}}$ of the subgroup $K_{p^r}$. As shown in $\S\ref{parahoric_TFC_sec}$, this specializes to the Kottwitz conjecture for parahoric subgroups in unramified groups. Conjecture \ref{TFC} was subsequently proved for Drinfeld case Shimura varieties with $\Gamma_1(p)$-level structure by the author and Rapoport \cite{HRa}, and for modular curves and for Drinfeld case Shimura varieties with arbitrary level structure by Scholze \cite{Sch1, Sch2}.  

The distributions in Conjecture \ref{TFC} are best seen as examples of a construction $V \leadsto Z_V$ which attaches to any algebraic representation $V$ of the Langlands dual group $^LG$ (for $G$ any connected reductive group over any $p$-adic field $F$), an element $Z_V$ in  the {\bf stable Bernstein center} of $G/F$.  This chapter elaborates the theory of the stable Bernstein center, following the lead of Vogan \cite{Vo}.  The set of all {\em infinitesimal characters}, i.e. the set of all $\widehat{G}$-conjugacy classes of admissible homomorphisms $\lambda: W_F \rightarrow \, ^LG$ (where $W_F$ is the Weil group of the local field $F$), is given the structure of an affine algebraic variety over $\mathbb C$, and the stable Bernstein center $\mathfrak Z^{\rm st}(G/F)$ is defined to be the ring of regular functions on this variety.\footnote{The difference between our treatment and Vogan's is in the definition of the variety structure on the set of infinitesimal characters.}

In order to describe the precise conjectural relation between the Bernstein and stable Bernstein centers of a $p$-adic group, it was necessary to formulate an enhancement ${\rm LLC+}$ of the usual conjectural local Langlands correspondence ${\rm LLC}$ for that group. Having this relation in hand, the construction $V \leadsto Z_V$ provides a supply of elements in the usual Bernstein center of $G/F$, which we call the {\em geometric Bernstein center}. 

It is for such distributions that one can formulate natural candidates for (Frobenius-twisted) endoscopic transfer, which we illustrate for standard endoscopy in Conjecture \ref{end_trans_conj} and for stable base change in Conjecture \ref{BC_conj}. These form part of the cadre of ``fundamental lemmas'' that one would need to pursue the ``pseudostabilization'' of (\ref{Lef^ss_intro}) and thereby express the cohomology of $Sh_{K_p}$ in terms of automorphic representations along the lines envisioned by Kottwitz \cite{Ko90} but for places with arbitrary bad reduction.  In the compact and non-endoscopic situations, we prove in Theorem \ref{Z_fcn_thm} that the various Conjectures we have made yield an expression of the semi-simple local Hasse-Weil zeta function in terms of semi-simple automorphic $L$-functions. Earlier unconditional results in this direction, for nice PEL situations, were established in \cite{H05}, \cite{HRa}, \cite{Sch1, Sch2}. We stress that the framework here should not be limited to PEL Shimura varieties, but should work more generally.

In recent work of Scholze and Shin \cite{SS}, the connection of the stable Bernstein center with Shimura varieties helped them to give nearly complete descriptions of the cohomology of many compact unitary Shimura varieties with bad reduction at $p$; they consider the ``EL cases'' where ${\bf G}_{\mathbb Q_p}$ is a product of Weil restrictions of general linear groups .  It would be interesting to extend the connection to further examples.

Returning to the original Kottwitz conjecture for parahoric level structure,  Conjecture \ref{TFC} in some sense subsumes it, since it makes sense for arbitrary level structure and without the hypothesis that ${\bf G}_{\mathbb Q_{p^r}}$ be unramified. However, Conjecture \ref{TFC} has the drawback that it assumes ${\rm LLC+}$ for ${\bf G}_{\mathbb Q_{p^r}}$. Further, it is still of interest to formulate the Kottwitz conjecture in the parahoric cases for {\em arbitrary groups}  in a concrete way that can be checked (for example) by explicit comparison of test functions with nearby cycles. In  $\S\ref{parahoric_TFC_sec}$ we formulate the Kottwitz conjecture for general groups, making use of the {\rm transfer homomorphisms}  of the Appendix $\S \ref{appendix_sec}$ to determine test functions on arbitrary groups from test functions on their quasi-split inner forms. The definition of transfer homomorphisms requires a theory of Bernstein isomorphisms more general than what was heretofore available. Therefore, in the Appendix we establish these isomorphisms in complete generality in a nearly self-contained way, and also provide some related structure theory results that should be of independent interest.

Here is an outline of the contents of this chapter. In $\S\ref{bernstein_cntr_rev}$ we review the Bernstein center of a $p$-adic group, including the algebraic structure on the Bernstein variety of all supercuspidal supports. In $\S\ref{LLC_sec}$ we recall the conjectural local Langlands correspondence (LLC), and discuss additional desiderata we need in our elaboration of the stable Bernstein center in $\S \ref{st_bernstein_cntr_sec}$. In particular in $\S\ref{LLC+_subsec}$ we describe the enhancement ${\rm (LLC+)}$ which plays a significant role throughout the chapter, and explain why it holds for general linear groups in Remark \ref{comp_w_ind_rem} and Corollary \ref{invariant_Levi}.  The distributions $Z_V$ are defined in $\S\ref{const_Z_V_subsec}$, and are used to formulate the test function conjecture, Conjecture \ref{TFC}, in $\S \ref{test_fcn_subsec}$. In the rest of $\S\ref{LK_approach_sec}$, we describe the nearby cycles variant Conjecture \ref{RPsi_TFC} along with some of the endoscopic transfer conjectures needed for the ``pseudostabilization'', and assuming these conjectures we prove in Theorem \ref{Z_fcn_thm}  the expected form of the semi-simple local Hasse-Weil zeta functions, in the compact and non-endoscopic cases.  In $\S\ref{parahoric_TFC_sec}$ we give a concrete reformulation of the test function conjecture in parahoric cases, recovering the Kottwitz conjecture and generalizing it to all groups using the material from the Appendix. The purpose of $\S\ref{evidence_TFC_sec}$ and $\S\ref{evidence_transfer_sec}$ is to list some of the available evidence for Conjectures \ref{TFC} and \ref{BC_conj}. In $\S\ref{explicit_sec}$ certain test functions are described very explicitly. Finally, the Appendix gives the treatment of Bernstein isomorphisms and the transfer homomorphisms, alluded to above.

\medskip

{\bf Acknowledgments.}
I am very grateful to Guy Henniart for supplying the proof of Proposition \ref{GL_n_invariance} and for allowing me to include his proof in this chapter.  I warmly thank Timo Richarz for sending me his unpublished article \cite{Ri} and for letting me quote a few of his results in Lemma \ref{Timo}.  I am indebted to Brooks Roberts for proving Conjecture \ref{comp_w_induction} for ${\rm GSp}(4)$ (see Remark \ref{comp_w_ind_rem}).  I thank my colleagues Jeffrey Adams and Niranjan Ramachandran for useful conversations. I also thank Robert Kottwitz for his influence on the ideas in this chapter and for his comments on a preliminary version. I thank Michael Rapoport for many stimulating conversations about test functions over the years. I am grateful to the referee for helpful suggestions and remarks.

\section{Notation}

If $G$ is a connected reductive group over a $p$-adic field $F$, then $\mathfrak R(G)$ will denote the category of smooth representations of $G(F)$ on $\mathbb C$-vector spaces.  We will write $\pi \in \mathfrak R(G)_{\rm irred}$ or $\pi \in \Pi(G/F)$ if $\pi$ is an irreducible object in $\mathfrak R(G)$.

If $G$ as above contains an $F$-rational parabolic subgroup $P$ with $F$-Levi factor $M$ and unipotent radical $N$, define the modulus function $\delta_P : M(F) \rightarrow \mathbb R_{>0}$ by
$$
\delta_P(m) = |{\rm det}({\rm Ad}(m) \, ; \, {\rm Lie}(N(F)))|_F
$$
where $|\cdot|_F$ is the normalized absolute value on $F$. By $\delta^{1/2}_P(m)$ we mean the positive square-root of the positive real number $\delta_P(m)$.  For $\sigma \in \mathfrak R(M)$, we frequently consider the normalized induced representation
$$
i^G_P(\sigma) = {\rm Ind}_{P(F)}^{G(F)}(\delta_P^{1/2}\sigma).
$$

We let $1_S$ denote the characteristic function of a subset $S$ of some ambient space. If $S \subset G$, let $^g S = gSg^{-1}$. If $f$ is a function on $S$, define the function $^gf$ on $^gS$ by $^gf(\cdot) = f(g^{-1}\cdot g)$.

Throughout the chapter we use the Weil form of the local or global Langlands $L$-group $^LG$.

\section{Review of the Bernstein center} \label{bernstein_cntr_rev}

We shall give a brief synopsis of \cite{BD} that is suitable for our purposes.  Other useful references are \cite{Be92}, \cite{Ren}, and \cite{Roc}.

The Bernstein center $\mathfrak Z(G)$ of a $p$-adic group $G$ is defined as the ring of endomorphisms of the identity functor on the category of smooth representations $\mathfrak R(G)$.  It can also be realized as an algebra of certain distributions, as the projective limit of the centers of the finite-level Hecke algebras, and as the ring of regular functions on a certain algebraic variety.  We describe these in turn.

\subsection{Distributions} \label{distribution_sec}

We start by defining the convolution algebra of distributions.

We write $G$ for the rational points of a connected reductive group over a $p$-adic field.  
Thus $G$ is a totally disconnected locally compact Hausdorff topological group.  Further $G$ is unimodular; fix a Haar measure $dx$. Let $C^\infty_c(G)$ denote the set of $\mathbb C$-valued compactly supported and locally constant functions on $G$.   Let $\mathcal H(G, \, dx) = (C_c^\infty(G), *_{dx})$, the convolution product $*_{dx}$ being defined using the Haar measure $dx$.  

A {\em distribution} is a $\mathbb C$-linear map $D: C^\infty_c(G) \rightarrow \mathbb C$.  For each $f \in C^\infty(G)$ we define $\breve{f} \in C^\infty(G)$ by $\breve{f}(x) = f(x^{-1})$ for $x \in G$.  We set
$$
\breve{D}(f) := D(\breve{f}).
$$
We can convolve a distribution $D$ with a function $f \in C^\infty_c(G)$ and get a new function $D * f \in C^\infty(G)$, by setting 
$$
(D * f)(g) = \breve{D}(g \cdot f),
$$
where $(g \cdot f)(x) := f(xg)$.  The function $D *f $ does not automatically have compact support. We say $D$ is {\em essentially compact} provided that $D * f \in C^\infty_c(G)$ for every $f \in C^\infty_c(G)$.  

We define $^gf$ by $^g f(x) := f(g^{-1}xg)$ for $x,g \in G$.  We say that $D$ is $G$-{\em invariant} if $D(\,^gf) = D(f)$ for all $g, f$. The set $\mathcal D(G)^G_{\rm ec}$ of $G$-invariant essentially compact distributions on $C^\infty_c(G)$ turns out to have the structure of a commutative $\mathbb C$-algebra.  We describe next the convolution product and its properties.

Given distributions $D_1, D_2$ with $D_2$ essentially compact, we define another distribution $D_1 * D_2$ by
$$
(D_1 * D_2)(f) = \breve{D}_1( D_2 * \breve{f}).
$$

\begin{lemma} \label{D_properties}
The convolution products $D *f$ and $D_1 * D_2$ have the following properties:
\begin{enumerate}
\item[(a)] For $\phi \in C^\infty_c(G)$ let $D_{\phi\, dx}$ (sometimes abbreviated $\phi \, dx$) denote the essentially compact distribution given by $f \mapsto \int_G f(x)\phi(x) \,dx$.  Then $D_{\phi \, dx} * f = \phi *_{dx} f$.
\item[(b)] If $f \in C^\infty_c(G)$, then $D* (f\,dx) = (D * f)\, dx$.  In particular, $D_{\phi_1 \, dx} * D_{\phi_2 \, dx} = D_{\phi_1 *_{dx} \phi_2 \, dx}$.
\item[(c)] If $D_2$ is essentially compact, then $(D_1 * D_2) * f = D_1 *(D_2 * f)$.  If $D_1$ and $D_2$ are each essentially compact, so is $D_1 * D_2$.
\item[(d)] If $D_2$ and $D_3$ are essentially compact, then $(D_1 * D_2) * D_3 = D_1 * (D_2 * D_3)$.
\item[(e)] An essentially compact distribution $D$ is $G$-invariant if and only if $D * (1_{Ug}\, dx) = (1_{Ug}\, dx) *D$ for all compact open subgroups $U \subset G$ and $g \in G$.  Here $1_{Ug}$ is the characteristic function of the set $Ug$.  
\item[(f)] If $D$ is essentially compact and $f_1, f_2 \in C^\infty_c(G)$, then 
$D * (f_1 *_{dx} f_2) = (D * f_1) *_{dx} f_2$.
\end{enumerate}
\end{lemma}

\begin{cor} The pair $(\mathcal D(G)^G_{\rm ec}, *)$ is a commutative and associative $\mathbb C$-algebra.
\end{cor}

\subsection{The projective limit}

Let $J \subset G$ range over the set of all compact open subgroups of $G$.   Let $\mathcal H(G)$ denote the convolution algebra of compactly-supported measures on $G$, and let $\mathcal H_J(G) \subset \mathcal H(G)$ denote the ring of $J$-bi-invariant compactly-supported measures, with center $\mathcal Z_J(G)$.  The ring $\mathcal H_J(G)$ has as unit $e_J = 1_J \, dx_J$, where $1_J$ is the characteristic function of $J$ and $dx_J$ is the Haar measure with ${\rm vol}_{dx_J}(J) = 1$.  Note that if $J' \subset J$, then $dx_{J'} = [J : J'] \, dx_J$.

Let $\mathcal Z(G, J)$ denote the center of the algebra $\mathcal H(G, J)$ consisting of compactly-supported $J$-bi-invariant functions on $G$ with product $*_{dx_J}$.  There is an isomorphism $\mathcal Z(G,J) ~\widetilde{\rightarrow} ~ \mathcal Z_J(G)$ by $z_J \mapsto z_J dx_J$.  

For $J' \subset J$ there is an algebra map $\mathcal Z(G,J') \rightarrow \mathcal Z(G,J)$, given by $z_{J'} \mapsto z_{J'} *_{dx_{J'}} 1_J$.  Equivalently, we have $\mathcal Z_{J'}(G) \rightarrow \mathcal Z_J(G)$ given by $z_{J'} dx_{J'} \mapsto z_{J'} dx_{J'} * (1_J dx_J)$.  

We can view $\mathfrak R(G)$ as the category of non-degenerate smooth $\mathcal H(G)$-modules, and 
any element of $\varprojlim \mathcal Z(G,J)$ acts on objects in $\mathfrak R(G)$ in a way that commutes with the action of $\mathcal H(G)$.  Hence there is a canonical homomorphism $\varprojlim \mathcal Z(G, J) \rightarrow \mathfrak Z(G)$. 

There is also a canonical homomorphism 
\begin{equation*}
\varprojlim Z(G,J) \rightarrow \mathcal D(G)^G_{\rm ec}
\end{equation*} 
since $Z = (z_J)_J \in \varprojlim \mathcal Z(G, J)$ gives a distribution on $f \in C^\infty_c(G)$ as follows: choose $J \subset G$ sufficiently small that $f$ is right-$J$-invariant, and set
\begin{equation} \label{Z_dist_defn}
Z(f) = \int_G z_J(x) \, f(x) \, dx_J.
\end{equation}
This is independent of the choice of $J$. Note that for $f \in \mathcal H(G, J)$ we have $Z * f = z_J *_{dx_J} f$, and in particular $Z * 1_J = z_J$, for all $J$. To see $Z = (z_J)_J$ as a distribution is really $G$-invariant, note that for $f \in \mathcal H(G, J)$, the identities $Z * f = z_J *_{dx_J} f = f *_{dx_J} z_J$ imply that $Z *(fdx) = (fdx) *Z$.  This in turn shows that $Z$ is $G$-invariant by Lemma \ref{D_properties}(e). 
 
Now $\S1.4-1.7$ of \cite{BD} show that the above maps yield isomorphisms 
\begin{equation} \label{Bern_cntr_as_distr}
\mathfrak Z(G) ~ \widetilde{\leftarrow} ~ \varprojlim \mathcal Z(G, J) ~ \widetilde{\rightarrow} ~ \mathcal D(G)^G_{\rm ec}.
\end{equation}

\begin{cor} \label{Hecke_action_cor}
Let $Z \in \mathfrak Z(G)$, and suppose a finite-length representation $\pi \in \mathfrak R(G)$ has the property that $Z$ acts on $\pi$ by a scalar $Z(\pi)$. 
\begin{enumerate}
\item[(a)] For every compact open subgroup $J \subset G$, $Z * 1_J$ acts on the left on $\pi^J$ by the scalar $Z(\pi)$.
\item[(b)] For every $f \in \mathcal H(G)$, ${\rm tr}(Z*f \, |\, \pi) = Z(\pi)\, {\rm tr}(f \, | \, \pi)$.
\end{enumerate}
\end{cor}

\subsection{Regular functions on the variety of supercuspidal supports}

\subsubsection{Variety structure on set of supercuspidal supports}
We describe the variety of supercuspidal supports in some detail. Also we will describe it in a slightly unconventional way, in that we use the Kottwitz homomorphism to parametrize the (weakly) unramified characters on $G(F)$. This will be useful later on, when we compare the Bernstein center with the stable Bernstein center.

Let us recall the basic facts on the Kottwitz homomorphism \cite{Ko97}.  Let $L$ be the completion $\widehat{F}^{\rm un}$ of the maximal unramified extension $F^{\rm un}$ in some algebraic closure of $F$, and let $\bar{L} \supset \bar{F}$ denote an algebraic closure of $L$. Let $I = {\rm Gal}(\bar{L}/L) \cong {\rm Gal}(\bar{F}/F^{\rm un})$ denote the inertia group. Let $\Phi \in {\rm Aut}(L/F)$ be the inverse of the Frobenius automorphism $\sigma$. In \cite{Ko97}~is defined a functorial surjective homomorphism for any connected reductive $F$-group $H$
\begin{equation} \label{K-hom}
\kappa_H: H(L) \twoheadrightarrow X^*(Z(\widehat{H}))_I,
\end{equation}
where $\widehat{H} = \widehat{H}(\mathbb C)$ denotes the Langlands dual group of $H$.  
By \cite[$\S7$]{Ko97}, it remains surjective on taking $\Phi$-fixed points:
$$
\kappa_H: H(F) \twoheadrightarrow X^*(Z(\widehat{H}))_I^\Phi.
$$
We define
\begin{align*}
H(L)_1 &:= {\rm ker}(\kappa_H) \\
H(F)_1 &:= {\rm ker}(\kappa_H) \cap H(F).
\end{align*}
We also define $H(F)^1 \supseteq H(F)_1$ to be the kernel of the map $H(F) \rightarrow X^*(Z(\widehat{H}))_I^\Phi/tors$ derived from $\kappa_H$.  

If $H$ is anisotropic modulo center, then $H(F)^1$ is the unique maximal compact subgroup of $H(F)$ and $H(F)_1$ is the unique parahoric subgroup of $H(F)$ (see e.g. \cite{HRo}).  Sometimes the two subgroups coincide: for example if $H$ is any {\em unramified} $F$-torus, then $H(F)_1 = H(F)^1$.

We define 
$$X(H) := {\rm Hom}_{\rm grp}(H(F)/H(F)^1, \mathbb C^\times),$$ 
the group of {\em unramified characters} on $H(F)$.   This definition of $X(H)$ agrees with the usual one as in \cite{BD}. We define 
$$
X^{\rm w}(H) := {\rm Hom}_{\rm grp}(H(F)/H(F)_1, \mathbb C^\times)$$ 
and call it the group of {\em weakly unramified characters} on $H(F)$.

We follow the notation of \cite{BK} in discussing supercuspidal supports and inertial equivalence classes.  As indicated earlier in $\S3.1$, for convenience we will sometimes write $G$ when we mean the group $G(F)$ of $F$-points of an $F$-group $G$. 

A {\em cuspidal pair} $(M,\sigma)$ consists of an $F$-Levi subgroup $M \subseteq G$ and a supercuspidal representation $\sigma$ on $M$.  The $G$-conjugacy class of the cuspidal pair $(M,\sigma)$ will be denoted $(M,\sigma)_G$.  We define the inertial equivalence classes: we write $(M,\sigma) \sim (L, \tau)$ if there exists $g \in G$ such that $gMg^{-1} = L$ and $^g\sigma = \tau \otimes \chi$ for some $\chi \in X(L)$.  Let $[M,\sigma]_G$ denote the equivalence class of $(M,\sigma)_G$.

If $\pi \in \mathfrak R(G)_{\rm irred}$, then the supercuspidal support of $\pi$ is the unique element $(M, \sigma)_G$ such that $\pi$ is a subquotient of the induced representation $i^G_P(\sigma)$, where $P$ is any $F$-parabolic subgroup having $M$ as a Levi subgroup.  Let $\mathfrak X_G$ denote the set of all supercuspidal supports $(M,\sigma)_G$.  Denote by the symbol $\mathfrak s = [M,\sigma]_G$ a typical inertial class.  

For an inertial class $\mathfrak s = [M,\sigma]_G$, define the set $\mathfrak X_\mathfrak s = \{ (L,\tau)_G ~ | ~ (L,\tau) \sim (M,\sigma) \}$. We have
$$
\mathfrak X_G = \coprod_{\mathfrak s} \mathfrak X_{\mathfrak s}.
$$
We shall see below that $\mathfrak X_G$ has a natural structure of an algebraic variety, and the Bernstein components $\mathfrak X_\mathfrak s$ form the connected components of that variety.

First we need to recall the variety structure on $X(M)$. As is well-known, $X(M)$ has the structure of a complex torus.  To describe this, we first consider the weakly unramified character group $X^{\rm w}(M)$.  This is a diagonalizable group over $\mathbb C$.  In fact, by Kottwitz we have an isomorphism
$$
M(F)/M(F)_1 \cong X^*(Z(\widehat{M})^I)^{\Phi} = X^*((Z(\widehat{M})^I)_\Phi).
$$
This means that 
$$X^{\rm w}(M)(\mathbb C) = {\rm Hom}_{\rm grp}(M(F)/M(F)_1, \mathbb C^\times) = {\rm Hom}_{\rm alg}(\mathbb C[X^*((Z(\widehat{M})^I)_\Phi)], \mathbb C),$$
in other words,
\begin{equation} \label{weakly_unram_chars}
X^{\rm w}(M) = (Z(\widehat{M})^I)_\Phi.
\end{equation}
Another way to see (\ref{weakly_unram_chars}) is to use Langlands' duality for quasicharacters, which is an isomorphism
$$
{\rm Hom}_{\rm cont}(M(F), \mathbb C^\times) \cong H^1(W_F, Z(\widehat{M})).
$$
(Here $W_F$ is the Weil group of $F$; see $\S\ref{LLC_sec}$.) Under this isomorphism, $X^{\rm w}(M)$ is identified with the image of the inflation map $H^1(W_F/I, Z(\widehat{M})^I) \rightarrow H^1(W_F, Z(\widehat{M}))$, that is, with $H^1(\langle \Phi \rangle, Z(\widehat{M})^I) = (Z(\widehat{M})^I)_\Phi$. This last identification is given by the map sending a cocycle $\varphi^{cocyc} \in Z^1(\langle \Phi \rangle, Z(\widehat{M})^I)$ to $\varphi^{cocyc}(\Phi) \in (Z(\widehat{M})^I)_\Phi$. The two ways of identifying $X^{\rm w}(M)$ with $(Z(\widehat{M})^I)_\Phi$, that via the Kottwitz isomorphism and that via Langlands duality, agree.\footnote{We normalize the Kottwitz homomorphism as in \cite{Ko97}, so that $\kappa_{\mathbb G_m}: L^\times \rightarrow \mathbb Z$ is the valuation map sending a uniformizer $\varpi$ to 1.  Then the claimed agreement holds provided we normalize the Langlands duality for tori as in (\ref{Ldual_norm}).} For a more general result which implies this agreement, see \cite[Prop.~4.5.2]{Kal}.

Now we can apply the same argument to identify the torus $X(M)$.  We first get
$$
X(M) (\mathbb C)= {\rm Hom}_{\rm grp}(M(F)/M(F)^1, \mathbb C^\times).$$
 Since $M(F)/M(F)^1$ is the quotient of $M(F)/M(F)_1$ by its torsion, it follows that 
$$X(M) =  ((Z(\widehat{M})^I)_\Phi)^\circ =: (Z(\widehat{M})^I)^\circ_\Phi,$$
the neutral component of $X^{\rm w}(M)$.   


Now we turn to the variety structure on $\mathfrak X_{\mathfrak s}$.  We fix a cuspidal pair $(M,\sigma)$ representing $\mathfrak s_M = [M,\sigma]_M$ and $\mathfrak s = [M,\sigma]_G$.  
Let the corresponding Bernstein components be denoted $\mathfrak X_\mathfrak s$ and $\mathfrak X_{\mathfrak s_M}$. As sets, we have $\mathfrak X_\mathfrak s = \{ (M,\sigma \chi)_G \}$ and $\mathfrak X_{\mathfrak s_M} = \{ (M, \sigma \chi)_M \}$, where $\chi \in X(M)$. 
The torus $X(M)$ acts on $\mathfrak X_{\mathfrak s_M}$ by $\chi \mapsto (M, \sigma \chi)_M$.  The isotropy group is ${\rm stab}_\sigma := \{ \chi ~ | ~ \sigma \cong \sigma \chi \}$.  Let $Z(M)^\circ$ denote the neutral component of the center of $M$. Then ${\rm stab}_\sigma$ belongs to the kernel of the map $X(M) \rightarrow X(Z(M)^\circ)$, $\chi \mapsto \chi|_{Z(M)^\circ(F)}$, hence ${\rm stab}_\sigma$ is a {\em finite} subgroup of $X(M)$.  Thus $\mathfrak X_{\mathfrak s_M}$ is a torsor under the torus $X(M)/{\rm stab}_\sigma$, and thus has the structure of an affine variety over $\mathbb C$.

There is a surjective map
\begin{align*}
\mathfrak X_{\mathfrak s_M} &\rightarrow \mathfrak X_\mathfrak s \\
(M,\sigma \chi)_M & \mapsto (M, \sigma \chi)_G, \,\,\, \chi \in X(M)/{\rm stab}_\sigma.
\end{align*}
Let $N_G([M,\sigma]_M) := \{ n \in N_G(M) ~ | ~ \, ^n\sigma \cong \sigma \chi \,\, \mbox{for some $\chi \in X(M)$} \}$. Then the fibers of $\mathfrak X_{\mathfrak s_M} \rightarrow \mathfrak X_{\mathfrak s}$ are precisely the orbits the finite group $W^G_{[M,\sigma]_M}:= N_G([M,\sigma]_G)/M$ on $\mathfrak X_{\mathfrak s_M}$. Via
$$
\mathfrak X_\mathfrak s  = W^G_{[M,\sigma]_M} \backslash \mathfrak X_{\mathfrak s_M}
$$
the set $\mathfrak X_\mathfrak s$ acquires the structure of an irreducible affine variety over $\mathbb C$.  Up to isomorphism, this structure does not depend on the choice of the cuspidal pair $(M,\sigma)$.

\subsubsection{The center as regular functions on $\mathfrak X_G$}

An element $z \in \mathfrak Z(G)$ determines a regular function $\mathfrak X$: for a point $(M, \sigma)_G \in \mathfrak X_\mathfrak s$, $z$ acts on $i^G_P(\sigma)$ by a scalar $z(\sigma)$ and the function $(M,\sigma)_G \mapsto z(\sigma)$ is a regular function on $\mathfrak X_G$.  This is the content of \cite[Prop.~2.11]{BD}.  In fact we have by \cite[Thm.~2.13]{BD} an isomorphism
\begin{equation} \label{Bern_cntr_reg_fcns}
\mathfrak Z(G) ~ \widetilde{\rightarrow} ~ \mathbb C[\mathfrak X_G].
\end{equation}

Together with (\ref{Bern_cntr_as_distr}) this gives all the equivalent ways of realizing the Bernstein center of $G$.

\section{The local Langlands correspondence} \label{LLC_sec}

We need to recall the general form of the conjectural local Langlands correspondence $({\rm LLC})$ for a connected reductive group $G$ over a $p$-adic field $F$. Let $\bar{F}$ denote an algebraic closure of $F$. Let $W_F \subset {\rm Gal}(\bar{F}/F) =: \Gamma_F$ be the {\em Weil group of $F$}.  It fits into an exact sequence of topological groups 
$$
\xymatrix{
1 \ar[r] &  I_F \ar[r] & W_F \ar[r]^{\rm val} & \mathbb Z \ar[r] & 1,}
$$
where $I_F$ is the inertia subgroup of $\Gamma_F$ and where, if $\Phi \in W_F$ is a geometric Frobenius element (the inverse of an arithmetic Frobenius element), then ${\rm val}(\Phi) = -1$. Here $I_F$ has its profinite topology and $\mathbb Z$ has the discrete topology. Sometimes we write $I$ for $I_F$ in what follows.  

Recall the {\em Weil-Deligne} group is $W'_F := W_F \ltimes \mathbb C$, where 
$wzw^{-1} = |w|z$ for $w \in W_F$ and $z \in \mathbb C$, with $|w| := q_F^{{\rm val}(w)}$ for $q_F = \#(\mathcal O_F/(\varpi_F))$, the cardinality of the residue field of $F$.

A Langlands parameter is an {\em admissible} homomorphism $\varphi: W'_F \rightarrow \, ^LG$, where $^LG := \widehat{G} \rtimes W_F$.  This means:
\begin{itemize}
\item $\varphi$ is compatible with the projections $W'_F \rightarrow W_F$ and $\nu: \, ^LG \rightarrow W_F$;
\item $\varphi$ is continuous and respects Jordan decompositions of elements in $W'_F$ and $^LG$ (cf. \cite[$\S8$]{Bo79} for the definition of Jordan decomposition in the group $W_F \ltimes \mathbb C$ and what it means to respect Jordan decompositions here);
\item if $\varphi(W'_F)$ is contained in a Levi subgroup of a parabolic subgroup of $^LG$, then that parabolic subgroup is {\em relevant} in the sense of \cite[$\S3.3$]{Bo79}.  (This condition is automatic if $G/F$ is quasi-split.) 
\end{itemize}

Let $\Phi(G/F)$ denote the set of $\widehat{G}$-conjugacy classes of admissible homomorphisms $\varphi: W'_F \rightarrow \, ^LG$ and let $\Pi(G/F) = \mathfrak R(G(F))_{\rm irred}$ the set of irreducible smooth (or admissible) representations of $G(F)$ up to isomorphism.

\begin{conj} [{\rm LLC}] There is a finite-to-one surjective map $\Pi(G/F) \rightarrow \Phi(G/F)$, which satisfies the desiderata of \cite[$\S 10$]{Bo79}. 
\end{conj}
The fiber $\Pi_\varphi$ over $\varphi \in \Phi(G/F)$ is called the $L$-packet for $\varphi$. 

\medskip

 We mention a few desiderata of the LLC that will come up in what follows. First, LLC for ${\mathbb G}_m$ is nothing other than Langlands duality for ${\mathbb G}_m$, which we normalize as follows: for $T$ any split torus torus over $F$, with dual torus $\widehat{T}$, 
\begin{align*}
{\rm Hom}_{\rm conts}(T(F), \mathbb C^\times) &= {\rm Hom}_{\rm conts}(W_F, \widehat{T}) \\
\xi &\leftrightarrow \varphi_\xi
\end{align*}
satisfies, for every $\nu \in X_*(T) = X^*(\widehat{T})$ and $w \in W_F$,
\begin{equation} \label{Ldual_norm}
\nu(\varphi_\xi(w)) = \xi(\nu({\rm Art}_F^{-1}(w))).
\end{equation}
Here ${\rm Art}^{-1}_F: W^{\rm ab}_F \rightarrow F^\times$ is the reciprocity map of local class field theory which sends any geometric Frobenius element $\Phi \in W_F$ to a uniformizer in $F$.

Next, we think of Langlands parameters in two ways, either as continuous $L$-{\em homomorphisms}
$$
\varphi: W'_F \rightarrow \, ^LG 
$$
modulo $\widehat{G}$-conjugation, or as continuous {\em 1-cocycles}
$$
\varphi^{cocyc}: W'_F \rightarrow \widehat{G}
$$
modulo 1-coboundaries (where $W'_F$ acts on $\widehat{G}$ via the projection $W'_F \rightarrow W_F$).  The dictionary between these is
$$
\varphi(w) = (\varphi^{cocyc}(w), \bar{w}) \in \widehat{G} \rtimes W_F
$$
for $w \in W'_F$ and $\bar{w}$ the image of $w$ under $W'_F \rightarrow W_F$.

The desideratum we will use explicitly is the following (a special case of \cite[10.3(2)]{Bo79}: given any Levi pair $(M, \sigma)$ (where $\sigma \in \Pi(M/F)$) with representing 1-cocycle $\varphi^{cocyc}_\sigma: W'_F \rightarrow \widehat{M}$, and any unramified 1-cocycle $z^{cocyc}_\chi: W_F \rightarrow Z(\widehat{M})^I$ representing $\chi \in X(M)$ via the Langlands correspondence for quasi-characters on $M(F)$, we have
$$
\varphi^{cocyc}_{\sigma \chi} = \varphi^{cocyc}_\sigma \cdot z^{cocyc}_\chi
$$
modulo 1-coboundaries with values in $\widehat{M}$. We may view $z^{cocyc}_\chi$ as a 1-cocycle on $W'_F$ which is trivial on $\mathbb C$; since it takes values in the center of $\widehat{M}$, the right hand side is a 1-cocycle whose cohomology class is independent of the choices of 1-cocycles $\varphi^{cocyc}_\sigma$ and $z^{cocyc}_\chi$ in their respective cohomology classes.  Hence the condition just stated makes sense.  Concretely, if $\chi \in X(M)$ lifts to an element $z \in Z(\widehat{M})^I$, then up to $\widehat{M}$-conjugacy we have
\begin{equation} \label{action}
\varphi_{\sigma \chi}(\Phi) = (z,1) \varphi_\sigma(\Phi) \in \widehat{M} \rtimes W_F.
\end{equation}

\begin{Remark} \label{diff_param_rem}
There is a well-known dictionary between equivalence classes of admissible homomorphisms $\varphi: W_F \ltimes \mathbb C \rightarrow \, ^LG$ and equivalence classes of admissible homomorphisms $W_F \times {\rm SL}_2(\mathbb C) \rightarrow \, ^LG$. For a complete explanation, see  \cite[Prop.~2.2]{GR}.  
Because of this equivalence, it is common in the literature for the Weil-Deligne group $W'_F$ to sometimes be defined as $W_F \ltimes \mathbb C$, and sometimes as $W_F \times {\rm SL}_2(\mathbb C)$.
\end{Remark}

\section{The stable Bernstein center} \label{st_bernstein_cntr_sec}

\subsection{Infinitesimal characters}

Following Vogan \cite{Vo}, we term a $\widehat{G}$-conjugacy class of an admissible\footnote{``Admissible'' is defined as for the parameters $\varphi: W'_F \rightarrow \, ^LG$ (e.~g.~$\lambda(W_F)$ consists of semisimple elements of $^LG$) {\bf except} that we omit the ``relevance condition''. This is because the restriction $\varphi|_{W_F}$ of a Langlands parameter could conceivably factor through a non-relevant Levi subgroup of $^LG$ (even though $\varphi$ does not) and we want to include such restrictions in what we call infinitesimal characters.} homomorphism
$$
\lambda: W_F \rightarrow \, ^LG
$$
an {\em infinitesimal character}.  Denote the $\widehat{G}$-conjugacy class of $\lambda$ by $(\lambda)_{\widehat{G}}$. In this section we give a geometric structure to the set of all infinitesimal characters for a group $G$.  It should be noted that the variety structure we define here differs from that put forth by Vogan in \cite[$\S7$]{Vo}.

If $\varphi : W'_F \rightarrow \, ^LG$ is an admissible homomorphism, then its restriction $\varphi|_{W_F}$ represents an infinitesimal character.  Here it is essential to consider restriction along the proper embedding $W_F \hookrightarrow W'_F$: if $W'_F$ is thought of as $W_F \ltimes \mathbb C$, then this inclusion is $w \mapsto (w,0)$; if $W'_F$ is thought of as $W_F \times {\rm SL}_2(\mathbb C)$, then the inclusion is $w \mapsto (w, {\rm diag}(|w|^{1/2}, |w|^{-1/2}))$. If $\varphi_\pi \in \Phi(G/F)$ is attached by ${\rm LLC}$ to $\pi \in \Pi(G/F)$, then following Vogan \cite{Vo} we shall call the $\widehat{G}$-conjugacy class $(\varphi_\pi|_{W_F})_{\widehat{G}}$ the {\em infinitesimal character of $\pi$}. 

If $G$ is quasi-split over $F$, then conjecturally {\em every} infinitesimal character $\lambda$ is represented by a restriction $\varphi_\pi|_{W_F}: W_F \rightarrow \, ^LG$ for {\em some} $\pi \in \Pi(G/F)$.  

Assume ${\rm LLC}$ holds for $G/F$. Let $\lambda$ be an infinitesimal character for $G$. Define the {\bf infinitesimal class} to be the following finite union of $L$-packets
$$
\displaystyle
\Pi_\lambda := \coprod_{\varphi \leadsto \lambda} \Pi_\varphi.
$$
Here $\varphi$ ranges over $\widehat{G}$-conjugacy classes of admissible homomorphisms $W'_F \rightarrow \, ^LG$ such that $(\varphi|_{W_F})_{\widehat{G}} = (\lambda)_{\widehat{G}}$, and $\Pi_\varphi$ is the corresponding $L$-packet of smooth irreducible representations of $G(F)$.

\subsection{${\bf LLC+}$} \label{LLC+_subsec}

In order to relate the Bernstein variety $\mathfrak X$ with the variety $\mathfrak Y$ of infinitesimal characters, we will assume the Local Langlands Correspondence (LLC) for $G$ and all of its $F$-Levi subgroups.  We assume all the desiderata listed by Borel in \cite{Bo79}.

There are two additional desiderata of LLC we need. 

\begin{defn} We will declare that $G$ satisfies ${\rm LLC+}$ if the LLC holds for $G$ and its $F$-Levi subgroups, and these correspondences are compatible with normalized parabolic induction in the sense of the Conjecture \ref{comp_w_induction} below, and invariant under certain isomorphisms in the sense of Conjecture \ref{weak_invariant} below.
\end{defn}

Let $M \subset G$ denote an $F$-Levi subgroup.  Then the inclusion $M \hookrightarrow G$ induces an embedding $^LM \hookrightarrow \, ^LG$ which is well-defined up to $\widehat{G}$-conjugacy (cf. \cite[$\S3$]{Bo79}). 

\begin{conj} \label{comp_w_induction} {\rm (Compatibility of {\rm LLC} with parabolic induction)} Let $\sigma \in \Pi(M/F)$ and $\pi \in \Pi(G/F)$ and assume $\pi$ is an irreducible subquotient of $i^G_P(\sigma)$, where $P = MN$ is any $F$-parabolic subgroup of $G$ with $F$-Levi factor $M$.  Then the infinitesimal characters 
$$
\varphi_{\pi}|_{W_F} : W_F \rightarrow \, ^LG
$$
and
$$
\varphi_{\sigma}|_{W_F}: W_F \rightarrow \, ^LM \hookrightarrow \, ^LG
$$
are $\widehat{G}$-conjugate.
\end{conj}

\begin{Remark} \label{comp_w_ind_rem}

\noindent (1) The conjecture implies that the restriction $\varphi_\pi|_{W_F}$ depends only on the supercuspidal support of $\pi$. This latter statement is a formal consequence of Vogan's Conjecture 7.18 in \cite{Vo}, but the Conjecture \ref{comp_w_induction} is slightly more precise. In Proposition \ref{map_structure} we will give a construction of the map $f$ in Vogan's Conjecture 7.18, by sending a supercuspidal support $(M, \sigma)_G$ (a ``classical infinitesimal character'' in \cite{Vo}) to the infinitesimal character $(\varphi_\sigma|_{W_F})_{\widehat{G}}$. With this formulation, the condition on $f$ imposed in Vogan's Conjecture 7.18 is exactly the compatibility in the conjecture above.

\noindent (2) The conjecture holds for ${\rm GL}_n$, and is implicit in the way the local Langlands correspondence for ${\rm GL}_n$ is extended from supercuspidals to all representations (see Remark  13.1.1 of \cite{HRa}).  It was a point of departure in Scholze's new characterization of LLC for ${\rm GL}_n$ \cite{Sch3}, and that paper also provides another proof of the conjecture in that case.

\noindent (3) I was informed by Brooks Roberts (private communication), that the conjecture holds for ${\rm GSp}(4)$.

\noindent (4)  Given a parameter $\varphi: W'_F \rightarrow \, ^LG$, there exists a certain $P = MN$ and a certain tempered parameter $\varphi_M: W'_F \rightarrow \, ^LM$ and a certain real-valued unramified character $\chi_M$ on $M(F)$ whose parameter is in the interior of the Weyl chamber determined by $P$, such that the $L$-packet $\Pi_\varphi$ consists of  Langlands quotients $J (\pi_M \otimes \chi_M)$, for $\pi_M$ ranging over the packet $\Pi_{\varphi_M}$.  The parameter $\varphi$ is the twist of $\varphi_M$ by the parameter associated to the character $\chi_M$.  This reduces the conjecture to the case of tempered representations. One can further reduce to the case of discrete series representations.  
\end{Remark}

The following is a very natural kind of functoriality which should be satisfied for all groups.

\begin{conj} {\rm (Invariance of LLC under isomorphisms)} \label{invariant}
Suppose $\phi: (G, \pi) ~ \widetilde{\rightarrow} ~ (G', \pi')$ is an isomorphism of connected reductive $F$-groups together with irreducible smooth representations on them.  Then the induced isomorphism $^L\phi: \, ^LG' ~ \widetilde{\rightarrow} ~ \, ^LG$ (well-defined up to an inner automorphism of $\widehat{G}$), takes the $\widehat{G}'$-conjugacy class of $\varphi_{\pi'} : W'_F \rightarrow \, ^LG'$ to the $\widehat{G}$-conjugacy class of $\varphi_{\pi}: W'_F \rightarrow \, ^LG$.
\end{conj}

\begin{prop}  \label{GL_n_invariance}
Conjecture \ref{invariant} holds when $G = {\rm GL}_n$.
\end{prop}

\begin{proof} ({\rm Guy Henniart}).   It is enough to consider the case where $G' = {\rm GL}_n$ and $\phi$ is an $F$-automorphism of ${\rm GL}_n$.

The functorial properties in the Langlands correspondence for ${\rm GL}_n$ are:
\begin{enumerate}
\item[(i)] Compatibility with class field theory, that is, with the case where $n=1$.
\item[(ii)] The determinant of the Weil-Deligne group representation corresponds to the central character: this is Langlands functoriality for the homomorphism ${\rm det}: {\rm GL}_n(\mathbb C) \rightarrow {\rm GL}_1(\mathbb C)$.
\item[(iii)] Compatibility with twists by characters, i.e., Langlands functoriality for the obvious homomorphism of dual groups ${\rm GL}_1(\mathbb C) \times {\rm GL}_n(\mathbb C) \rightarrow {\rm GL}_n(\mathbb C)$.
\item[(iv)] Compatibility with taking contragredients: this is Langlands functoriality with respect to the automorphism $g \mapsto \, ^tg^{ -1}$ (transpose inverse), since it is known that for ${\rm GL}_n(F)$ this sends an irreducible representation to a representation isomorphic to its contragredient.
\end{enumerate}

These properties are enough to imply the desired functoriality for $F$-automorphisms of ${\rm GL}_n$.  When $n = 1$, the functoriality is obvious for any $F$-endomorphism of ${\rm GL}_1$.  When $n$ is at least 2, an $F$-automorphism of ${\rm GL}_n$ induces an automorphism of ${\rm SL}_n$ hence an automorphism of the Dynkin diagram which must be the identity or, (when $n \geq 3$) the opposition automorphism. Hence up to conjugation by ${\rm GL}_n(F)$, the $F$-automorphism is the identity on ${\rm SL}_n$, or possibly (when $n \geq 3$) transpose inverse. Consequently the $F$-automorphism can be reduced (by composing with an inner automorphism or possibly with transpose inverse) to one which is the identity on ${\rm SL}_n$, hence is of the form  $g \mapsto g \cdot c({\rm det}(g))$ where $c \in X_*(Z({\rm GL}_n))$. But this implies that it is the identity unless $n=2$, in which case it could also be $g \mapsto g \cdot {\rm det}(g)^{-1}$. In that exceptional case, the map induced on the dual group ${\rm GL}_2(\mathbb C)$ is also $g \mapsto g \cdot {\rm det}(g)^{-1}$, and the desired result holds by invoking (ii) and (iii) above.
\end{proof}

\begin{cor} \label{invariant_Levi}
Let $M = {\rm GL}_{n_1} \times \cdots \times {\rm GL}_{n_r} \subset {\rm GL}_n$ be a standard Levi subgroup. Let $g \in {\rm GL}_n(F)$. Then Conjecture \ref{invariant} holds for the isomorphism $c_g: M ~ \widetilde{\rightarrow} ~ \, ^gM$ given by conjugation by $g$.
\end{cor}

\begin{proof}
It is enough to consider the case where $g$ belongs to the normalizer of $M$ in ${\rm GL}_n$. Let $T \subset M$ be the standard diagonal torus in ${\rm GL}_n$. Then $g \in N_G(T) M$. Thus composing $g$ with a permutation matrix which normalizes $M$ we may assume that $c_g$ preserves each diagonal factor ${\rm GL}_{n_i}$.  The desired functoriality follows by applying Proposition \ref{GL_n_invariance} to each ${\rm GL}_{n_i}$.
\end{proof}

For the purposes of comparing the Bernstein center and the stable Bernstein center as in Proposition \ref{map_structure}, we need only this weaker variant of Conjecture \ref{invariant}.

\begin{conj} {\rm (Weak invariance of LLC)} \label{weak_invariant} Let $M \subseteq G$ be any $F$-Levi subgroup and let $g \in G(F)$. Then Conjecture \ref{invariant} holds for the isomorphism 
$c_g : M ~ \widetilde{\rightarrow} ~ \, ^gM$.
\end{conj}

\subsection{Variety structure on the set of infinitesimal characters} \label{variety_str_subsec}

It is helpful to rigidify things on the dual side by choosing the data $\widehat{G} \supset \widehat{B} \supset  \widehat{T}$ of a Borel subgroup and maximal torus which are stable under the action of $\Gamma_F$ on $\widehat{G}$ and which form part of the data of a $\Gamma_F$-invariant splitting for $\widehat{G}$ (cf. \cite[$\S1$]{Ko84a}). The variety structure we will define will be independent of this choice, up to isomorphism, since different choices such that $\widehat{B} \supset \widehat{T}$ are conjugate under $\widehat{G}^{\Gamma_F}$ (\cite[Cor.~1.7]{Ko84a}). Let $^LB := \widehat{B} \rtimes W_F$ and $^LT = \widehat{T} \rtimes W_F$.

Following \cite[$\S3.3$]{Bo79}, we say a parabolic subgroup $\mathcal P \subseteq \, ^LG$ is {\em standard} if $\mathcal P \supseteq \, ^LB$.  Then its neutral component $\mathcal P^\circ := \mathcal P \cap \widehat{G}$ is a $W_F$-stable standard parabolic subgroup of $\widehat{G}$ (containing $\widehat{B}$), and $\mathcal P = \mathcal P^\circ \rtimes W_F$.  Every parabolic subgroup in $^LG$ is $\widehat{G}$-conjugate to a unique standard parabolic subgroup.

Assume $\mathcal P$ is standard and let $\mathcal M^\circ \subset \mathcal P^\circ$ be the unique Levi factor with $\mathcal M^\circ \supseteq \widehat{T}$; it is $W_F$-stable. Then $\mathcal M := N_{\mathcal P}(\mathcal M^\circ)$ is a Levi subgroup of $\mathcal P$ in the sense of \cite[$\S3.3$]{Bo79}, and $\mathcal M = \mathcal M^\circ \rtimes W_F$. The Levi subgroups $\mathcal M \subset \, ^LG$ which arise this way are called {\em standard}. Every Levi subgroup in $^LG$ is $\widehat{G}$-conjugate to at least one standard Levi subgroup; two different standard Levi subgroups may be conjugate under $\widehat{G}$.  Denote by $\{\mathcal M \}$ the set of standard Levi subgroups in $^LG$ which are $\widehat{G}$-conjugate to a fixed standard Levi subgroup $\mathcal M$.

Now suppose $\lambda: W_F \rightarrow \, ^LG$ is an admissible homomorphism.  Then there exists a minimal Levi subgroup of $^LG$ containing $\lambda(W_F)$.  Any two such are conjugate by an element of $C_\lambda^\circ$, where $C_\lambda$ is the subgroup of $\widehat{G}$ commuting with $\lambda(W_F)$, by (the proof of) \cite[Prop.~3.6]{Bo79}.  

Suppose $\lambda_1, \lambda_2 : W_F \rightarrow \, ^LG$ are $\widehat{G}$-conjugate.  Then there exists a $\widehat{G}$-conjugate $\lambda_1^+$ (resp. $\lambda_2^+$) of $\lambda_1$ (resp. $\lambda_2$) and a standard Levi subgroup $\mathcal M_1$ (resp. $\mathcal M_2$) containing $\lambda^+_1(W_F)$ (resp. $\lambda^+_2(W_F)$) minimally.  Write $g\lambda_1^+ g^{-1} = \lambda^+_2 $ for some $g \in \widehat{G}$.  Then the Levi subgroups $ g\mathcal M_1 g^{-1}$ and $\mathcal M_2$ contain $\lambda_2^+(W_F)$ minimally, hence by \cite[Prop.~3.6]{Bo79} are conjugate by an element $s \in C_{\lambda^+_2}^\circ$.  Then $sg(\mathcal M_1) (sg)^{-1} = \mathcal M_2$, and thus $\{\mathcal M_1 \} = \{\mathcal M_2 \}$.

Hence any $\widehat{G}$-conjugacy class $(\lambda)_{\widehat{G}}$ gives rise to a {\em unique} class of standard Levi subgroups $\{\mathcal M_\lambda\}$, with the property that the image of some element $\lambda^+ \in (\lambda)_{\widehat{G}}$ is contained minimally by $\mathcal M_\lambda$ for some $\mathcal M_\lambda$ in this class.

 A similar argument shows the following lemma. 

\begin{lemma} \label{norm_conjugation}
Let $\lambda^+_1$ and $\lambda^+_2$ be admissible homomorphisms with $(\lambda^+_1)_{\widehat{G}} = (\lambda^+_2)_{\widehat{G}}$, and suppose $\lambda^+_1(W_F)$ and $\lambda^+_2(W_F)$ are contained minimally by a standard Levi subgroup $\mathcal M$.  Then there exists $n \in N_{\widehat{G}}(\mathcal M)$ such that $^n\lambda^+_1 = \lambda^+_2$.
\end{lemma}

The following lemma is left to the reader.

\begin{lemma} \label{normalizer_lem} If $\mathcal M \subseteq \, ^LG$ is a standard Levi subgroup, then 
$$
N_{\widehat{G}}(\mathcal M) = \{ n \in N_{\widehat{G}}(\mathcal M^\circ) ~ | ~ \, \mbox{$n\mathcal M^\circ$ is $W_F$-stable} \}.
$$
Consequently, conjugation by $n \in N_{\widehat{G}}(\mathcal M)$ preserves the set $(Z(\mathcal M^\circ)^I)_\Phi^\circ$. More generally, if $\mathcal M_1$ and $\mathcal M_2$ are standard Levi subgroups of $^LG$ and if we define the transporter subset by
$$
{\rm Trans}_{\widehat{G}}(\mathcal M_1, \mathcal M_2) := \{ g \in \widehat{G} ~ | ~ g\mathcal M_1 g^{-1} = \mathcal M_2 \},
$$
then 
$$
{\rm Trans}_{\widehat{G}}(\mathcal M_1, \mathcal M_2) = \{ g \in {\rm Trans}_{\widehat{G}}(\mathcal M^\circ_1, \mathcal M^\circ_2) ~ | ~ \, \mbox{$g\mathcal M^\circ_1 = \mathcal M^\circ_2 g$ is $W_F$-stable} \}.
$$
Consequently, conjugation by $g \in {\rm Trans}_{\widehat{G}}(\mathcal M_1, \mathcal M_2)$ sends $(Z(\mathcal M^\circ_1)^I)_\Phi ^\circ$ into $(Z(\mathcal M^\circ_2)^I)_\Phi ^\circ$.
\end{lemma}

We can now define the notion of {\em inertial equivalence} $(\lambda_1)_{\widehat{G}} \sim (\lambda_2)_{\widehat{G}}$ of infinitesimal characters.

\begin{defn} We say $(\lambda_1)_{\widehat{G}}$ and $(\lambda_2)_{\widehat{G}}$ are {\em inertially equivalent} if 
\begin{itemize}
\item $\{\mathcal M_{\lambda_1} \} = \{\mathcal M_{\lambda_2}\}$; 
\item  there exists $\mathcal M \in \{\mathcal M_{\lambda_1} \}$, and $\lambda_1^+ \in (\lambda_1)_{\widehat{G}}$ and $\lambda_2^+ \in (\lambda_2)_{\widehat{G}}$ whose images are minimally contained by $\mathcal M$, and an element $z \in (Z(\mathcal M^\circ)^I)_\Phi^\circ$, such that
$$
(z\lambda_1^+)_{\mathcal M^\circ} = (\lambda_2^+)_{\mathcal M^\circ}.
$$
\end{itemize}
We write $[\lambda]_{\widehat{G}}$ for the inertial equivalence class of $(\lambda)_{\widehat{G}}$.
\end{defn}
Note that $\mathcal M$ automatically contains $(z\lambda_1^+)(W_F)$ minimally if it contains $\lambda^+_1(W_F)$ minimally. 

\begin{lemma} \label{equiv_rel}
The relation $\sim$ is an equivalence relation on the set of infinitesimal characters.
\end{lemma}
\begin{proof} 
Use Lemmas \ref{norm_conjugation} and \ref{normalizer_lem}. 
\end{proof}

\begin{Remark} 
To define $(\lambda_1)_{\widehat{G}} \sim (\lambda_2)_{\widehat{G}}$ we used the choice of $\widehat{G} \supset \widehat{B}\supset \widehat{T}$ (which was assumed to form part of a $\Gamma_F$-invariant splitting for $\widehat{G}$) in order to define the notion of standard Levi subgroup of $^LG$.  However, the equivalence relation $\sim$ is independent of this choice, since as remarked above, any two $\Gamma_F$-invariant splittings for $\widehat{G}$ are conjugate under $\widehat{G}^{\Gamma_F}$, by \cite[Cor.~1.7]{Ko84a}.
\end{Remark} 

\begin{Remark}
The property we need of standard Levi subgroups $\mathcal M \subseteq \, ^LG$ is that they are {\em decomposable}, that is, $\mathcal M^\circ := \mathcal M \cap \widehat{G}$ is $W_F$-stable, and $\mathcal M = \mathcal M^\circ \rtimes W_F$. Any standard Levi subgroup is decomposable.  In our discussion, we could have avoided choosing a notion of standard Levi, by associating to each $(\lambda)_{\widehat{G}}$ a unique class of decomposable Levi subgroups $\{ \mathcal M \}$, all of which are $\widehat{G}$-conjugate, such that $\lambda$ factors minimally through some $\mathcal M \in \{ \mathcal M \}$.
\end{Remark}

Now fix a standard Levi subgroup $\mathcal M \subseteq \, ^LG$. We write $\mathfrak t_{\mathcal M^\circ}$ for an inertial equivalence class of admissible homomorphisms $W_F \rightarrow \mathcal M$.  We write $\mathfrak Y_{\mathfrak t_{\mathcal M^\circ}}$ for the set of $\mathcal M^\circ$-conjugacy classes contained in this inertial class. We want to give this set the structure of an affine algebraic variety over $\mathbb C$.  Define the torus
\begin{equation} \label{Y_torus}
Y(\mathcal M^\circ) := (Z(\mathcal M^\circ)^I)_\Phi ^\circ.
\end{equation}
Then $Y(\mathcal M^\circ)$ acts transitively on $\mathfrak Y_{\mathfrak t_{\mathcal M^\circ}}$.  Fix a representative
$$
\lambda: W_F \rightarrow \mathcal M
$$
for this inertial class, so that $\mathfrak t_{\mathcal M^\circ} = [\lambda]_{\mathcal M^\circ}$. 

\begin{lemma} \label{ker_lambda_finite}
The $Y(\mathcal M^\circ)$-stabilizer 
$$
{\rm stab}_\lambda := \{ z \in Y(\mathcal M^\circ) ~ | ~ (z\lambda)_{\mathcal M^\circ} = (\lambda)_{\mathcal M^\circ} \}
$$
is finite.
\end{lemma}

\begin{proof}
There exists an integer $r \geq 1$ such that $\Phi^r$ acts trivially on $\mathcal M^\circ$.  The group ${\rm stab}_\lambda$ is contained in the preimage of the finite group $Z(\mathcal M^\circ)^{\Gamma_F} \cap (\mathcal M^\circ)_{\rm der}$ under the norm homomorphism
\begin{align*}
N_r: (Z(\mathcal M^\circ)^I)_\Phi &\rightarrow Z(\mathcal M^\circ)^{\Gamma_F}, \\
z &\mapsto z\Phi(z) \cdots \Phi^{r-1}(z) 
\end{align*}
and the kernel of this homomorphism is finite.
\end{proof}
Then $\mathfrak Y_{\mathfrak t_{\mathcal M^\circ}}$ is a torsor under the quotient torus
$$
 \mathfrak Y_{\mathfrak t_{\mathcal M^\circ}} \cong Y(\mathcal M^\circ)/{\rm stab}_\lambda.
$$
In this way the left hand side acquires the structure of an affine algebraic variety.  Up to isomorphism, this structure is independent of the choice of $\lambda$ representing $\mathfrak t_{\mathcal M^\circ}$.

Now let $\mathfrak t$ denote an inertial class of infinitesimal characters for $G$, and let $\mathfrak Y_{\mathfrak t}$ denote the set of infinitesimal characters in $\mathfrak t$.  Recall $\mathfrak t$ gives rise to a unique class of standard Levi subgroups $\{ \mathcal M \}$, having the property that some representative $\lambda$ for $\mathfrak t$ factors minimally through some $\mathcal M \in \{ \mathcal M \}$. Fix such a representative $\lambda: W_F \rightarrow \mathcal M \hookrightarrow \, ^LG $ for $\mathfrak t$, so that $\mathfrak t = [\lambda]_{\widehat{G}}$ and $\mathfrak t_{\mathcal M^\circ} = [\lambda]_{\mathcal M^\circ}$. By our previous work, there is a {\em surjective} map
\begin{align*}
\mathfrak Y_{\mathfrak t_{\mathcal M^\circ}} & \rightarrow \mathfrak Y_{\mathfrak t} \\
(z\lambda)_{\mathcal M^\circ} & \mapsto (z\lambda)_{\widehat{G}}.
\end{align*}
where $z \in Y(\mathcal M^\circ)/{\rm stab}_\lambda$.   Let 
$$N_{\widehat{G}}(\mathcal M,[\lambda]_{\mathcal M^\circ}) = \{ n \in N_{\widehat{G}}(\mathcal M) ~ | ~  (\, ^n\lambda)_{\mathcal M^\circ} = 
(z\lambda)_{\mathcal M^\circ}, \, \mbox{for some $z \in Y(\mathcal M^\circ)$} \}.$$
 From the above discussion we see the following.
\begin{lemma} 
The fibers of $\mathfrak Y_{\mathfrak t_{\mathcal M^\circ}} \rightarrow \mathfrak Y_{\mathfrak t}$ are precisely the orbits of the finite group $W^{\widehat{G}}_{[\lambda]_{\mathcal M^\circ}} := N_{\widehat{G}}(\mathcal M,[\lambda]_{\mathcal M^\circ}) / \mathcal M^\circ$ on $\mathfrak Y_{\mathfrak t_{\mathcal M^\circ}}$.
\end{lemma}
Hence $\mathfrak Y_{\mathfrak t} = W^{\widehat{G}}_{\mathfrak t_{\mathcal M^\circ}} \backslash \mathfrak Y_{\mathfrak t_{\mathcal M^\circ}}$ acquires the structure of an affine variety over $\mathbb C$.  Thus $\mathfrak Y = \coprod_{\mathfrak t} \mathfrak Y_{\mathfrak t}$ is an affine variety over $\mathbb C$ and each $\mathfrak Y_{\mathfrak t}$ is a connected component.

Let $\mathfrak Z^{\rm st}(G)$ denote the ring of regular functions on the affine variety $\mathfrak Y$.  We call this ring the {\bf stable Bernstein center} of $G/F$.

\subsection{Base change homomorphism of the stable Bernstein center}

Let $E/F$ be a finite extension in $\overline{F}/F$ with ramification index $e$ and residue field extension $k_E/k_F$ of degree $f$.  Then $W_E \subset W_F$ and $I_E \subseteq I_F$.  Further, we can take $\Phi_E := \Phi_F^f$ as a geometric Frobenius element in $W_E$.  Let $\mathfrak Y^{G/F}$ resp. $\mathfrak Y^{G/E}$ denote the variety of infinitesimal characters associated to $G$ resp. $G_E$.

\begin{prop} \label{change_of_field}
The map $(\lambda)_{\widehat{G}} \mapsto (\lambda|_{W_E})_{\widehat{G}}$ determines a morphism of algebraic varieties $\mathfrak Y^{G/F} \rightarrow \mathfrak Y^{G/E}$.  
\end{prop}

\begin{defn} \label{b_E/F_defn}
We call the corresponding map $b_{E/F}: \mathfrak Z^{\rm st}(G_E) \rightarrow \mathfrak Z^{\rm st}(G)$ the {\em base change homomorphism} for the stable Bernstein center.
\end{defn}

\begin{proof}
Suppose $\lambda: W_F \rightarrow  \widehat{G} \rtimes W_F$ factors minimally through the standard Levi subgroup $\mathcal M \subset \widehat{G} \rtimes W_F$ and that its restriction $\lambda|_{W_E} : W_E \rightarrow \widehat{G} \rtimes W_E$ factors minimally through the standard Levi subgroup $\mathcal M_E \subset \widehat{G} \rtimes W_E$.  We may assume $\mathcal M_E^\circ \subseteq \mathcal M^\circ$ and thus $Z(\mathcal M^\circ) \subset Z(\mathcal M_E^\circ)$.  

There is a homomorphism of tori
\begin{align} \label{Y_bc_hom}
Y(\mathcal M^\circ) = (Z(\mathcal M^\circ)^{I_F})_{\Phi_F}^\circ &\longrightarrow (Z(\mathcal M_E^\circ)^{I_E})_{\Phi_F^f}^\circ = Y(\mathcal M_E^\circ) \\
z &\longmapsto z_f := N_f(z) := z\cdot \Phi_F(z) \cdots \Phi^{f-1}_F(z).\notag
\end{align}

Recall that $z \in (Z(\mathcal M^\circ)^{I_F})_{\Phi_F}$ is identified with the image of the element $z(\Phi_F) \in Z(\mathcal M^\circ)^{I_F}$, where $z$ is viewed as a cohomology class $z \in H^1(\langle \Phi_F \rangle, Z(\mathcal M^\circ)^{I_F})$.  Using the same fact for $E$ in place of $F$, it follows that $(z\lambda)|_{W_E} = z_f \lambda|_{W_E}$, where $z_f$ is defined as above.  Thus the map $(z\lambda)_{\widehat{G}} \mapsto ((z\lambda)|_{W_E})_{\widehat{G}}$ lifts to the map
$$
(z \lambda)_{\mathcal M^\circ} \mapsto (z_f \lambda|_{W_E})_{\mathcal M^\circ} \mapsto (z_f \lambda|_{W_E})_{\widehat{G}},
$$
and being induced by (\ref{Y_bc_hom}), the latter is an algebraic morphism.
\end{proof}

\subsection{Relation between the Bernstein center and the stable Bernstein center}

The varieties $\mathfrak X$ and $\mathfrak Y$ are defined unconditionally.  In order to relate them, we need to assume ${\rm LLC+}$ holds.  

\begin{prop} \label{map_structure} Assume ${\rm LLC+}$ holds for the group $G$.  Then the map $(M,\sigma)_G \mapsto (\varphi_\sigma|_{W_F})_{\widehat{G}}$ defines a quasi-finite morphism of affine algebraic varieties 
$$
f: \mathfrak X \rightarrow \mathfrak Y.  
$$
It is surjective if $G/F$ is quasi-split.
\end{prop}

The reader should compare this with Conjecture 7.18 in \cite{Vo}.  Our variety structure on the set $\mathfrak Y$ is different from that put forth by Vogan, and our $f$ is given by a simple and explicit rule. In view of ${\rm LLC+}$ our $f$ automatically satisfies the condition which Vogan imposed on the map in his Conjecture 7.18: if $\pi$ has supercuspidal support $(M,\sigma)_G$, then the infinitesimal character of $\pi$ is $f((M,\sigma)_G)$.

\begin{proof}
It is easy to see that the map $(M,\sigma)_G \mapsto (\varphi_\sigma|_{W_F})_{\widehat{G}}$ is well-defined. We need to show that an isomorphism $c_g: (M,\sigma) ~ \widetilde{\rightarrow} ~ (\,^gM, \, ^g\sigma)$ given by conjugation by $g \in G(F)$ gives rise to parameters $\varphi_{\sigma}: W'_F \rightarrow \, ^LM \hookrightarrow \, ^LG$ and $\varphi_{\,^g\sigma}: W'_F \rightarrow \,^L(\,^gM) \hookrightarrow \, ^LG$ which differ by an inner automorphism of $\widehat{G}$.  In view of Conjecture \ref{weak_invariant} applied to $M$, the isomorphism $^L(\,^gM)  ~\widetilde{\rightarrow} ~ \, ^LM$ takes $\varphi_{\, ^g\sigma}$ to an $\widehat{M}$-conjugate of $\varphi_\sigma$.  On the other hand the embeddings $^LM \hookrightarrow \, ^LG$ and $^L(\, ^gM) \hookrightarrow \, ^LG$ are defined using based root systems in such a way that it is obvious that they are $\widehat{G}$-conjugate. 


To examine the local structure of this map, we first fix a $\lambda$ and a standard $\mathcal M_\lambda$ through which $\lambda$ factors minimally.  Let $\mathfrak t = [\lambda]_{\widehat{G}}$.  Then over $\mathfrak Y_\mathfrak t$ the map $f$ takes the form
\begin{equation} \label{local_map}
\coprod_{\mathfrak s_M \leadsto \mathfrak t} \mathfrak X_{\mathfrak s_M} \rightarrow \mathfrak Y_\mathfrak t.
\end{equation}
Here $\mathfrak s_M$ ranges over the inertial classes $[M,\sigma]_G$ such that $(\varphi_\sigma|_{W_F})_{\widehat{G}}$ is inertially equivalent to $(\lambda)_{\widehat{G}}$.  We now fix a representative $(M,\sigma)$ for $\mathfrak s_M$.  Given such a $\varphi_\sigma$, its restriction $\varphi_\sigma|_{W_F}$ factors through a $\widehat{G}$-conjugate of $^LM$.  But $(\varphi_\sigma|_{W_F})_{\widehat{G}} \sim (\lambda)_{\widehat{G}}$ implies that (up to conjugation by $\widehat{G}$) $\varphi_\sigma|_{W_F}$ factors minimally through $\mathcal M_\lambda$. Thus we may assume that $\mathcal M_\lambda \subseteq \, ^LM$.  The corresponding inclusion $Z(\widehat{M}) \hookrightarrow Z(\mathcal M_\lambda^\circ)$ induces a morphism of algebraic tori
$$
Y(\widehat{M}) = (Z(\widehat{M})^I)_\Phi^\circ \rightarrow (Z(\mathcal M_\lambda^\circ)^I)_\Phi^\circ = Y(\mathcal M_\lambda^\circ).
$$
Further, recall $X(M) \cong Y(\widehat{M})$ by the Kottwitz isomorphism (or the Langlands duality for quasi-characters), by the rule $\chi \mapsto z_\chi^{cocyc}(\Phi)$.  

Taking (\ref{action}) into account, we see that (\ref{local_map}) on $\mathfrak X_{\mathfrak s_M}$, given by
$$
(M,\sigma \chi)_G \mapsto (\varphi_{\sigma\chi}|_{W_F})_{\widehat{G}}
$$
for $\chi \in X(M)/{\rm stab}_\sigma$, lifts to the map
\begin{equation} \label{alg_map}
X(M)/{\rm stab}_\sigma \rightarrow Y(\mathcal M^\circ_\lambda)/{\rm stab}_\lambda,
\end{equation}
which is the obvious map induced by $X(M) \rightarrow Y(\widehat{M}) \rightarrow Y(\mathcal M_\lambda^\circ)$, up to translation by an element in $Y(\mathcal M^\circ_\lambda)$ measuring the difference between $(\varphi_{\sigma}|_{W_F})_{\mathcal M^\circ_\lambda}$ and $(\lambda)_{\mathcal M^\circ_\lambda}$.  The map (\ref{alg_map}) is clearly a morphism of algebraic varieties. Hence the map $f$ is a morphism of algebraic varieties.

The fibers of $f$ are finite by a property of LLC. Finally, if $G/F$ is quasi-split, the morphism $f$ is surjective by another property of LLC.
\end{proof}

\begin{cor} Assume $G/F$ satisfies ${\rm LLC+}$, so that the map $f$ in Proposition \ref{map_structure} exists. Then $f$ induces a $\mathbb C$-algebra homomorphism $\mathfrak Z^{\rm st}(G) \rightarrow \mathfrak Z(G)$.  It is injective if $G/F$ is quasi-split.
\end{cor} 

\begin{Remark} For the group ${\rm GL}_n$ the constructions above are unconditional because the local Langlands correspondence and its enhancement ${\rm LLC+}$ are known (cf.~Remark \ref{comp_w_ind_rem}(2) and Corollary \ref{invariant_Levi}).  One can see that $\mathfrak X_{{\rm GL}_n} \rightarrow \mathfrak Y_{{\rm GL}_n}$ is an isomorphism and hence $\mathfrak Z^{\rm st}({\rm GL}_n) = \mathfrak Z({\rm GL}_n)$.
\end{Remark}

\begin{Remark}
As remarked by Scholze and Shin \cite[$\S6$]{SS}, one may conjecturally characterize the image of $\mathfrak Z^{\rm st}(G) \rightarrow \mathfrak Z(G)$ in a way that avoids direct mention of $L$-parameters.  According to them it should consist of the distributions $D \in \mathcal D(G)^G_{\rm ec}$ such that, for any function $f \in C^\infty_c(G(F)$ whose stable orbital integrals vanish at semi-simple elements, the function $D * f$ also has this property.  See \cite[$\S6$]{SS} for further discussion of this.  From conjectured relations between stable characters and stable orbital integrals, one can conjecturally rephrase the condition on $D$ in terms of stable characters, as
\begin{equation} \label{st_char_cond}
{\rm SO}_{\varphi}(D*f) = 0, \, \forall \varphi ,\,\, \mbox{if} \,\, {\rm SO}_\varphi(f) =0, \, \forall \varphi.
\end{equation}
An element of $\mathfrak Z^{\rm st}(G)$ acts by the same scalar on all $\pi \in \Pi_\varphi$, and so the above condition holds if $D \in f(\mathfrak Z^{\rm st}(G))$.  The converse direction is much less clear, and implies non-trivial statements about the relation between supercuspidal supports, $L$-packets, and infinitesimal classes.  Indeed, suppose we are given $D \in \mathfrak Z(G)$ that satisfies (\ref{st_char_cond}).  This should mean that it acts by the same scalar on all $\pi \in \Pi_\varphi$. On the other hand, saying $D$ comes from $\mathfrak Z^{\rm st}(G)$ would mean that $D$ acts by the same scalar on all $\pi \in \Pi_\lambda$ and those scalars vary algebraically as $\lambda$ ranges over $\mathfrak Y$.  So if for some $\lambda$ the infinitesimal class
$$
\Pi_\lambda = \coprod_{\varphi \leadsto \lambda} \Pi_\varphi
$$
contains an $L$-packet $\Pi_{\varphi_0} \subsetneq \Pi_\lambda$ such that the set of supercuspidal supports coming from $\Pi_{\varphi_0}$ does not meet the set of those coming from any $\Pi_{\varphi'}$ with $\varphi'$ not conjugate to $\varphi_0$, then one could construct a regular function $D \in \mathfrak Z(G)$ which is constant on the $L$-packets $\Pi_{\varphi}$ but not constant on $\Pi_\lambda$, and thus not in $f(\mathfrak Z^{\rm st}(G))$.  In that case (\ref{st_char_cond}) would not be sufficient to force $D \in f(\mathfrak Z^{\rm st}(G))$, and the conjecture of Scholze-Shin would be false.  In that case, one could define the subring $\mathfrak Z^{\rm st*}(G) \subseteq \mathfrak Z(G)$ of regular functions on the Bernstein variety which take the same value on all supercuspidal supports of representations in the same $L$-packet.  This would then perhaps better deserve the title ``stable Bernstein center'' and it would be strictly larger than $f(\mathfrak Z^{\rm st}(G))$ at least in some cases.

 To illustrate this in a more specific setting, suppose $G/F$ is quasi-split and $\lambda$ does not factor through any proper Levi subgroup of $^LG$. Then by Proposition \ref{all_sc_condition} below, we expect $\Pi_\lambda$  to consist entirely of supercuspidal representations.  {\bf If} $\Pi_\lambda$ contains at least two $L$-packets $\Pi_\varphi$, {\bf then} there would exist a $D \in \mathfrak Z(G)$ which is constant on the $\Pi_\varphi$'s yet not constant on $\Pi_\lambda$, and the Scholze-Shin conjecture should be false.  Put another way, if the Scholze-Shin conjecture is true, we expect that whenever $\lambda$ does not factor through a proper Levi in $^LG$, the infinitesimal class $\Pi_\lambda$ consists of at most one $L$-packet.\footnote{Note added in proof (Feb.~2014): In fact this statement holds: if $\lambda$ does not factor through a proper Levi subgroup of $^LG$, then there is at most one way to extend it to an admissible homomorphism $\varphi: W'_F \rightarrow \, ^LG$.}
\end{Remark}

\subsection{Aside: when does an infinitesimal class consist only of supercuspidal representations?}

\begin{prop} \label{all_sc_condition}
Assume $G/F$ is quasi-split and ${\rm LLC+}$ holds for $G$. Then $\Pi_\lambda$ consists entirely of supercuspidal representations if and only if $\lambda$ does not factor through any proper Levi subgroup $^LM \subsetneq \, ^LG$.
\end{prop}

\begin{proof}
If $\Pi_\lambda$ contains a nonsupercuspidal representation $\pi$ with supercuspidal support $(M,\sigma)_G$ for $M \subsetneq G$, then by ${\rm LLC+}$, we may assume $\varphi_\pi|_{W_F}$, and hence $\lambda$, factors through the proper Levi subgroup $^LM \subsetneq \, ^LG$.

Conversely, if $\lambda$ factors minimally through a standard Levi subgroup $\mathcal M_\lambda \subsetneq \, ^LG$, then we must show that $\Pi_\lambda$ contains a nonsupercuspidal representation of $G$. Since $G/F$ is quasi-split, we may identify $\mathcal M_\lambda = \, ^LM_\lambda$ for an $F$-Levi subgroup $M_\lambda \subsetneq G$.

Now for $\mathfrak t = [\lambda]_{\widehat{G}}$, the map (\ref{local_map}) is surjective.  For any $F$-Levi subgroup $M \supsetneq M_\lambda$, a component of the form $\mathfrak X_{[M,\sigma]_G}$ has dimension ${\rm dim}\,(Z(\widehat{M})^I)_\Phi^\circ < {\rm dim}\,(Z(\widehat{M_\lambda})^I)_\Phi^\circ = {\rm dim}\,\mathfrak Y_{\mathfrak t}$.  Thus the union of the components of the form $\mathfrak X_{[M,\sigma]_G}$ with $M \supsetneq M_\lambda$ cannot surject onto $\mathfrak Y_\mathfrak t$.  Thus there must be a component of the form $\mathfrak X_{[M_\lambda, \sigma_\lambda]}$ appearing in the left hand side of (\ref{local_map}).  We may assume $\varphi_{\sigma_\lambda}$ factors through $^LM_\lambda$ along with $\lambda$.   Writing $(\lambda)_{\widehat{M}_\lambda} = (z^{cocyc}_\chi \varphi_{\sigma_\lambda})_{\widehat{M}_\lambda}$ for some $\chi \in X(M_\lambda)$, it follows that $\Pi_\lambda$ contains the nonsupercuspidal representations with supercuspidal support $(M_\lambda, \sigma_\lambda \chi)_G$.
\end{proof}

\subsection{Construction of the distributions $Z_V$} \label{const_Z_V_subsec}

Let $(r,V)$ be a finite-dimensional algebraic representation of $^LG$ on a complex vector space.  
Given a geometric Frobenius element $\Phi \in W_F$ and an admissible homomorphism $\lambda: W_F \rightarrow \, ^LG$, we may define the {\em semi-simple trace} 
$$
{\rm tr}^{\rm ss}(\lambda(\Phi), V)) := {\rm tr}(r\lambda(\Phi), V^{r\lambda(I_F)}).
$$
Note this is independent of the choice of $\Phi$. 

This notion was introduced by Rapoport \cite{Ra90} in order to define semi-simple local $L$-functions $L(s, \pi_p, r)$, and is parallel to the notion for $\ell$-adic Galois representations used in \cite{Ra90}~to define semi-simple local zeta functions $\zeta^{\rm ss}_{\mathfrak p}(X,s)$; see also \cite{HN02a, H05}.

The following result is an easy consequence of the material in $\S\ref{variety_str_subsec}$.

\begin{prop} \label{Z_V_construction}
The map $\lambda \mapsto {\rm tr}^{\rm ss}(\lambda(\Phi),V)$ defines a regular function on the variety $\mathfrak Y$ hence defines an element $Z_V \in \mathfrak Z^{\rm st}(G)$ by
$$
Z_V((\lambda)_{\widehat{G}}) =   {\rm tr}^{\rm ss}(\lambda(\Phi),V).
$$
We use the same symbol $Z_V$ to denote the corresponding element in $\mathfrak Z(G)$ given via $\mathfrak Z^{\rm st}(G) \rightarrow \mathfrak Z(G)$. The latter has the property
\begin{equation} \label{Z_V_Z(G)}
Z_V(\pi) = {\rm tr}^{\rm ss}(\varphi_\pi(\Phi),V)
\end{equation}
for every $\pi \in \Pi(G/F)$, where $Z_V(\pi)$ stands for $Z_V((M,\sigma)_G)$ if $(M,\sigma)_G$ is the supercuspidal support of $\pi$.
\end{prop}

\begin{Remark}
One does not really need the full geometric structure on the set $\mathfrak Y$ in order to construct $Z_V \in \mathfrak Z(G)$: one may show directly, assuming that LLC and Conjecture \ref{comp_w_induction} hold, that $\pi \mapsto {\rm tr}^{\rm ss}(\varphi_\pi(\Phi),V)$ descends to give a regular function on $\mathfrak X$ and hence (\ref{Z_V_Z(G)}) defines an element $Z_V \in \mathfrak Z(G)$. Using the map $f$ simply makes the construction more transparent (but has the drawback that we also need to assume Conjecture \ref{weak_invariant}).
\end{Remark}

\section{The Langlands-Kottwitz approach for arbitrary level structure} \label{LK_approach_sec}

\subsection{The test functions} \label{test_fcn_subsec}

Let $({\bf G},X)$ be a Shimura datum, where $X$ is the ${\bf G}(\mathbb R)$-conjugacy class of an $\mathbb R$-group homomorphism $h: {\rm R}_{\mathbb C/\mathbb R} {\mathbb G}_m \rightarrow {\bf G}_{\mathbb R}$.  This gives rise to the reflex field ${\mathbf E} \subset \mathbb C$ and a ${\bf G}(\mathbb C)$-conjugacy class $\{ \mu \} \subset X_*({\bf G}_\mathbb C)$ which is defined over ${\bf E}$.  

Choose a quasi-split group ${\bf G}^*$ over $\mathbb Q$ and an inner twisting ${
\psi}: {\bf G}^* \rightarrow {\bf G}$ of $\mathbb Q$-groups.  In particular we get an inner twisting ${\bf G}^*_{\bf E} \rightarrow {\bf G}_{\bf E}$ as well as an isomorphism of $L$-groups $^L({\bf G}_{\bf E}) ~ \widetilde{\rightarrow} ~ \, ^L({\bf G}^*_{\bf E})$. 

Let $\overline{\mathbb Q} \subset \mathbb C$ denote the algebraic numbers, so that we have an inclusion ${\bf E} \subset \overline{\mathbb Q}$ and we can 
regard $\{ \mu \}$ as a ${\bf G}(\overline{\mathbb Q})$-conjugacy class in $X_*({\bf G}_{\overline{\mathbb Q}})$ which is defined over ${\bf E}$ (cf. \cite[Lemma 1.1.3]{Ko84b}). 
Using ${\bf \psi}$ regard $\{ \mu \}$ as a ${\bf G}^*(\overline{\mathbb Q})$-conjugacy class in $X_*({\bf G}^*_{\overline{\mathbb Q}})$, defined over ${\bf E}$.  By Kottwitz' lemma  (\cite[1.1.3]{Ko84b}), $\{ \mu \}$ is represented by an ${\bf E}$-rational cocharacter $\mu: \mathbb G_m \rightarrow {\bf G}^*_{\bf E}$.  Following Kottwitz' argument in \cite[2.1.2]{Ko84b}, it is easy to show that there exists a unique representation $({\bf r}_{-\mu}, V_{-\mu})$ of $^L({\bf G}_{\bf E})$ such that as a representation of $\widehat{\bf G}^*$ it is an irreducible representation with extreme weight $-\mu$ and the Weil group $W_{\bf E}$ acts trivially on the highest-weight space 
corresponding to any $\Gamma_{\bf E}$-fixed splitting for $\widehat{\bf G}^*_{\bf E}$.

Using ${\bf \psi}$ we can regard $({\bf r}_{-\mu}, V_{-\mu})$ as a representation of $^L({\bf G}_{\bf E})$.  The isomorphism class of this representation depends only on the equivalence class of the inner twisting ${\bf \psi}$, thus only on ${\bf G}$ and $\{ \mu \}$.

Now we fix a rational prime $p$ and set $G := {\bf G}_{\mathbb Q_p}$. Choose a prime ideal $\mathfrak p \subset {\bf E}$ lying above $p$, and set $E := {\bf E}_\mathfrak p$.  Choose an algebraic closure $\overline{\mathbb Q}_p$ of $\mathbb Q_p$ and fix henceforth an isomorphism of fields $\mathbb C \cong \overline{\mathbb Q}_p$ such that the embedding ${\bf E} \hookrightarrow \mathbb C \cong \overline{\mathbb Q}_p$ corresponds to the prime ideal $\mathfrak p$. This gives rise to an embedding $\overline{\mathbb Q} \hookrightarrow \overline{\mathbb Q}_p$ extending ${\bf E} \hookrightarrow \overline{\mathbb Q}_p$, and thus to an embedding $W_E \hookrightarrow W_{\bf E}$. We get from this an embedding $^L(G_E) \hookrightarrow \, ^L({\bf G}_{\bf E})$.  Via this embedding we can regard $({\bf r}_{-\mu},V_{-\mu})$ as a representation $(r_{-\mu}, V_{-\mu})$ of $^L(G_E)$. 

Associated to $(r_{-\mu}, V_{-\mu}) \in {\rm Rep}(\,^L(G_E))$ we have an element $Z_{V_{-\mu}}$ in the Bernstein center $\mathfrak Z(G_E)$.  Of course here and in what follows, we are assuming ${\rm LLC+}$ holds for $G_E$.

Now we review briefly the Langlands-Kottwitz approach to studying the local Hasse-Weil zeta functions of Shimura varieties. Let $Sh_{K_p} = Sh({\bf G}, h^{-1}, K^p K_p)$ denote the canonical model\footnote{We use this term in the same sense as Kottwitz \cite{Ko92a}, comp.~Milne \cite[$\S1$, esp.~1.10]{Mil}.} over ${\bf E}$ for the Shimura variety attached to the data $({\bf G}, h^{-1}, K^p K_p)$ for some sufficiently small compact open subgroup $K^p \subset {\bf G}(\mathbb A^p_f)$ and some compact open subgroup $K_p \subset {\bf G}(\mathbb Q_p)$. We limit ourselves to constant coefficients 
$\overline{\mathbb Q}_\ell$ in the generic fiber of $Sh_{K_p}$ (here $\ell \neq p$ is a fixed rational prime).  Let $\Phi_p$ denote any geometric Frobenius element in ${\rm Gal}(\overline{\mathbb Q}_p/\mathbb Q_p)$. Then in the Langlands-Kottwitz approach to the semi-simple local zeta function $\zeta^{\rm ss}_\mathfrak p(s, Sh_{K_p})$, one needs to prove an identity of the form
\begin{equation} \label{Lef^ss}
{\rm tr}^{\rm ss}(\Phi^r_p, {\rm H}^\bullet_c(Sh_{K_p} \otimes_{\bf E} \overline{\mathbb Q}_p \, , \overline{\mathbb Q}_\ell)) = \sum_{(\gamma_0; \gamma, \delta)} c(\gamma_0; \gamma, \delta) ~ {\rm O}_\gamma(1_{K^p}) ~ {\rm TO}_{\delta \theta}(\phi_r).
\end{equation}
Here the semi-simple Lefschetz number ${\rm Lef}^{\rm ss}(\Phi^r_p, Sh_{K_p})$ 
on the left hand side is the alternating semi-simple trace of Frobenius on the compactly-supported $\ell$-adic cohomology groups\footnote{The Langlands-Kottwitz method really applies to the middle intersection cohomology groups of the Baily-Borel compactification and not just to the cohomology groups with compact supports; see \cite{Ko90} and \cite{Mor} for some general conjectures and results in this context, at primes of good reduction. The identity (\ref{Lef^ss}) corresponds to the contribution of the interior, at primes of arbitrary reduction, and is a first step toward understanding the intersection cohomology groups.} of $Sh_{K_p}$ (see \cite{Ra90} and \cite{HN02a} for the notion of semi-simple trace).  The expression on the right has precisely the same form as the counting points formula proved by Kottwitz in certain good reduction cases (PEL type A or C, $K_p$ hyperspecial; cf. \cite[(19.6)]{Ko92a}).  The integer $r \geq 1$ ranges over integers of the form $j \cdot [k_{E_0}:\mathbb F_p]$, $j \geq 1$, where $E_0/\mathbb Q_p$ is the maximal unramified subextension of $E/\mathbb Q_p$ and $k_{E_0}$ is its residue field.  Thus $\Phi^r_p = \Phi^j_{\mathfrak p}$ where $\Phi_{\mathfrak p}$ is a geometric Frobenius element in ${\rm Gal}(\overline{\mathbb Q}_p/E)$.  Finally, $\phi_r$ is an element in the Hecke algebra $\mathcal H(G(\mathbb Q_{p^r}), K_{p^r})$ with values in $\overline{\mathbb Q}_\ell$, where $\mathbb Q_{p^r}$ is the unique degree $r$ unramified extension in $\overline{\mathbb Q}_p/\mathbb Q_p$, and where $K_{p^r} \subset G(\mathbb Q_{p^r})$ is a suitable compact open subgroup which is assumed to be a natural analogue of $K_p \subset G(\mathbb Q_p)$. To be more precise about $K_{p^r}$, in practice there is a smooth connected $\mathbb Z_p$-model $\mathcal G$ for $G$, such that $K_p = \mathcal G(\mathbb Z_p)$. In that case, we always take $K_{p^r} = \mathcal G(\mathbb Z_{p^r})$, where $\mathbb Z_{p^r}$ is the ring of integers in $\mathbb Q_{p^r}$. In forming ${\rm TO}_{\delta \sigma}(\phi_r)$, the Haar measure on $G(\mathbb Q_{p^r})$ is normalized to give $K_{p^r}$ measure $1$. 

Let $E_j/E$ be the unique unramified extension of degree $j$ in $\overline{\mathbb Q}_p/E$.  Let $E_{j0}/\mathbb Q_p$ be the maximal unramified subextension of $E_j/\mathbb Q_p$.  So $E/E_0$ and $E_j/E_{j0}$ are totally ramified of the same degree, and $E_{j0} = \mathbb Q_{p^r}$.  

We make the choice of $\sqrt{p} \in \overline{\mathbb Q}_\ell$, and use it to define $\delta_P^{1/2}$ as a function with values in $\mathbb Q(\sqrt{p}) \subset \overline{\mathbb Q}_\ell$. We can now specify the test function $\phi_r \in \mathcal Z(G(E_{j0}),K_{j0})$, which will take values in $\overline{\mathbb Q}_\ell$.

In the construction of the elements $Z_V \in \mathfrak Z^{\rm st}(G)$, everything works the same way for $(r,V)$ a representation of $^LG := \widehat{G}(\overline{\mathbb Q}_\ell) \rtimes W_F$ on a $\overline{\mathbb Q}_\ell$-vector space.  We henceforth take this point of view. Let $(r_{-\mu,j}, V_{-\mu, j}) \in {\rm Rep}(\, ^L(G_{E_j}))$ denote the restriction of $(r_{-\mu}, V_{-\mu}) \in {\rm Rep}(\, ^L(G_E))$ via $\, ^L(G_{E_j}) \hookrightarrow \, ^L(G_E)$.  We can then induce to get a representation $(r^{E_{j0}}_{-\mu,j}, V^{E_{j0}}_{-\mu,j})$ of $^L(G_{E_{j0}})$. By Mackey theory, we get the same representation if we first induce to $^L(G_{E_0})$ and then restrict to $^L(G_{E_{j0}})$, that is, we have
\begin{equation} \label{test_rep_defn}
(r^{E_{j0}}_{-\mu,j}, V^{E_{j0}}_{-\mu,j}) := {\rm Ind}^{\widehat{G} \rtimes W_{E_{j0}}}_{\widehat{G} \rtimes W_{E_j}} \, {\rm Res}^{\widehat{G}\rtimes W_{E}}_{\widehat{G} \rtimes W_{E_j}} r_{-\mu} = {\rm Res}^{\widehat{G} \rtimes W_{E_0}}_{\widehat{G} \rtimes W_{E_{j0}}} \, {\rm Ind}^{\widehat{G} \rtimes W_{E_0}}_{\widehat{G} \rtimes W_E} r_{-\mu}.
\end{equation}

This gives rise to $Z_{V^{E_{j0}}_{-\mu,j}} \in \mathfrak Z^{\rm st}(G_{E_{j0}})$.   By abuse of notation, we use the same symbol to denote the image of this in the Bernstein center: $Z_{V^{E_{j0}}_{-\mu,j}} \in \mathfrak Z(G_{E_{j0}})$. Of course here we are viewing $\mathfrak Z(G_{E_{j0}})$ as $\overline{\mathbb Q}_\ell$-valued regular functions on the Bernstein variety, or equivalently as $\overline{\mathbb Q}_\ell$-valued invariant essentially compact distributions: the topology on $\mathbb C$ playing no role, it is harmless to identify it with $\overline{\mathbb Q}_\ell$.


The following is the conjecture formulated jointly with R.~Kottwitz.

\begin{conj} {\rm (Test function conjecture)} \label{TFC}
Let $d = {\rm dim}(Sh_{K_p})$. The test function $\phi_r$ in \textup{(}\ref{Lef^ss}\textup{)} may be taken to be $p^{rd/2}\big(Z_{V^{E_{j0}}_{-\mu,j}} * 1_{K_{p^r}}\big)$.  In particular, $\phi_r$ may be taken in the center $\mathcal Z(G(\mathbb Q_{p^r}), K_{p^r})$ of $\mathcal H(G(\mathbb Q_{p^r}), K_{p^r})$ and these test functions vary compatibly with change in the level $K_p$ in an obvious sense.
\end{conj}
The same test functions should be used when one incorporates arbitrary Hecke operators away from $p$ into (\ref{Lef^ss}).

Following Rapoport's strategy (cf.~\cite{Ra90}, \cite{Ra05}, \cite{H05}), one seeks to find a natural integral model $\mathcal M_{K_p}$ over $\mathcal O_{E}$ for $Sh_{K_p}$, and then rephrase the above conjecture using the method of nearby cycles $R\Psi := R\Psi^{\mathcal M_{K_p}}(\overline{\mathbb Q}_\ell)$.

\begin{conj} \label{RPsi_TFC}
There exists a natural integral model $\mathcal M_{K_p}/\mathcal O_E$ for $Sh_{K_p}$, such that
\begin{equation} \label{RPsi_TFC_eq}
\sum_{x \in \mathcal M_{K_p}(k_{E_{j0}})} {\rm tr}^{\rm ss}(\Phi_p^r, R\Psi_x) = \sum_{(\gamma_0; \gamma, \delta)} c(\gamma_0; \gamma, \delta) ~ {\rm O}(1_{K^p}) ~ {\rm TO}_{\delta \theta}(\phi_r),
\end{equation}
where $\phi_r = p^{rd/2} \Big(Z_{V^{E_{j0}}_{-\mu,j}} * 1_{K_{p^r}}\Big)$ as in Conjecture \ref{TFC}.
\end{conj}
\begin{Remark} Implicit in this conjecture is that the method of nearby cycles can be used for compactly-supported cohomology. In fact we could conjecture there exists a suitably nice compactification of $\mathcal M_{K_p}/\mathcal O_E$ so that the natural map
$$
{\rm H}^i_c(\mathcal M_{K_p} \otimes_{\mathcal O_E} \overline{\mathbb F}_p \, , \, R\Psi(\overline{\mathbb Q}_\ell)) \rightarrow {\rm H}^i_c( Sh_{K_p} \otimes_{E} \overline{\mathbb Q}_p \, , \, \overline{\mathbb Q}_\ell)
$$
is a Galois-equivariant isomorphism. For $G = {\rm GSp}_{2g}$ and where $\mathcal M_{K_p}$ is the natural integral model for $Sh_{K_p}$ for $K_p$ an Iwahori subgroup, this was proved by Benoit Stroh. Of course, one is really interested in intersection cohomology groups of the Baily-Borel compactification (see footnote 5), and in fact Stroh \cite{Str} computed the nearby cycles and verified the analogue of the Kottwitz conjecture on nearby cycles (see Conjecture \ref{Kottwitz_conj_RPsi} below) for these compactifications.
\end{Remark}

\begin{Remark} Some unconditional versions of Conjectures \ref{TFC} and \ref{RPsi_TFC} have been proved.  See $\S\ref{evidence_TFC_sec}$.
\end{Remark}

\subsection{Endoscopic transfer of the stable Bernstein center}

Part of the Langlands-Kottwitz approach is to perform a ``pseudostabilization'' of (\ref{Lef^ss}), and in particular prove the ``fundamental lemmas'' that are required for this. The {\em stabilization} expresses ({\ref{Lef^ss}}) in the form $\sum_{\bf H} i({\bf G},{\bf H}) \, ST_e^*({\bf h})$, the sum over global $\mathbb Q$-elliptic endoscopic groups ${\bf H}$ for ${\bf G}$ of the $({\bf G},{\bf H})$-regular $\mathbb Q$-elliptic part of the geometric side of the stable trace formula for $({\bf H}, {\bf h})$ (cf. notation of \cite{Ko90}), for a certain {\em transfer} function ${\bf h} \in C^\infty_c({\bf H}(\mathbb A))$.  (By contrast in ``pseudostabilization'' which is used in certain situations, one instead writes ({\ref{Lef^ss}) in terms of the trace formula for ${\bf G}$ and not its quasi-split inner form, and this is sometimes enough, as in ~e.g.~Theorem \ref{Z_fcn_thm} below.) For stabilization one needs to produce elements ${\bf h}_p \in C^\infty_c({\bf H}(\mathbb Q_p))$ which are Frobenius-twisted endoscopic transfers of $\phi_r$. The existence of such transfers ${\bf h}_p$ is due mainly to the work of Ng\^{o} \cite{Ngo} and Waldspurger \cite{Wal97, Wal04, Wal08}. But we hope to have a priori spectral information about the transferred functions ${\bf h}_p$.  

A guiding principle is that the nearby cycles on an appropriate ``local model'' for ${Sh}_{K_p}$ should {\em naturally} produce a central element as a test function $\phi_r$, which should coincide with that given by the test function conjecture (cf. Conjecture \ref{RPsi_TFC}); then its spectral behavior is known by construction.  In that case one can formulate a conjectural endoscopic transfer $h_p$ of $\phi_r$ with known spectral behavior. 

General Frobenius-twisted endoscopic transfer homomorphisms $\mathfrak Z^{\rm st}(G_{\mathbb Q_{p^r}}) \rightarrow \mathfrak Z^{\rm st}(H_{\mathbb Q_p})$ will be described elsewhere. Here for simplicity we content ourselves to describe two special cases: standard (untwisted) endoscopic transfer of the geometric Bernstein center, and the base change transfer for the stable Bernstein center.

\subsubsection{Endoscopic transfer of the geometric Bernstein center}

Let us fix an endoscopic triple $(H,s,\eta_0)$ for $G$ over a $p$-adic field $F$ (cf. \cite[$\S7$]{Ko84a}), and suppose we have fixed an extension $\eta : \, ^LH \rightarrow \, ^LG$ of $\eta_0 : \widehat{H} \hookrightarrow \widehat{G}$ (we suppose we are in a situation, e.g. $G_{\rm der} = G_{\rm sc}$, where such extensions always exist). We could hope the natural map 
\begin{align*}
\mathfrak Y^{H/F} &\longrightarrow \mathfrak Y^{G/F} \\
(\lambda)_{\widehat{H}} &\longmapsto (\eta \circ \lambda)_{\widehat{G}}.
\end{align*}
would be algebraic and hence would induce an {\em endoscopic transfer homomorphism} $\mathfrak Z^{\rm st}(G) \rightarrow \mathfrak Z^{\rm st}(H)$. By invoking further expectations about endoscopic lifting, one would then formulate a map on the level of Bernstein centers, $\mathfrak Z(G) \rightarrow \mathfrak Z(H)$, which we could write as $Z \mapsto Z|_\eta$.  But these assertions are not obvious.  Fortunately, in practice we need this construction rather on the {\em geometric Bernstein center}.

\begin{defn}
Assume ${\rm LLC+}$ holds for $G/F$. We define the {\bf geometric Bernstein center} $\mathfrak Z^{\rm geom}(G)$ to be the subalgebra of $\mathfrak Z^{\rm st}(G)$ generated by the elements $Z_V$ as $V$ ranges over ${\rm Rep}(\,^LG)$.
\end{defn}
The terminology {\em geometric Bernstein center} is motivated by $\S\ref{geom_Lang_sec}$ below.

Let $V|_{\eta} \in {\rm Rep}(\,^LH)$ denote the restriction of $V \in {\rm Rep}(\,^LG)$ along $\eta$.  Further assume ${\rm LLC+}$ also holds for $H/F$. Then $Z_V \mapsto Z_{V|_{\eta}}$ determines  a map $\mathfrak Z^{\rm geom}(G) \rightarrow \mathfrak Z^{\rm geom}(H)$.  Write $Z^G_V$ (resp. $Z^H_{V|_{\eta}}$) for the image of $Z_V$ (resp. $Z_{V|_{\eta}}$) in $\mathfrak Z(G)$ (resp. $\mathfrak Z(H)$).

\begin{conj} \label{end_trans_conj} Assume ${\rm LLC+}$ holds for both $G$ and $H$. Then in the above situation the distribution $Z^H_{V|_{\eta}} \in \mathfrak Z(H)$ is the endoscopic transfer of $Z^G_V \in \mathfrak Z(G)$ in the following sense: whenever a function $\phi^H \in C^\infty_c(H(F))$ is a transfer of a function $\phi \in C^\infty_c(G(F))$, then $Z^H_{V|_{\eta}} * \phi^H$ is a transfer of $Z^G_V * \phi$.
\end{conj}

This conjecture and its Frobenius-twisted analogue were announced by the author in April  2011 at Princeton \cite{H11}.  A very similar statement subsequently appeared as Conjecture 7.2 in \cite{SS}. Considering the untwisted case for simplicity, the difference is that in \cite{SS}, the authors take in place of $Z_V$ an element in the stable Bernstein center essentially of the form
$$
(\lambda)_{\widehat{G}} \mapsto {\rm tr}(\lambda(\Phi_F), V_{-\mu}),$$
where here the usual trace, not the semi-simple trace, is used. That conjecture is proved in \cite{SS}~in all EL or quasi-EL cases, by invoking special features of general linear groups such as the existence of base change representations.

Formally, Conjecture \ref{end_trans_conj} contains as a special case the ``fundamental lemma implies spherical transfer'' result of Hales \cite{Hal} (see also Waldspurger \cite{Wal97}). Indeed if $K, K_H$ are hyperspecial maximal compact subgroups in $G(F), H(F)$, then $1_{K_H}$ is a transfer of $1_K$ by the fundamental lemma, and hence $Z^H_{V|_{\eta}} * 1_{K_H}$ is a transfer of $Z^G_V * 1_K$.  But by the Satake isomorphism, every $K$-spherical function on $G(F)$ is of the form $Z^G_V * 1_K$ for some representation $V$ (comp.~$\S\ref{geom_Lang_sec}$).

Even in more general situations, Conjecture \ref{end_trans_conj} is most useful when applied to a pair $\phi, \phi^H$ of unit elements in appropriate Hecke algebras. At least when $G$ splits over $F^{un}$, Kazhdan-Varshavsky proved in \cite{KV} that for some explicit scalar $c$, the Iwahori unit $c 1_{I_H}$ is a transfer of the Iwahori unit $1_I$. As another example, if $K_n^G \subset G(F)$ is the $n$-th principal congruence subgroup in $G(F)$, then for some explicit scalar $c$ the function $c1_{K_n^H}$ is a transfer of $1_{K_n^G}$ (proved by Ferrari \cite{Fer} under some mild restrictions on the residue characteristic of $F$), and thus $c (Z^H_{V|_\eta} * 1_{K^H_n})$ should be an explicit transfer of $Z^G_V * 1_{K^G_n}$. A Frobenius-twisted analogue of Ferrari's theorem together with the Frobenius-twisted analogue of Conjecture \ref{end_trans_conj} would give an explicit Frobenius-twisted transfer of the test function $\phi_r$ from Conjecture \ref{TFC}, if $K_p$ is a principal congruence subgroup.

\subsubsection{Base change of the stable Bernstein center}

We return to the situation of Proposition \ref{change_of_field}, but we specialize it to cyclic Galois extensions of $F$ and furthermore we assume $G/F$ is quasi-split. Let $E/F$ be any finite cyclic Galois subextension of $\overline{F}/F$ with Galois group $\langle \theta \rangle$, and with corresponding inclusion of Weil groups $W_E \hookrightarrow W_F$.  

If $\phi \in \mathcal H(G(E))$ and $f \in \mathcal H(G(F))$ are functions in the corresponding Hecke algebras of locally constant compactly-supported functions, then we say $\phi, f$ are {\em associated} (or {\em $f$ is a base-change transfer of $\phi$}), if the following result holds for the stable (twisted) orbital integrals:  for every semisimple element $\gamma \in G(F)$, we have
\begin{equation} \label{stable_BC_defn}
{\rm SO}_\gamma(f) = \sum_{\rm \delta}\Delta(\gamma, \delta) \, {\rm SO}_{\delta \theta}(\phi)
\end{equation}
where the sum is over stable $\theta$-conjugacy classes $\delta \in G(E)$ with semisimple norm $\mathcal N\delta$, and where $\Delta(\gamma, \delta) = 1$ if $\mathcal N\delta = \gamma$ and $\Delta(\gamma,\delta) = 0$ otherwise.  See e.g.~\cite{Ko86}, \cite{Ko88}, \cite{Cl90}, or \cite{H09} for further discussion.  

\begin{conj} \label{BC_conj} In the above situation, consider $Z \in \mathfrak Z^{\rm st}(G_E)$, and consider its image, also denoted by $Z$, in $\mathfrak Z(G_E)$. Consider $b_{E/F}(Z) \in \mathfrak Z^{\rm st}(G)$ (cf.~Def.~\ref{b_E/F_defn}) and also denote by $b_{E/F}(Z)$ its image in $\mathfrak Z(G)$. Then $b_{E/F}(Z)$ is the base-change transfer of $Z \in \mathfrak Z(G_E)$, in the following sense: whenever a function $f \in C_c^\infty(G(F))$ is a base-change transfer of $\phi \in C_c^\infty(G(E))$, then $b_{E/F}(Z) * f$ is a base-change transfer of $Z * \phi$.
\end{conj}

\begin{prop} \label{BC_GLn}
Conjecture \ref{BC_conj} holds for ${\rm GL}_n$.
\end{prop}

\begin{proof}
The most efficient proof follows Scholze's proof of  Theorem C in \cite{Sch2} which makes essential use of the existence of cyclic base change lifts for ${\rm GL}_n$. Let $\pi \in \Pi({\rm GL}_n/F)$ be a tempered irreducible representation with base change lift $\Pi \in \Pi({\rm GL}_n/E)$, a tempered representation which is characterized by the character identity $\Theta_\Pi((g,\theta)) = \Theta_\pi(Ng)$ for all elements $g \in {\rm GL}_n(E)$ with regular semisimple norm $Ng$ (\cite[Thm.~6.2, p.~51]{AC}). Here $(g, \theta) \in {\rm GL}_n(E) \rtimes {\rm Gal}(E/F)$ and $\theta$ acts on $\Pi$ by the normalized intertwiner $I_\theta: \Pi \rightarrow \Pi$ of \cite[p.~11]{AC}.

Suppose $f$ is a base-change transfer of $\phi$. Using the Weyl integration formula and its twisted analogue (cf. \cite[p.~36]{AC}), we see that
$$
{\rm tr}((\phi, \theta)\,| \,\Pi) = {\rm tr}(f \, | \, \pi).
$$
Multiplying by the constant $Z(\Pi) = b_{E/F}(Z)(\pi)$, we get
$$
{\rm tr}((Z * \phi, \theta) \, | \, \Pi) = {\rm tr}(b_{E/F}(Z) * f \, | \, \pi).
$$
(Use Corollary \ref{Hecke_action_cor} and its twisted analogue.) There exists a base-change transfer $h \in C^\infty_c({\rm GL}_n(F))$ of $Z * \phi$ (\cite[Prop.~3.1]{AC}).  Using the same argument as above for the pair $Z * \phi$ and $h$, we conclude that  ${\rm tr}(b_{E/F}(Z) * f - h \, | \, \pi) = 0$ for every tempered irreducible $\pi \in \Pi({\rm GL}_n/F)$. By Kazhdan's density theorem (Theorem 1 in \cite{Kaz})  the regular semi-simple orbital integrals of $b_{E/F}(Z) * f$ and $h$ agree. Thus the (twisted) orbitals integrals of $b_{E/F}(Z) * f$ and $\phi$ match at all regular semi-simple elements, and hence at all semi-simple elements by Clozel's Shalika germ argument (\cite[Prop.~7.2]{Cl90}).
\end{proof}

\begin{Remark} Unconditional versions of Conjecture \ref{BC_conj} are available for parahoric and pro-p Iwahori-Hecke algebras, when $G/F$ is unramified.\footnote{The pro-p Sylow subgroup of an Iwahori subgroup $I \subset G(F)$ coincides with its pro-unipotent radical $I^+$, and it has become conventional to term the Hecke algebra $C^\infty_c(I^+ \backslash G(F)/I^+)$ the {\em pro-$p$ Iwahori-Hecke algebra}.} See $\S\ref{evidence_transfer_sec}$.
\end{Remark}

\subsection{Application: local Hasse-Weil zeta functions}

By Kottwitz' base change fundamental lemma for units \cite{Ko86}, we know $1_{K_p}$ is a base-change transfer of $1_{K_{p^r}}$ whenever $K_p = \mathcal G(\mathbb Z_p)$ and $K_{p^r} = \mathcal G(\mathbb Z_{p^r})$ for a smooth connected $\mathbb Z_p$-model $\mathcal G$ for $G$.  Then Conjectures \ref{TFC} and \ref{BC_conj} together say that 
\begin{equation} \label{f^j_defn}
f^{(j)}_p := p^{rd/2} \, b_{E_{j0}/\mathbb Q_p}(Z_{V^{E_{j0}}_{-\mu,j}}) * 1_{K_p}
\end{equation} 
is a base-change transfer of a test function $\phi_r$ that satisifies (\ref{Lef^ss}).  

Setting 
\begin{equation} \label{V^E_0}
(r^{E_0}_{-\mu}, V^{E_0}_{-\mu}) := {\rm Ind}^{\widehat{G} \rtimes W_{E_0}}_{\widehat{G} \rtimes W_{E}} r_{-\mu}
\end{equation}
we have for any admissible parameter $\varphi : W'_{\mathbb Q_p} \rightarrow \,^L(G_{\mathbb Q_p})$ and any $\pi_p \in \Pi_\varphi(G/\mathbb Q_p)$ the identity
\begin{equation} \label{tr_f^j_p}
{\rm tr}( f^{(j)}_p \, | \, \pi_p) = p^{rd/2} \,  {\rm dim}(\pi_p^{K_p}) \, {\rm tr}^{\rm ss}(\varphi(\Phi^r_p) \, , \, V^{E_0}_{-\mu}),
\end{equation}
where $r = j[E_0: \mathbb Q_p]$. In the compact and non-endoscopic cases, the above discussion allows us to express $\zeta^{\rm ss}_{\mathfrak p}(s, Sh_{K_p})$ in terms of semi-simple automorphic $L$-functions.  To explain this we need a detour on the point of view taken in \cite{L1, L2} (comp.~\cite[$\S2.2$]{Ko84b}).

Recall $({\bf r}_{-\mu}, V_{-\mu}) \in {\rm Rep}(\,^L(G_{\mathbf E}))$. Consider the Langlands representation 
$${\bf r} := {\rm Ind}^{\widehat{G} \rtimes W_\mathbb Q}_{\widehat{G} \rtimes W_{\mathbf E}} {\bf r}_{-\mu},$$
and for each prime $\mathfrak p$ of ${\mathbf E}$ dividing $p$, consider
$$
{\bf r}_\mathfrak p := {\rm Ind}^{\widehat{G} \rtimes W_{\mathbb Q_p}}_{\widehat{G} \rtimes W_{\mathbf E_\mathfrak p}} {\rm Res}^{\widehat{G} \rtimes W_{\mathbf E}}_{\widehat{G} \rtimes W_{\mathbf E_{\mathfrak p}}} {\bf r}_{-\mu} = {\rm Ind}^{\widehat{G} \rtimes W_{\mathbb Q_p}}_{\widehat{G} \rtimes W_{{\bf E}_\mathfrak p}} r_{-\mu}.
$$
Mackey theory gives
$$
{\rm Res}^{\widehat{G} \rtimes W_{\mathbb Q}}_{\widehat{G} \rtimes W_{\mathbb Q_p}} {\bf r} = \bigoplus_{\mathfrak p | p} {\bf r}_{\mathfrak p}.
$$
If $\mathfrak p$ is understood, let ${\mathbf E}_{\mathfrak p 0}/\mathbb Q_p$ denote the maximal unramified subextension of ${\mathbf E}_{\mathfrak p}/\mathbb Q_p$, and set $E = {\mathbf E}_{\mathfrak p}$ and $E_0 := {\mathbf E}_{\mathfrak p 0}$.  Then we have
\begin{equation} \label{r_mathfrakp}
{\bf r}_\mathfrak p = {\rm Ind}^{\widehat{G} \rtimes W_{\mathbb Q_p}}_{\widehat{G} \rtimes W_{E}} r_{-\mu} = {\rm Ind}^{\widehat{G} \rtimes W_{\mathbb Q_p}}_{\widehat{G} \rtimes W_{E_0}} r^{E_0}_{-\mu}.
\end{equation}

\begin{lemma}
Suppose $\pi_p \in \Pi_\varphi(G/\mathbb Q_p)$. Then 
\begin{equation}\label{r_vs_j}
[E_0: \mathbb Q_p]^{-1} {\rm tr}^{\rm ss}(\varphi(\Phi^r_p), {\bf r}_\mathfrak p) = \begin{cases}
{\rm tr}^{\rm ss}(\varphi(\Phi^j_\mathfrak p), r^{E_0}_{-\mu}) \,\,\,\, \mbox{if $r = j [E_0: \mathbb Q_p]$} \\ 0, \,\,\,\,\, \mbox{if $[E_0: \mathbb Q_p]\not| r$} \end{cases}
\end{equation}
\end{lemma}

\begin{proof}
There is an isomorphism of $\widehat{G} \rtimes W_{\mathbb Q_p}$-modules 
$${\bf r}_\mathfrak p \cong \mathbb C[\widehat{G} \rtimes W_{\mathbb Q_p}] \otimes_{\mathbb C[\widehat{G} \rtimes W_{E_0}]} r^{E_0}_{-\mu},$$ 
and ${\bf r}_\mathfrak p^{\varphi(I_{\mathbb Q_p})}$ has a $\mathbb C$-basis of the form $\{\varphi(\Phi_p^i) \otimes w_k\}$ where $0 \leq i \leq [E_0:\mathbb Q_p]-1$ and $\{ w_k\}$ comprises a $\mathbb C$-basis for $(r^{E_0}_{-\mu})^{\varphi(I_{\mathbb Q_p})}$.  The lemma follows.
\end{proof}

The following result shows the potential utility of Conjectures \ref{TFC} and \ref{BC_conj}. It applies not just to PEL Shimura varieties, but to any Shimura variety where these conjectures are known. Similar results will hold when incorporating Hecke operators away from $p$.

\begin{theorem}\label{Z_fcn_thm}
Suppose ${\bf G}_{\rm der}$ is anisotropic over $\mathbb Q$, so that the associated Shimura variety $Sh_{K_p} = Sh({\bf G}, h^{-1}, K^pK_p)$ is proper over ${\bf E}$.  Suppose ${\bf G}$ has ``no endoscopy'', in the sense that the group $\mathfrak K({\bf G}_{\gamma_0}/\mathbb Q)$ is trivial for every semisimple element $\gamma_0 \in {\bf G}(\mathbb Q)$, as in e.g.~\cite{Ko92b}. Let $\mathfrak p$ be a prime ideal of ${\bf E}$ dividing $p$. Assume $({\rm LLC}+)$ \textup{(}cf. $\S\ref{LLC+_subsec}$\textup{)}, and Conjectures \ref{TFC} and \ref{BC_conj} hold for all groups ${\bf G}_{\mathbb Q_{p^r}}$. 

Then in the notation above, we have
\begin{equation} \label{Z_fcn_eq}
\displaystyle
\zeta^{\rm ss}_{\mathfrak p}(s, Sh_{K_p}) = \prod_{\pi_f} L^{\rm ss}(s-\dfrac{d}{2}, \pi_p, {\bf r}_\mathfrak p)^{a(\pi_f)\, {\rm dim}(\pi^K_f)},
\end{equation}
where $\pi_f = \pi^p \otimes \pi_p$ runs over irreducible admissible representations of ${\bf G}(\mathbb A_f)$ and the integer $a(\pi_f)$ is given by
$$
a(\pi_f) = \sum_{\pi_\infty \in \Pi_\infty} m(\pi_f \otimes \pi_\infty) \, {\rm tr}(f_\infty | \pi_\infty),
$$
where $m(\pi_f \otimes \pi_\infty)$ is the multiplicity of $\pi_f \otimes \pi_\infty$ in 
${\rm L}^2({\bf G}(\mathbb Q) A_{\bf G}(\mathbb R)^\circ \backslash {\bf G}(\mathbb A))$. Here $A_{\bf G}$ is the $\mathbb Q$-split component of the center of ${\bf G}$ \textup{(}which we assume is also its $\mathbb R$-split component\textup{)}. Further $\Pi_\infty$ is the set of irreducible admissible representations of ${\bf G}(\mathbb R)$ which have trivial infinitesimal and central characters, and $f_\infty$ is defined as in \cite{Ko92b} to be $(-1)^{{\rm dim}(Sh_K)}$ times a pseudo-coefficient of an essentially discrete series member $\pi^0_\infty \in \Pi_\infty$.
\end{theorem}

\begin{proof}
The method follows closely the argument of Kottwitz in \cite{Ko92b} (comp.~\cite[$\S13.4$]{HRa}), so we just give an outline.  We will use freely the notation of Kottwitz and \cite{HRa}. Set $f = [{\bf E}_{\mathfrak p 0}: \mathbb Q_p]$. By definition we have
\begin{equation} \label{log_zeta}
{\rm log} \, \zeta^{\rm ss}_{\mathfrak p}(s, Sh_{K_p}) = \sum_{j =1}^\infty {\rm Lef}^{\rm ss}(\Phi^j_\mathfrak p , Sh_{K_p}) \, \dfrac{p^{-jfs}}{j}.
\end{equation}
By using (\ref{Lef^ss}) together with Conjectures \ref{TFC} and \ref{BC_conj}, the arguments of Kottwitz \cite{Ko92b} show that for each $j \geq 1$
\begin{equation}
{\rm Lef}^{\rm ss}(\Phi^j_\mathfrak p, Sh_{K_p}) = \tau({\bf G}) \sum_{\gamma_0}  {\rm SO}_{\gamma_0}(f^p \, f^{(j)}_p \, f_\infty),
\end{equation}
where $f^{(j)}_p$ is defined as in (\ref{f^j_defn}) and $f^p$ is the characteristic function of $K^p \subset {\bf G}(\mathbb A^p_f)$. Here $\gamma_0$ ranges over all stable conjugacy classes in ${\bf G}(\mathbb Q)$.

Since ${\bf G}_{\rm der}$ is anisotropic over $\mathbb Q$, the trace formula for any $f \in C^\infty_c(A_{\bf G}(\mathbb R)^\circ \backslash {\bf G}(\mathbb A))$ takes the simple form
\begin{equation} \label{simple_trace_form}
\sum_\gamma \tau ({\bf G}_\gamma) {\rm O}_\gamma (f) = \sum_{\pi} m(\pi) \, {\rm tr}(f | \pi),
\end{equation}
where $\gamma$ ranges over conjugacy classes in ${\bf G}(\mathbb Q)$ and $\pi$ ranges over irreducible representations in ${\rm L}^2({\bf G}(\mathbb Q) A_{\bf G}(\mathbb R)^\circ \backslash {\bf G}(\mathbb A))$.  By \cite[Lemma~4.1]{Ko92b}, the vanishing of all $\mathfrak K({\bf G}_{\gamma_0}/\mathbb Q)$ means that
\begin{equation*}
\sum_{\gamma} \tau({\bf G}_\gamma) \, {\rm O}_\gamma(f) = \tau({\bf G}) \sum_{\gamma_0} {\rm SO}_{\gamma_0}(f).
\end{equation*}
It follows that
\begin{align*}
{\rm Lef}^{\rm ss}(\Phi^j_\mathfrak p, Sh_{K_p}) &= \sum_\pi m(\pi) \, {\rm tr}(f^pf^{(j)}_pf_\infty | \pi) \\
&= \sum_{\pi_f} \sum_{\pi_\infty \in \Pi_\infty} m(\pi_f \otimes \pi_\infty) \cdot {\rm tr}(f^p | \pi^p_f) \cdot {\rm tr}(f^{(j)}_p | \pi_p) \cdot {\rm tr}(f_\infty | \pi_\infty) \\
&= \sum_{\pi_f} a(\pi_f) \, {\rm dim}(\pi^K_f) \, p^{jfd/2} \, {\rm tr}^{\rm ss}(\varphi_{\pi_p}(\Phi^j_\mathfrak p), V^{E_0}_{-\mu}),
\end{align*}
the last equality by (\ref{tr_f^j_p}).

By definition we have
\begin{equation*}
{\rm log} \, L^{\rm ss}(s, \pi_p, {\bf r}_\mathfrak p) = \sum_{r=1}^{\infty} {\rm tr}^{\rm ss}(\varphi_{\pi_p}(\Phi^r_p), {\bf r}_{\mathfrak p}) \, \dfrac{p^{-rs}}{r}.
\end{equation*}
Now (\ref{Z_fcn_eq}) follows by invoking (\ref{r_vs_j}).
\end{proof}

\begin{Remark}
Unconditional versions of Theorem \ref{Z_fcn_thm} are available for some parahoric or pro-p-Iwahori level cases, or for certain compact ``Drinfeld case'' Shimura varieties with arbitrary level; these cases are alluded to in $\S\ref{evidence_TFC_sec}$.
\end{Remark}

\subsection{Relation with geometric Langlands} \label{geom_Lang_sec}

For simplicity, assume $G$ is split over a $p$-adic or local function field $F$.  Assume $G$ satisfies ${\rm LLC+}$. From the construction of $Z_V$ in Proposition \ref{Z_V_construction}, we have a map
\begin{align} \label{gen_Sat_map}
K_0{\rm Rep}_{\mathbb C}(\widehat{G}) &\rightarrow \mathcal Z(G,J) \\
V &\mapsto Z_V * 1_J  \notag
\end{align}
for any compact open subgroup $J \subset G(F)$, which gives rise to a commutative diagram 
$$
\xymatrix@1{
&& \mathcal Z(G,J) \ar[d]^{-*_J 1_{I}} \\
&&  \mathcal Z(G,I) \ar[d]^{-*_I 1_K} \\
K_0{\rm Rep}(\widehat{G}) \ar[rr]^{{\rm Sat}}_{\sim} \ar[urr]^{\rm Bern}_{\sim} \ar@/^/[uurr] && \mathcal H(G,K) 
}
$$
whenever $J \subseteq I \subset K$ where $I$ resp. $K$ is an Iwahori resp. special maximal compact subgroup, and where the bottom two arrows are the Bernstein resp. Satake isomorphisms.  We warn the reader that the oblique arrow $K_0{\rm Rep}(\widehat{G}) \rightarrow \mathcal Z(G,J)$ is injective but not surjective in general, and also it is additive but not an algebra homomorphism in general. 

Gaitsgory \cite{Ga} constructed the two arrows ${\rm Sat}$ and ${\rm Bern}$ geometrically when $F$ is a function field, using nearby cycles for a degeneration of the affine Grassmannian ${\rm Gr}_G$ to the affine flag variety  ${\rm Fl}$ for $G$. One can hope that, as in the Iwahori case \cite{Ga}, one can construct the arrow $K_0 {\rm Rep}(\widehat{G}) \rightarrow \mathcal Z(G,J)$ categorically using nearby cycles for a similar degeneration of ${\rm Gr}_G$ to a ``partial affine flag variety'', namely an {\em fpqc}-quotient $L\mathfrak J/L^+\mathfrak J$ where $\mathfrak J$ is a smooth connected group scheme over $\overline{\mathbb F}_p[[t]]$ with generic fiber $\mathfrak J_{\mathbb F_p((t))} = G_{\mathbb F_p((t))}$ and $\mathfrak J(\mathbb F_p[[t]]) = J$.   Here $L\mathfrak J$ (resp.~$L^+\mathfrak J$) is the ind-scheme (resp. scheme) over $\mathbb F_p$ representing the sheaf of groups for the {\em fpqc}-topology whose sections for an $\mathbb F_p$-algebra $R$ are given by $L\mathfrak J({\rm Spec}\,R) = \mathfrak J(R[\![t]\!]\tiny{[\frac{1}{t}]})$ (resp.~$L^+\mathfrak J({\rm Spec}\,R) = \mathfrak J(R[\![t]\!])$). 
At least for $J = I^+$, the pro-p Iwahori subgroup, this will be realized in forthcoming joint work of the author and Benoit Stroh.

In a related vein, the geometric Satake equivalence of Mirkovic-Vilonen \cite{MV} is a categorical version of the Satake isomorphism {\rm Sat}, and this is usually stated when $G$ is a split group over $F = \mathbb F_p((t))$. One can ask for a version of this when $G$ is nonsplit, possibly not even quasisplit, over such a field $F$. The correct Satake isomorphism to ``categorify'' appears to be the one described in \cite{HRo}.  In many cases where $G$ is quasisplit and split over a tamely ramified extension of $F$, this has been carried out in recent work of X.~Zhu \cite{Zhu}.

\section{Test functions in the parahoric case} \label{parahoric_TFC_sec}

We fix $r = j[E_0:\mathbb Q_p]$ for some $j \in \mathbb N$.  We assume $K_p$ is a parahoric subgroup of $G(\mathbb Q_p)$, and we let $K_{p^r}$ denote the corresponding parahoric subgroup of $G(\mathbb Q_{p^r})$.  Assuming ${\rm LLC+}$ holds for $G_{\mathbb Q_{p^r}}$, we can speak of the test function 
\begin{equation} \label{test_fcn0}
\phi_r = p^{rd/2}\big(Z_{V^{E_{j0}}_{-\mu,j}} * 1_{K_{p^r}}\big) \in \mathcal Z(G(\mathbb Q_{p^r}), K_{p^r}).
\end{equation}
We wish to give a more concrete description of this function, making use of Bernstein's isomorphism for $\mathcal Z(G(\mathbb Q_{p^r}), K_{p^r})$ which is detailed in the Appendix, $\S\ref{appendix_sec}$.

In the next two subsections, we are concerned with the case where $G_{\mathbb Q_{p^r}}$ is 
quasisplit. We write $F := \mathbb Q_{p^r}$. Choose a maximal $F$-split torus $A$ in $G$, and let $T$ denote its centralizer in $G$. Fix an $F$-rational Borel subgroup $B$ containing $T$. Let $K_F \subset G(F)$ denote the parahoric subgroup corresponding to $K_p$.

By Kottwitz \cite[Lem.~(1.1.3)]{Ko84b}, the $G(\overline{\mathbb Q}_p)$-conjugacy class $\{ \mu \}$ is represented by an $F$-rational cocharacter $\mu \in X_*(T)^{\Phi_F} = X_*(A)$. It is clear that $E$, the field of definition of $\{ \mu \}$, is contained in any subfield of $\overline{\mathbb Q}_p$ which splits $G$. 

Given $\pi \in \Pi(G/F)$ with $\pi^{K_{F}} \neq 0$, to understand (\ref{test_fcn0}) we need to compute the scalar
\begin{equation} \label{parah_1st}
{\rm tr}(\varphi_\pi(\Phi_F), (V^{E_0}_{-\mu})^{\varphi_\pi(I_F)}).
\end{equation}
There is an unramified character $\chi$ of $T(F)$ such that $\pi$ is a subquotient of $i^G_B(\chi)$, and we may assume $\varphi_\pi|_{W_F} = \varphi_\chi|_{W_F}$. Since $\chi$ is unramified, $\varphi_\chi(I_F) = 1 \rtimes I_F \subset \widehat{T} \rtimes W_F$. Regarding $\chi$ as an element of $\widehat{T}$, (\ref{Ldual_norm}) implies that we may write $\varphi_\chi(\Phi_F) = \chi \rtimes \Phi_F \in \widehat{T} \rtimes W_F$. Then we need to compute 
\begin{equation} \label{qs_to_compute}
{\rm tr}(\chi \rtimes \Phi_F \, , \, (V^{E_0}_{-\mu})^{1 \rtimes I_F}).
\end{equation} 

\subsection{Unramified groups and the Kottwitz conjecture}

Let us consider the case where $G_{\mathbb Q_{p^r}}$ is unramified. Since we are assuming $G$ splits over an unramified extension of $\mathbb Q_p$, it follows that $E/\mathbb Q_p$ is unramified, i.e. $E = E_0$ and $V^{E_0}_{-\mu} = V_{-\mu}$. Moreover $F = E_{j0}$ contains $E$ with degree $j$.

Further, since $G$ splits over $F^{\rm un}$,  we have $V_{-\mu}^{1\rtimes I_F} = V_{-\mu}$. So we are reduced to computing ${\rm tr}(\chi \rtimes \Phi_F \, , \, V_{-\mu})$.  Exactly as in Kottwitz' calculation of the Satake transform in \cite[p.~295]{Ko84b}, we see that (\ref{parah_1st}) is
\begin{equation}
{\rm tr}(\chi \rtimes \Phi_F \, , \, V_{-\mu}) = \sum_{\lambda \in W(F)\cdot \mu} (-\lambda)(\chi).
\end{equation}
Here $W(F) = W(G,A)$ is the relative Weyl group for $G/F$, and we view $\lambda \in X_*(A) = X_*(T)^{\Phi_F}$ as a character on $\widehat{T}$. This proves the following result.

\begin{lemma} In the above situation, 
\begin{equation} \label{Z_V_vs_Bernstein}
Z_{V_{-\mu, j}^{E_{j0}}} * 1_{K_{p^r}} = z_{-\mu,j},
\end{equation}
where the {\em Bernstein function} $z_{-\mu,j}$ (cf. Definition \ref{Bern_fcn_def}) is the unique element of $\mathcal Z(G(F), K_F)$ which acts (on the left) on the normalized induced representation $i^G_B(\chi)^{K_F}$ by the scalar $\sum_{\lambda \in W(F)\cdot \mu} (-\lambda)(\chi)$, for any unramified character $\chi: T(F) \rightarrow \mathbb C^\times$.
\end{lemma}
Of course the advantage of $z_{-\mu, j}$ is that unlike the left hand side of (\ref{Z_V_vs_Bernstein}), it is defined unconditionally. A relatively self-contained, elementary, and efficient approach to Bernstein functions is given in $\S\ref{appendix_sec}$. 

Thus Conjecture \ref{TFC} in this situation is equivalent to the {\em Kottwitz Conjecture}. 

\begin{conj} {\rm (Kottwitz conjecture)} \label{Kottwitz_conj}
In the situation where ${\bf G}_{\mathbb Q_{p^r}}$ is unramified and $K_p$ is a parahoric subgroup, the function $\phi_r$ in \textup{(}\ref{Lef^ss}\textup{)} may be taken to be $p^{rd/2} z_{-\mu, j}$.
\end{conj}

Conjecture \ref{Kottwitz_conj} was formulated by Kottwitz in 1998, about 11 years earlier than Conjecture \ref{TFC}. There is a closely related conjecture of Kottwitz concerning nearby cycles on Rapoport-Zink local models ${\bf M}^{\rm loc}_{K_p}$ for $Sh_{K_p}$. We refer to \cite{RZ, Ra05} for definitions of local models, and to \cite{H05, HN02a} for further details about the following conjecture in various special cases.

\begin{conj} {\rm (Kottwitz Conjecture for Nearby Cycles)} \label{Kottwitz_conj_RPsi}
Write $\mathcal G$ for the Bruhat-Tits parahoric group scheme over $\mathbb Z_{p^r}$ with generic fiber $G_{\mathbb Q_{p^r}}$ and with $\mathcal G(\mathbb Z_{p^r}) = K_{p^r}$. Let $\mathcal G_t$ denote the analogous parahoric group scheme over $\mathbb F_{p^r}[[t]]$ with the ``same'' special fiber as $\mathcal G$. Then there is an $L^+\mathcal G_{t, \mathbb F_{p^r}}$-equivariant embedding of ${\bf M}^{\rm loc}_{K_p, \mathbb F_{p^r}}$ into the affine flag variety $L\mathcal G_{t,\mathbb F_{p^r}} / L^+ \mathcal G_{t, \mathbb F_{p^r}}$, via which we can identify the semisimple trace of Frobenius function $x \mapsto {\rm tr}^{\rm ss}({\rm Fr}_{p^r}, R\Psi^{{\bf M}^{\rm loc}_{K_p}}_x)$ on $x \in {\bf M}^{\rm loc}_{K_p}(\mathbb F_{p^r})$ with the function $p^{dr/2}z_{-\mu,j} \in \mathcal Z(\mathcal G_t(\mathbb F_{p^r}(\!(t)\!))\, , \, \mathcal G_t(\mathbb F_{p^r}[\![t]\!]))$.
\end{conj}

\subsection{The quasisplit case}

The group $\widehat{G}^{I_F}$ is a possibly disconnected reductive group, with maximal torus $(\widehat{T}^{I_F})^\circ$ (see the proof of Theorem 8.2 of \cite{St}). Now we may restrict the representation $V^{E_0}_{-\mu}$ to the subgroup $\widehat{G}^{I_F} \rtimes W_F \subset \widehat{G} \rtimes W_F$. Let $\chi$ be a weakly unramified character of $T(F)$; by (\ref{weakly_unram_chars}) we can view $\chi \in (\widehat{T}^{I_F})_{\Phi_F}$. The only $\widehat{T}^{I_F}$-weight spaces of $(V^{E_0}_{-\mu})^{1\rtimes I_F}$ which contribute to (\ref{qs_to_compute}) are indexed by the $\Phi_F$-fixed weights, i.e.~by those in $X^*(\widehat{T}^{I_F})^{\Phi_F}$.  (It is important to note that it is the weight spaces for the diagonalizable group $\widehat{T}^{I_F}$, and not for the maximal torus $(\widehat{T}^{I_F})^\circ$, which come in here.) This is consistent with Theorem \ref{my_J_thm} of the Appendix, and may be expressed as follows.

\begin{prop} \label{conj_qs_case} In the general quasisplit situation, $Z_{V^{E_{j0}}_{-\mu,j}} * 1_{K_{p^r}}$ is the unique function in $\mathcal Z(G(\mathbb Q_{p^r}), K_{p^r})$ which acts on the left on each weakly unramified principal series representation $i^G_B(\chi)^{K_F}$ by the scalar  \textup{(}\ref{qs_to_compute}\textup{)}, and thus is a certain linear combination of Bernstein functions $z_{-\lambda, j}$ where $-\lambda \in X^*(\widehat{T}^{I_F})^{\Phi_F}$ ranges over the $W(G,A)$-orbits of $\Phi_F$-fixed $\widehat{T}^{I_F}$-weights in $V^{E_0}_{-\mu}$.
\end{prop}

It is an interesting exercise to write out the linear combinations of Bernstein functions explicitly in each given case. Once this is done, the result can be used to find explicit descriptions of test functions for inner forms of quasi-split groups. This is the subject of the next subsection.

\subsection{Passing from quasisplit to general cases via transfer homomorphisms}

\subsubsection{Test function conjecture via transfer homomorphisms}

We use freely the notation and set-up explained in the Appendix $\S \ref{App:transfer_hom_sec}$.  Let $G^*$ be a quasi-split $F$-group with an inner twisting $\psi: G \rightarrow G^*$. Let $J^* \subset G^*(F)$ resp.~$J \subset G(F)$ be parahoric subgroups and consider the {\em normalized transfer homomorphism} $\tilde{t}: \mathcal Z(G^*(F), J^*) \rightarrow \mathcal Z(G(F), J)$ from Definition \ref{tilde_t_defn}.

The following conjecture indicates that test functions for the quasisplit group $G^*$ should determine test functions for $G$. This is compatible with the global considerations  which led to Theorem \ref{Z_fcn_thm}. 

\begin{conj} \label{transfer_conj} Let $K_{p^r}$ resp.~$K^*_{p^r}$ be parahoric subgroups of $G(\mathbb Q_{p^r})$ resp.~$G^*(\mathbb Q_{p^r})$, with corresponding normalized transfer homomorphism $\tilde{t}: \mathcal Z(G^*(\mathbb Q_{p^r}), K^*_{p^r}) \rightarrow \mathcal Z(G(\mathbb Q_{p^r}), K_{p^r})$. If $\phi_r^* \in \mathcal Z(G^*(\mathbb Q_{p^r}), K^*_{p^r})$ is the function $p^{rd/2}(Z^{G^*}_{V^{E_{j0}}_{-\mu,j}} * 1_{K^*_{p^r}})$ described in Proposition \ref{conj_qs_case} for the data $(G^*_{\mathbb Q_{p^r}}, \{-\mu\}, K^*_{p^r})$, then $\phi_r := \tilde{t}(\phi^*_r) \in \mathcal Z(G(\mathbb Q_{p^r}), K_{p^r})$ is a test function satisfying \textup{(}\ref{Lef^ss}\textup{)} for the original data $(G_{\mathbb Q_{p^r}}, \{-\mu\}, K_{p^r})$.
\end{conj}  

Assuming Conjecture \ref{TFC} holds, another way to formulate this is that the normalized transfer homomorphism $\tilde{t}$ takes the function $Z^{G^*}_{V^{E_{j0}}_{-\mu,j}} * 1_{K^*_{p^r} }\in \mathcal Z(G^*(\mathbb Q_{p^r}), K^*_{p^r})$ to the function $Z^G_{V^{E_{j0}}_{-\mu,j}} * 1_{K_{p^r}} \in \mathcal Z(G(\mathbb Q_{p^r}), K_{p^r})$. But the point of Conjecture \ref{transfer_conj} is to provide an explicit test function for the non-quasisplit data $(G_{\mathbb Q_{p^r}}, \{-\mu \}, K_{p^r})$ which can be compared with direct geometric calculations of the nearby cycles attached to this data, and thus to provide a method to prove an unconditional analogue of Conjecture \ref{TFC} for such data. This is illustrated in $\S\ref{comp_w_Rap}$ below.

The next two paragraphs show that Conjecture \ref{transfer_conj} is indeed reasonable.

\subsubsection{A calculation for ${\rm GL}_2$}

Take $G^* = {\rm GL}_{2,F}$ and $G = D^\times$, where $D$ is the central simple division algebra over $F$ of dimension $4$.

Here we will explicitly calculate and compare the test functions associated to $({\rm GL}_{2, F}, \{-\mu \}, I_F)$ and $(D^\times, \{-\mu\}, \mathcal O_{D}^\times)$, where $\mu = (1,0)$, and where $I_F \subset {\rm GL}_2(F)$ and $\mathcal O_{D}^\times \subset D^\times$ are the standard Iwahori subgroups.  This calculation will show that the normalized transfer homomorphism takes one test function to the other. This is required in order for both Conjectures \ref{TFC} and \ref{transfer_conj} to hold true.

Write $z^*_{-\mu} = Z^{\rm GL_2}_{V_{-\mu}} * 1_{I_F} \in \mathcal Z({\rm GL}_2(F), I_F)$ and $z_{-\mu} = Z^{D^\times}_{V_{-\mu}} * 1_{\mathcal O_{D}^\times} \in \mathcal H(D^\times, \mathcal O_{D}^\times) \cong \mathbb C[\mathbb Z]$.  The last isomorphism is induced by the Kottwitz homomorphism, which in this case is the normalized valuation ${\rm val_F} \circ {\rm Nrd}_D: D^\times \twoheadrightarrow \mathbb Z$, where ${\rm val}_F$ is the normalized valuation for $F$ and ${\rm Nrd}_D: D \rightarrow F$ is the reduced norm.

Write $\bar{\mu} = (0,1)$ and let $B^*$ denote the Borel subgroup of {\em lower triangular} matrices in ${\rm GL}_2$. Then $z^*_{-\mu}$ acts on the left on $i^{\rm GL_2}_{B^*}(\chi)^{I_F}$ by the scalar $${\rm tr}(\chi \rtimes \Phi_F, V_{-\mu}) = (-\mu -\bar{\mu})(\chi),$$ for any unramified character $\chi \in {\rm Hom}(T^*(F)/T^*(F)_1, \mathbb C^\times)$. We may view $\chi$ as a diagonal $2 \times 2$ complex matrix $\chi = {\rm diag}(\chi_1, \chi_2)$.

To calculate $z_{-\mu}$ we need a few preliminary remarks. First we parametrize unramified characters $\eta \in {\rm Hom}(D^\times/\mathcal O^\times_{D}, \mathbb C^\times)$ by writing $\eta = \eta_0 \circ {\rm Nrd}_D$, where $\eta_0 \in \mathbb C^\times$ is viewed as the unramified character on $F^\times$ which sends $\varpi_F \mapsto \eta_0$.  The map ${\rm Nrd}_D: D^\times \rightarrow F^\times$ is Langlands dual to the diagonal embedding $\mathbb G_m(\mathbb C) \rightarrow {\rm GL}_2(\mathbb C)$, and it follows that the cocycles $z_\eta$ and $z_{\eta_0} $ attached by Langlands duality to the quasicharacters $\eta$ and $\eta_0$ satisfy
$$
\begin{array}{ccc}
z_\eta = \begin{bmatrix} z_{\eta_0} & 0 \\ 0 & z_{\eta_0} \end{bmatrix}, & \mbox{and thus,} &
z_\eta(\Phi_F) = \begin{bmatrix} \eta_0 & 0 \\ 0 & \eta_0 \end{bmatrix}.
\end{array}
$$
On the other hand, if $\mathbf 1$ denotes the trivial 1-dimensional representation of $D^\times$, then its Langlands parameter $\varphi^{D^\times}_{\mathbf 1}$ satisfies $\varphi^{D^\times}_{\mathbf 1}(\Phi_F) = {\rm diag}(q^{-1/2}, q^{1/2})$ (see, e.g.,~\cite[Thm.~4.4]{PrRa}). So using (\ref{action}), we obtain
\begin{align*}
{\rm tr}(\varphi^{D^\times}_\eta(\Phi_F), V_{-\mu}) &= {\rm tr}\Big(\begin{bmatrix} \eta_0 q^{-1/2} & 0 \\ 0 & \eta_0 q^{1/2} \end{bmatrix}, V_{-\mu}\Big) \\
&= (\delta^{-1/2}_{B^*}(\varpi^{-\mu}) \cdot -\mu|_{\widehat{Z}} ~+ ~\delta^{-1/2}_{B^*}(\varpi^{-\bar{\mu}})\cdot -\bar{\mu}|_{\widehat{Z}} )\Big( \begin{bmatrix} \eta_0 & 0 \\ 0 & \eta_0 \end{bmatrix} \Big).
\end{align*}
Here $\widehat{Z}$ is the center of $\widehat{G^*}$. Using the definition of $\tilde{t}$ we deduce the following result.

\begin{prop} The normalized transfer homomorphism $\tilde{t} : \mathcal Z({\rm GL}_{2}(F), I_F) \rightarrow \mathcal H(D^\times, \mathcal O^\times_D)$ sends $z^*_{-\mu}$ to $z_{-\mu}$.
\end{prop}

\subsubsection{Compatibility with nearby cycles in some anisotropic cases} \label{comp_w_Rap}

Suppose we are in a situation where $E = \mathbb Q_p$. As before, write $F = \mathbb Q_{p^r}$. Suppose $G_F = (D\otimes F)^\times \times \mathbb G_m$, where $D$ is a central division algebra over $E$ of degree $n^2$, for $n > 2$. This situation arises in the setting of ``fake unitary'' simple Shimura varieties (see, e.g.~\cite[$\S5$]{H01}). Let $G^* = {\rm GL}_n \times \mathbb G_m$, a split inner form of $G$ over $\mathbb Q_p$.

Suppose that $V_{-\mu} = \wedge^m(\mathbb C^n)$ for $0 < m < n$, i.e. the representation of $\widehat{G^*} = {\rm GL}_n(\mathbb C) \times \mathbb C^\times$ where the first factor acts via the irreducible representation with highest weight $(1^m,0^{n-m})$ and the second factor acts via scalars.  

Consider the local models ${\bf M}^{* \rm loc} = {\bf M}^{\rm loc}(G^*, \{-\mu\}, K^*_p)$ and ${\bf M}^{\rm loc} = {\bf M}^{\rm loc}(G, \{-\mu\}, K_p)$, where $K^*_p \subset G^*(F)$ and $K_p \subset G(F)$ are Iwahori subgroups. We can choose the inner twist $G \rightarrow G^*$ and the subgroups $K^*_p$ and $K_p$ so that 
$$
{\bf M}^{* \rm loc} (\overline{\mathbb F}_p) =  {\bf M}^{\rm loc} (\overline{\mathbb F}_p)
$$
and where the action of geometric Frobenius $\Phi_p$ on the right hand side is given by $\Phi_p = {\rm Ad}(c_{\Phi_p}) \cdot \Phi^*_p$ where $\Phi^*_p$ is the usual Frobenius action (on the left hand side) and where $\tau \mapsto {\rm Ad}(c_\tau)$ represents the class in $H^1(\mathbb Q_p, {\rm PGL}_n)$ corresponding to the inner twist $G \rightarrow G^*$.

Assume $(r,n) =1$ and set $q = p^r$. Then ${\bf M}^{\rm loc}(G, \{-\mu\}, K_p)(\mathbb F_q)$ consists of a single point.  To understand the corresponding test function we may ignore the ${\mathbb G}_m$-factor and pretend that $G = D^\times$ and $G^* = {\rm GL}_n$. Then the Kottwitz homomorphism $\kappa_G: G(F) \rightarrow \mathbb Z$ induces an isomorphism
$$
\mathcal H(G(\mathbb Q_{p^r}), K_{p^r}) \cong \mathbb C[\mathbb Z].
$$
The test function for the Shimura variety giving rise to the local Shimura data $(G, \{-\mu \}, K_p)$ should be calculated by understanding the function trace of Frobenius on nearby cycles on ${\bf M}^{\rm loc}$, similarly to Conjecture \ref{Kottwitz_conj_RPsi} in the unramified case. The test function should be of the form $C_q\cdot 1_m \in \mathbb C[\mathbb Z] = \mathcal H(G(\mathbb Q_p), K_{p^r})$ for some scalar $C_q$.

\begin{prop} \label{anistopic_eg_prop}
In the above situation, Conjecture \ref{transfer_conj} predicts that the coefficient $C_q$ is given by $C_q = \# {\rm Gr}(m,n)(\mathbb F_q)$, the number of $\mathbb F_q$-rational points on the Grassmannian variety ${\rm Gr}(m,n)$ parametrizing $m$-planes in $n$-space. 
\end{prop}
This is compatible with calculations of Rapoport of the trace of Frobenius on nearby cycles of the local models for such situations, see \cite{Ra90}. Thus the normalized transfer homomorphism gives a 
group-theoretic framework with which we could make further predictions about nearby cycles on the local models attached to non-quasiplit groups $G$, assuming we know explicitly the corresponding test function for a quasisplit inner form of $G$.

\begin{proof}
By the final sentence of Proposition \ref{t_const_term_prop}, we simply need to integrate the function $p^{rd/2} z^*_{-\mu} \in \mathcal Z({\rm GL}_n(\mathbb Q_{p^r}), K^*_{p^r})$ over the fiber of the Kottwitz homomorphism ${\rm val} \circ {\rm det}$ over $1_m \in \mathbb C[\mathbb Z]$.  This is a combinatorial problem which could be solved since we know $p^{rd/2}z^*_{-\mu}$ explicitly. However, it is easier to use geometry. Translating ``integration over the fiber of the Kottwitz homomorphism'' in terms of local models gives us the equality
$$
C_q \ = \sum_{x \in {\bf M}^{*\rm loc}(\mathbb F_q)} {\rm Tr}(\Phi_p^r, R\Psi^{\bf M^{*\rm loc}}(\mathbb Q_\ell)_x).
$$
(Here $\ell$ is a rational prime with $\ell \neq p$.) But the special fiber of ${\bf M}^{*\rm loc}$ embeds into the affine flag variety ${\rm Fl}_{{\rm GL}_n}$ for ${\rm GL}_n/\mathbb F_p$, and under the projection  $p: {\rm Fl}_{{\rm GL}_n} \rightarrow {\rm Gr}_{{\rm GL}_n}$ to the affine Grassmannian, ${\bf M}^{*\rm loc}$ maps onto ${\rm Gr}(m,n)$ and $Rp_*(R\Psi^{{\bf M}^{* \rm loc}}(\mathbb Q_\ell)) = \mathbb Q_\ell$, the constant $\ell$-adic sheaf on ${\rm Gr}(m,n)$ in degree $0$. Thus we obtain
$$
C_q = \sum_{x \in {\rm Gr}(m,n)(\mathbb F_q)} {\rm Tr}(\Phi_p^r, (\mathbb Q_\ell)_x) = \#{\rm Gr}(m,n)(\mathbb F_q)
$$
as desired. (The reader should note the similarity with Prop.~3.17 in \cite{Ra90}, which is justified in a slightly different way.)
\end{proof}

\section{Overview of evidence for the test function conjecture} \label{evidence_TFC_sec}

\subsection{Good reduction cases}

In case ${\bf G}_{\mathbb Q_p}$ is unramified and $K_p$ is a hyperspecial maximal compact subgroup of ${\bf G}(\mathbb Q_p)$, we expect $Sh({\bf G}, h^{-1}, K^pK_p)$ to have good reduction over $\mathcal O_{{\bf E}_\mathfrak p}$. In PEL cases this was proved by Kottwitz \cite{Ko92a}. In the same paper for PEL cases of type $A$ or $C$, it is proved that the function $\phi_r = 1_{K_{p^r}\mu(p^{-1})K_{p^r}}$ satisfies (\ref{Lef^ss}), which can be viewed as verifying Conjecture \ref{TFC} for these cases.

\subsection{Parahoric cases}

Assume $K_p$ is a parahoric subgroup. We will discuss only PEL Shimura varieties.

Here the approach is via the Rapoport-Zink local model ${\bf M}^{\rm loc}_{K_p}$ for a suitable integral model $\mathcal M_{K_p}$ for $Sh_{K_p}$ and the main ideas are due to Rapoport. We refer to the survey articles \cite{Ra90}, \cite{Ra05}, and \cite{H05} for more about how local models fit in with the Langlands-Kottwitz approach.  For much more about the geometry of local models, we refer the reader to the survey article of Pappas-Rapoport-Smithing \cite{PRS} and the references therein.

Using local models, the first step to proving Conjecture \ref{RPsi_TFC} is to prove Conjecture \ref{Kottwitz_conj_RPsi}. The first evidence was purely computational: in \cite{H01}, $z_{-\mu,j}$ was computed explicitly in the Drinfeld case and the result was compared with Rapoport's computation of the nearby cycles in that setting, proving Conjecture 7.1.3 directly. This result motivated Kottwitz' more general conjecture and also inspired Beilinson and Gaitsgory to construct the center of an affine Hecke algebra via a nearby cycles construction, a feat carried out in \cite{Ga}. Then in \cite{HN02a} Gaitsgory's method was adapted to prove Conjecture \ref{Kottwitz_conj_RPsi} for the split groups ${\rm GL}_n$ and ${\rm GSp}_{2n}$. This in turn was used to demonstrate Conjecture \ref{Kottwitz_conj} for certain special Shimura varieties in \cite{H05}, and then to prove the analogue of Theorem \ref{Z_fcn_thm} for those special Shimura varieties with parahoric level structure at $p$. The harmonic analysis ingredient needed for the latter was provided by \cite{H09}.

In his 2011 PhD thesis, Sean Rostami proved Conjecture \ref{Kottwitz_conj_RPsi} when $G$ is an unramified unitary group. In a recent breakthrough, Pappas and Zhu defined group-theoretic local models ${\bf M}^{\rm loc}_{K_p}$ whenever $G$ splits over a tamely ramified extension, and for unramified groups $G$ proved Conjecture \ref{Kottwitz_conj_RPsi}, see \cite{PZ}, esp.~Theorem 10.16. 

\subsection{Deeper level cases}
We again limit our discussion to PEL situations, where progress to date has occurred.

It is again natural to study directly the nearby cycles relative to a suitable integral model for the Shimura variety and hope that it gives rise to a test distribution in the Bernstein center.  For Shimura varieties in the ``Drinfeld case'' with $K_p$ a pro-p Iwahori subgroup of ${\bf G}(\mathbb Q_p) = {\rm GL}_n(\mathbb Q_p) \times \mathbb Q_p^\times$ (``$\Gamma_1(p)$-level structure at $p$''), one may use Oort-Tate theory to define suitable integral models and prove Conjectures \ref{RPsi_TFC} and \ref{TFC} for them. This was done by the author and Rapoport \cite{HRa} (and \cite{H12} provided the harmonic analysis ingredient needed to go further and prove Theorem \ref{Z_fcn_thm} in this case).

Around the same time as \cite{HRa}, Scholze studied in \cite{Sch1} nearby cycles on suitable integral models for the modular curves with arbitrarily deep full level structure at $p$.  In this way he proved Conjectures \ref{RPsi_TFC} and \ref{TFC} in these cases, taking the compactifications also into account, and thereby proved  the analogue of Theorem \ref{Z_fcn_thm} for the compactified modular curves at nearly all primes of bad reduction. The nearby cycles on his integral models naturally gave rise to some remarkable distributions in the Bernstein center, for which he gave explicit formulae (see $\S \ref{explicit_sec}$).

Then in \cite{Sch2} Scholze generalized the approach of \cite{Sch1} to compact Shimura varieties in the Drinfeld case, again finding an explicit description of nearby cycles. In this case, he was still able to produce a test function to plug into (\ref{Lef^ss}), or rather, simultaneously incorporating the base-change transfer results he needed in precisely this case, he found a test function that goes directly into the pseudostabilization of (\ref{Lef^ss}). This allowed him to prove Theorem \ref{Z_fcn_thm}.  In contrast to the modular curve situation, in higher rank the nearby cycles on Scholze's integral models do not directly produce distributions in the Bernstein center, and an explicit description of his test functions seems hopeless. But nevertheless Scholze was able to prove by indirect means Conjecture \ref{TFC} in this case.

The description of the nearby cycles in \cite{Sch2} provided one ingredient for Scholze's subsequent paper \cite{Sch3} which gave a new and streamlined proof of the local Langlands conjecture for general linear groups.

In later work Scholze \cite{Sch4} formalized his method of producing test functions in many cases, using deformation spaces of $p$-divisible groups, and this is used to give a nearly complete description of the cohomology groups of many compact unitary Shimura varieties in his joint work \cite{SS} with S.W.~Shin; their main assumption at $p$ is that ${\bf G}_{\mathbb Q_p}$ is a product of Weil restrictions of general linear groups. The advantage of what we could call the Langlands-Kottwitz-Scholze approach in this situation is that it yields in \cite{SS} a new construction of the Galois representations constructed earlier by Shin \cite{Sh}, in a shorter way that avoids Igusa varieties.

In these later developments, Conjecture \ref{TFC} does not play a central part, but the stable Bernstein center does nevertheless still play a clarifying role in the pseudostabilization process (e.g.~in \cite{SS}).  It seems that only certain integral models, such as those we see in many parahoric or pro-p Iwahori level cases, have the favorable property that their nearby cycles naturally give rise to distributions in the Bernstein center.  It remains an interesting problem to find such integral models in more cases, and to better understand the role of the Bernstein center in the study of Shimura varieties.

\section{Evidence for conjectures on transfer of the Bernstein center} \label{evidence_transfer_sec}

Here we present some evidence for the general principle that the (stable/geometric) Bernstein center is particularly well-behaved with respect to  (twisted) endoscopic transfer.  The primary evidence thus far consists of some unconditional analogues of Conjecture \ref{BC_conj}.

Let $G/F$ be an unramified group, and let $F_r/F$ be the degree $r$ unramified extension of $F$ in some algebraic closure of $F$.  In \cite{H09, H12}, the author defined {\em base change homomorphisms}
$$
b_r: \mathcal Z(G(F_r), J_r) \rightarrow \mathcal Z(G(F), J),
$$
where $J \subset G(F)$ is either a parahoric subgroup or a pro-p Iwahori subgroup, and where $J_r$ is the corresponding subgroup of $G(F_r)$.  Then we have ``base-change fundamental lemmas'' of the following form.\footnote{Relating to pro-p Iwahori level, a much stronger result is proved in \cite{H12} concerning the base change transfer of Bernstein centers of Bernstein blocks for depth-zero principal series representations.}

\begin{theorem} \label{uncond_BC_conj}
For any $\phi_r \in \mathcal Z(G(F_r), J_r)$, the function $b_r(\phi_r)$ is a base-change transfer of $\phi_r$ in the sense of \textup{(}\ref{stable_BC_defn}\textup{)}.
\end{theorem}

By Kottwitz \cite{Ko86}, the function $1_J$ is a base-change transfer of $1_{J_r}$. Hence for any $V_r \in {\rm Rep}(\,^L(G_{F_r}))$, Conjecture \ref{BC_conj} predicts that $b_{F_r/F}(Z_{V_r}) * 1_J$ is a base-change transfer of $Z_{V_r} * 1_{J_r}$. This is a consequence of Theorem \ref{uncond_BC_conj}, because of the following compatibility between the base-change operations in \cite{H09, H12} and in the context of stable Bernstein centers (cf.~Prop.~\ref{change_of_field}).

\begin{lemma} \label{BC_comparison}
In the above situations, $b_r(Z_{V_r} * 1_{J_r}) = b_{F_r/F}(Z_{V_r}) * 1_J$.
\end{lemma} 

\begin{proof}
First assume $J$ is a parahoric subgroup.  Let $\chi$ be any unramified character of $T(F)$. It is enough to show that the two functions act on the left by the same scalar on every unramified principal series representation $i^G_B(\chi)^J$.

Let $N_r: T(F_r) \rightarrow T(F)$ be the norm homomorphism. By the definition of $b_r$ in \cite{H09}, $b_r(Z_{V_r} * 1_{J_r})$ acts by the scalar  by which $Z_{V_r} * 1_{J_r} $ acts on $i^{G_r}_{B_r}(\chi \circ N_r)^{J_r}$. This is the scalar by which $Z_{V_r}$ acts on $i^{G_r}_{B_r}(\chi \circ N_r)$, which in view of ${\rm LLC+}$ is 
\begin{equation} \label{phi_chi_eq}
{\rm tr}^{\rm ss}(\varphi^{T_r}_{\chi \circ N_r}(\Phi_F^r), V_r) = {\rm tr}^{\rm ss}(\varphi^T_{\chi}(\Phi_F^r), V_r).
\end{equation}
But the right hand side is the scalar by which $b_{F_r/F}(Z_{V_r}) * 1_J$ acts on $i^G_B(\chi)^J$.

The equality $\varphi^{T_r}_{\chi \circ N_r}(\Phi_F^r) = \varphi^T_\chi(\Phi_F^r)$ we used in (\ref{phi_chi_eq}) follows from the commutativity of the diagram of Langlands dualities for tori
$$
\xymatrix{
{\rm Hom}_{\rm conts}(T(F), \mathbb C^\times) \ar[r] \ar[d]^{N_r} & H^1_{\rm conts}(W_F, \, ^LT) \ar[d]_{\rm Res} \\
{\rm Hom}_{\rm conts}(T(F_r), \mathbb C^\times) \ar[r] & H^1_{\rm conts}(W_{F_r}, \, ^LT_r)}
$$
which was proved in \cite[Lemma~8.1.3]{KV}. 

Now suppose $J = I^+$ is a pro-p Iwahori subgroup. Then the same argument works given the following fact: for any depth-zero character $\chi: T(F)_1 \rightarrow \mathbb C^\times$ and any extension of it to a character $\tilde{\chi}$ on $T(F)$, and any $z_r \in \mathcal Z(G(F_r), I_r^+)$, the function $b_r(z_r)$ acts on $i^G_B(\tilde{\chi})^{I^+}$ by the scalar by which $z_r$ acts on $i^{G_r}_{B_r}(\tilde{\chi} \circ N_r)^{I^+_r}$. This follows from the definition of $b_r$ given in Definition 10.0.3 of \cite{H12}, using \cite[Lemma 4.2.1]{H12}.
\end{proof}

Let us also mention again Scholze's Theorem C in \cite{Sch2}, which essentially proves Conjecture \ref{BC_conj} for ${\rm GL}_n$ (see Proposition \ref{BC_GLn}).

\section{Explicit computation of the test functions} \label{explicit_sec}

\subsection{Parahoric cases}

Conjecture \ref{TFC} implies that test functions are compatible with change of level. Therefore for the purposes of computing them for parahoric level, the key case is where $K_p$ is an Iwahori subgroup.  Thus, for the rest of this subsection we consider only Iwahori level structure. Since test functions attached to quasisplit groups should determine, in a computable way, those for inner forms (by Conjecture \ref{transfer_conj} and Proposition \ref{t_const_term_prop}), it is enough to understand  quasisplit groups. Via Proposition \ref{conj_qs_case} this boils down to giving explicit descriptions of the Bernstein functions $z_{-\lambda, j}$, assuming we have already expressed the test function explicitly in term of these -- this is automatic for unramified groups using the Kottwitz Conjecture (Conjecture \ref{Kottwitz_conj}).

Let us therefore consider the problem of explicitly computing Bernstein functions $z_\mu$ attached to any group $G/F$ and an Iwahori subgroup $I \subset G(F)$ ($F$ being any local non-archimedean field). For simplicity consider the case where $G/F$ is unramified, and regard $\mu$ as a dominant coweight in $X_*(A)$. The $\mu$ which arise in Conjecture \ref{Kottwitz_conj} are {\em minuscule}; however, we consider $\mu$ which are {\bf not necessarily} minuscule here. Let $\widetilde{W}$ denote the extended affine Weyl group of $G$ over $F$ (cf. $\S\ref{appendix_sec}$). Attached to $\mu$ is an the $\mu$-{\em admissible set} ${\rm Adm}(\mu) \subset \widetilde{W}$, defined by
$$
{\rm Adm}(\mu) = \{ x \in \widetilde{W} ~ | ~ x \leq t_\lambda, \,\, \mbox{for some} \,\, \lambda \in W(G,A)\cdot \mu \},
$$
where $\leq$ denotes the Bruhat order on $\widetilde{W}$ determined by the Iwahori subgroup $I$ and where $t_\lambda$ denotes the translation element in $\widetilde{W}$ corresponding to $\lambda \in X_*(A)$. The $\mu$-admissible set has been studied for its relation to the stratification by Iwahori-orbits in the local model ${\bf M}^{\rm loc}_{K_p}$; for much information see \cite{KR}, \cite{HN02b}, \cite{Ra05}. The strongest combinatorial results relating local models and ${\rm Adm}(\mu)$ are due to Brian Smithling, see e.g.~\cite{Sm1, Sm2, Sm3}. 

For our purposes, the set ${\rm Adm}(\mu)$ enters because it is the set indexing the double cosets in the support of $z_\mu$.

\begin{prop}
The support of $z_\mu$ is the union $\bigcup_{x \in {\rm Adm}(\mu)} IxI$.
\end{prop}

\begin{proof}
This was proved using the theory of alcove walks as elaborated by G\"{o}rtz \cite{G}, in the Appendix to \cite{HRa}. It applies to affine Hecke algebras with arbitrary parameters, hence the corresponding result holds for arbitrary groups, not just unramified groups.
\end{proof}

The following explicit formula was proved in \cite{H01} and in \cite{HP}. Let $T_x = 1_{IxI}$ for $x \in \widetilde{W}$. In the formulas here and below, $q = p^r$ is the cardinality of the residue field of $F$.

\begin{prop}
Assume $\mu$ is minuscule.  Assume the parameters for the Iwahori Hecke algebra are all equal. Then
$$
q^{\ell(t_\mu)/2} z_\mu = (-1)^{\ell(t_\mu)} \sum_{x \in {\rm Adm}(\mu)} (-1)^{\ell(x)} R_{x, t_{\lambda(x)}}(q) \, T_x,
$$
where $x$ decomposes as $x = t_{\lambda(x)} w \in X_*(A) \rtimes W = \widetilde{W}$ and where $R_{x,y}(q)$ denotes the $R$-polynomial of Kazhdan-Lusztig \cite{KL}, and $\ell$ the length function, for the quasi-Coxeter group $\widetilde{W}$.  A similar formula holds in the context of affine Hecke algebras with arbitrary parameters. In the Drinfeld case $G = {\rm GL}_n$ and $\mu = (1,0^{n-1})$, the coefficient of $T_x$ is $(1-q)^{\ell(t_\mu)-\ell(x)}$.  
\end{prop}

There are also explicit formulas for Bernstein functions $z_\mu$ when $\mu$ is not minuscule, but they tend to be much more complicated.  For related computations see \cite{HP} and \cite{GH}.

\subsection{A Pro-p Iwahori level case}

In the Drinfeld case where ${\bf G}_{\mathbb Q_p} = {\rm GL}_n \times {\mathbb G}_m$ and $\mu = (1,0^{n-1}) \times 1$, and where $K_p$ is a pro-p Iwahori subgroup, an explicit formula for the test function $\phi_r$ for $Sh({\bf G}, h^{-1}, K^pK_p)$ was found by the author and Rapoport.  We shall rephrase this slightly by ignoring the ${\mathbb G}_m$ factor and giving the formula for $G = {\rm GL}_n$.

\begin{prop} \textup{(}\cite[Prop.~12.2.1]{HRa}\textup{)} Let $q = p^r$. Let $I_r^+$ denote the standard pro-p Iwahori subgroup of $G_r := {\rm GL}_n(\mathbb Q_{p^r})$. Let $T$ denote the standard diagonal torus in ${\rm GL}_n$. In terms of natural embeddings $T(\mathbb F_q) \hookrightarrow G_r$, and $w \in \widetilde{W} \hookrightarrow G_r$ giving elements $tw^{-1} \in G_r$ representing $I^+_r \backslash G_r /I^+_r$, we have 
\begin{equation}
\phi_r(I^+_r \, t w^{-1} \, I^+_r) = \begin{cases} 0,\,\,\,\,\, \mbox{if $w \notin {\rm Adm}(\mu)$} \\
0, \,\,\,\,\, \mbox{if $w \in {\rm Adm}(\mu)$ but $N_r(t) \notin T_{S(w)}(\mathbb F_p)$} \\ 
(-1)^n \, (p-1)^{n-|S(w)|} \, (1-q)^{|S(w)|-n-1},  \,\,\, \mbox{otherwise}. \end{cases}
\end{equation}
Here $S(w)$ is the set of {\em critical indices} for $w$, equivalently $S(w)$ is the set of standard basis vectors $e_j \in \mathbb Z^n$ such that $w \leq t_{e_j}$ in the Bruhat order on $\widetilde{W}$ determined by the standard Iwahori subgroup of ${\rm GL}_n$.
\end{prop}

\subsection{Deeper level structures}

Here the known explicit descriptions pertain only to $G = {\rm GL}_2$ and first were proved by Scholze \cite{Sch1}. It remains an interesting question whether one can find explicit descriptions of test functions in higher rank groups with arbitrary level structure: even the Drinfeld case $G = {\rm GL}_n$, $\mu = (1,0^{n-1})$ looks difficult, cf. \cite{Sch2}. 

To state Scholze's result, we need some notation. As usual let $F$ be a nonarchimedean local field with ring of integers $\mathcal O$, uniformizer $\varpi$, and residue field cardinality 
$q$.  Let $B$ denote the $\mathcal O$-subalgebra ${\rm M}_2(\mathcal O)$ of ${\rm M}_2(F)$. For any $j \in \mathbb Z$ set $B_j = \varpi^jB$.  Let $K = B^\times$, the standard maximal compact subgroup of $G = {\rm GL}_2(F)$. For $n \geq 1$, let $K_n = 1 + B_n$; so $K_n$ is a principal congruence subgroup and is a normal subgroup of $K$.

Scholze defines a (compatible) family of functions $\phi_n \in \mathcal H(G,K_n)$ for $n\geq 1$. His definition uses two functions, $\ell : G \rightarrow \mathbb Z \cup \{\infty \}$ and $k: G \rightarrow \mathbb Z$. Let $\ell (g) = {\rm val} \circ {\rm det}(1-g)$.  Let $k(g)$ be the unique integer $k$ such that $g  \in B_k$ and $g \notin B_{k+1}$.  By definition $\phi_n$ is $0$ unless ${\rm val}\circ {\rm det}(g) =1$, ${\rm tr}(g) \in \mathcal O$, and $g \in B_{1-n}$. Assume these conditions, in which case one can check that $1-n \leq k(g) \leq 0$ and $\ell(g) \geq 0$.  Now define
$$
\phi_n(g) := \begin{cases} -1-q, \,\,\,\,\,\,\,\,\,\,\,\,\,\,\,\,\,\,\,\,\,\,\,\,\,\,\,\, \mbox{if ${\rm tr}(g) \in \varpi\mathcal O$}, \\
1-q^{2\ell(g)}, \,\,\,\,\,\,\,\,\, \,\,\,\,\,\,\,\,\,\,\,\,  \mbox{if ${\rm tr}(g) \in \mathcal O^\times$ and $\ell(g) < n + k(g)$}, \\
1 + q^{2(n + k(g))-1}, \,\,\, \mbox{if ${\rm tr}(g) \in \mathcal O^\times$ and $\ell(g) \geq n + k(g)$}.
\end{cases}
$$

\begin{prop} {\rm (Scholze \cite{Sch1})}
For each $n \geq 1$, the function $z_n := \frac{q-1}{[K:K_n]} \cdot \phi_n$ belongs to the center $\mathcal Z({\rm GL}_2(F), K_n)$, and the family $(z_n)_n$ is compatible with change of level and thus defines a distribution in the sense of \textup{(}\ref{Z_dist_defn}\textup{)}. This distribution is $q^{1/2}Z_V$ where $V$ is the standard representation ${\mathbb C}^2$ of the Langlands dual group ${\rm GL}_2(\mathbb C)$.
\end{prop}

In an unpublished work, Kottwitz gave another proof of this proposition and also described the same distribution in terms of a family $(\Phi_n)_n$ of functions $\Phi_n \in \mathcal Z({\rm GL}_2(F), I_n)$ where $I_n$ ranges over the ``barycentric'' Moy-Prasad filtration in the standard Iwahori subgroup $I \subset {\rm GL}_2(F)$.

By a completely different technique, in \cite{Var} Sandeep Varma extended both the results of Scholze and Kottwitz stated above, by describing the distributions attached to $V = {\rm Sym}^r({\mathbb C}^2)$ where $r$ is any odd natural number less than $p$, the residual characteristic of $F$.

\section{Appendix: Bernstein isomorphisms via types} \label{appendix_sec}

\subsection{Statement of Purpose}

Alan Roche proved the following beautiful result in \cite{Roc}, Theorem 1.10.3.1.

\begin{theorem} [Roche] \label{Roche_thm}
Let $e$ be an idempotent in the Hecke algebra $\mathcal H = \mathcal H(G(F))$. View $\mathcal H$ as a smooth $G(F)$-module via the left regular representation, and write $e = \sum_{\mathfrak s \in \mathfrak S} e_\mathfrak s$ according to the Bernstein decomposition $\mathcal H = \bigoplus_{\mathfrak s \in \mathfrak S} \mathcal H_\mathfrak s$.  Let $\mathfrak S_e = \{ \mathfrak s \in \mathfrak S\, |\, e_\mathfrak s \neq 0 \}$, and consider the category $\mathfrak R^{\mathfrak S_e}(G(F)) = \prod_{\mathfrak s \in \mathfrak S_e} \mathfrak R_{\mathfrak s}(G(F))$ and its categorical center $\mathfrak Z^{\mathfrak S_e} = \prod_{\mathfrak s \in \mathfrak S_e} \mathfrak Z_\mathfrak s$.  Let $\mathcal Z(e\mathcal H e)$ denote the center of the algebra $e \mathcal H e$.

Then the map $z \mapsto z(e)$ defines an algebra isomorphism $\mathfrak Z^{\mathfrak S_e} ~ \widetilde{\rightarrow} ~ \mathcal Z(e\mathcal H e)$.
 \end{theorem}

Roche's proof is decidedly non-elementary: besides the material developed in \cite{Roc}, it relies on some deep results of Bernstein cited there, most importantly Bernstein's Second Adjointness Theorem and the construction of an explicit progenerator for each Bernstein block $\mathfrak R_\mathfrak s(G(F))$.

In this chapter we use only the very special case of Roche's theorem where $e = e_J$ for a parahoric subgroup $J \subset G(F)$.  We will explain a more elementary approach to this special case.  It will rely only on the part of Bernstein's theory embodied in Proposition \ref{type_prop} below. Formally, the inputs needed are, first, the existence of Bernstein's categorical decomposition $\mathfrak R(G) = \prod_{\mathfrak s}\mathfrak R_\mathfrak s(G)$, which is proved for instance in \cite[Thm.~1.7.3.1]{Roc} in an elementary way, and, second, the internal structure of the Bernstein block $\mathfrak R_\mathfrak s(G)$ associated to a cuspidal pair $\mathfrak s = [(M(F), \tilde{\chi}]_G$ where $M$ is a minimal $F$-Levi subgroup of $G$ and $\tilde{\chi}$ is a character on $M(F)$ which is trivial on its unique parahoric subgroup. For such components, progenerators can be constructed in an elementary way, without using Bernstein's Second Adjointness Theorem.  In fact in what follows we describe this internal structure using a few straightforward elements of the theory of Bushnell-Kutzko types, all of which are contained in \cite{BK}. 


For $e = e_J$ Roche's theorem gives the identification of the center of the parahoric Hecke algebra, in other words a Bernstein isomorphism for the most general case, where $G/F$ is arbitrary and $J \subset G(F)$ is an arbitrary parahoric subgroup.  However we will provide a proof only  for the crucial case of $J = I$, an Iwahori subgroup of $G(F)$.  The general parahoric case should follow formally from the Iwahori case, following the method of Theorem 3.1.1 of \cite{H09}, provided one is willing to rely on some basic properties of intertwiners for principal series representations (a purely algebraic theory of such intertwiners was detailed for split resp.~unramified groups in \cite{HKP} resp.~\cite{H07}, and the extension to arbitrary groups should be similar to \cite{H07}).

The Iwahori case is approached in a different way by S.~Rostami \cite{Ro}. Rostami's proof yields more information, describing the Iwahori-Matsumoto and Bernstein presentations for the Iwahori-Hecke algebra and deducing the description of its center from its Bernstein presentation.

\subsection{Some notation}

The notation will largely come from \cite{HRo}.  Recall $L = \widehat{F^{\rm un}}$ and let $\sigma \in {\rm Aut}(L/F)$ the Frobenius automorphism, which has fixed field $F$. We use the symbol $\Lambda_G$ as an abbreviation for $X^*(Z(\widehat{G}))_{I_F}^{\sigma}$. Moreover, if $S$ denotes a maximal $L$-split torus in $G$ which is defined over $F$, with centralizer $T = {\rm Cent_G}(S)$, then  $\Omega_G$ will denote the subgroup of $\widetilde{W}^{\rm un} = N_G(S)(L)/T(L)_1$, the extended affine Weyl group for $G/L$, which preserves the alcove ${\bf a}$ in the apartment ${\bf S}$ corresponding to $S$, in the building $\mathcal B(G, L)$ of $G$ over $L$.

As always, we let $I$ be the Iwahori subgroup $I = \mathcal G_{\bf a}^{\circ \natural}(\mathcal O_F)$, which we recall we have chosen to be in good position relative to $A$: the corresponding alcove ${\bf a}^\sigma \subset \mathcal B(G,F)$ is required to belong in the apartment $\bf A$ corresponding to $A$.  

Let $v_F \in \overline{{\bf a}^\sigma}$ be a special vertex with corresponding special maximal parahoric subgroup $K = \mathcal G^{\circ \natural}_{v_F}(\mathcal O_F)$.  Thus $K \supset I$.

Recall $M$ is a minimal $F$-Levi subgroup of $G$.  Further, if $I$ is an Iwahori subgroup of $G(F)$, then $I_M := M(F) \cap I = M(F)_1$ is the corresponding Iwahori subgroup of $M(F)$ 
(cf.~\cite[Lem.~4.1.1]{HRo}).

Use the symbol $1$ to denote the trivial 1-dimensional representation of any group.

\subsection{Preliminary structure theory results}

Several of the results discussed here were proved independently by S.~Rostami and will appear with somewhat different proofs in \cite{Ro}.

\subsubsection{Iwahori-Weyl group over $F$}

The following lemma concerns variations on well-known results, and were first proved by Timo Richarz \cite{Ri}.  

Let $G_1$ denote the subgroup of $G(F)$ generated by all parahoric subgroups of $G(F)$.  By \cite[Lem.~17]{HRa1} and \cite{Ri}, we have $G_1 = G(F)_1$. Let $N_1 = N_G(A)(F) \cap G_1$, and let ${\bf S}$ denote the set of reflections through the walls of ${\bf a}$.  Then by \cite[Prop.~5.2.12]{BT2}, the quadruple $(G_1, I, N_1, {\bf S})$ is a (double) Tits system with affine Weyl group $W_{\rm aff} = N_1/I \cap N_1$, and the inclusion $G_1 \rightarrow G(F)$ is $BN$-adapted\footnote{In \cite{BT2} the symbol $B$ is used in place of $I$.} of connected type. 

\begin{lemma} [T. Richarz \cite{Ri}] \label{Timo}
\begin{enumerate}
 \item[(a)] Let $M_1 = M(F)_1$.  Define the Iwahori-Weyl group as $\widetilde{W} := N_G(A)(F)/M_1$. Then there is an isomorphism $\widetilde{W} = W_{\rm aff} \rtimes \Lambda_G$, which is canonical given the choice of base alcove ${\bf a}$. This gives 
$\widetilde{W}$ the structure of a quasi-Coxeter group.
\item[(b)] If $S \subset G$ is a maximal $L$-split torus which is $F$-rational and contains $A$, and if we set $T := {\rm Cent}_G(S)$ and $\widetilde{W}^{\rm un} := N_G(S)(L)/T(L)_1$, then the natural map $N_G(S)(L)^\sigma \rightarrow N_G(A)(F)$ induces an isomorphism $(\widetilde{W}^{\rm un})^\sigma ~ \widetilde{\rightarrow} ~ \widetilde{W}$.   
\end{enumerate}
\end{lemma}

Thus, in light of (b) we may reformulate the Bruhat-Tits decomposition of \cite[Prop.~8, Rem.~9]{HRa1}, as follows.

\begin{lemma}  The map $N_G(A)(F) \rightarrow G(F)$ induces a bijection
\begin{equation} \label{BT_decomp}
\widetilde{W} \cong I \backslash G(F)/I.
\end{equation}
\end{lemma}
Further, the Bruhat order $\leq$ and length function $\ell$ on $W_{\rm aff}$ extend in the usual way to $\widetilde{W}$, and we have for $w \in \widetilde{W}$ and $s \in W_{\rm aff}$ representing a simple affine reflection, the usual BN-pair relations
\begin{equation} \label{BN_pair}
IsIwI = \begin{cases} IswI, \,\,\, \mbox{if $w < sw$} \\ IwI \cup IswI, \,\,\, \mbox{if $sw < w$}. \end{cases}
\end{equation}

\subsubsection{Iwahori factorization}

Let $P = MN$ be an $F$-rational parabolic subgroup with Levi factor $M$, unipotent radical $N$ and opposite unipotent radical $\overline{N}$. Let$I_H = I \cap H$ for $H = N, \, \overline{N},$ or $M$.

\begin{lemma}  \label{Iwah_fact_lem}
In the above situation, we have the Iwahori factorization
\begin{equation} \label{Iwah_fact}
I = I_N \cdot I_M \cdot I_{\overline{N}}.
\end{equation}
\end{lemma}

\begin{proof}  We use the notation of \cite{BT2}.  By \cite[5.2.4]{BT2} with $\Omega := {\bf a}$, we have
$$
\mathfrak G^\circ_{\bf a}(\mathcal O^\natural) = {\rm U}_{\bf a}^{+\natural} {\rm U}_{\bf a}^{- \natural} {\rm N}^{\circ \natural}_{\bf a},
$$
where ${\rm N}^{\circ \natural}_{\bf a} := N^\natural \cap \mathfrak Z^\circ(\mathcal O^\natural) {\rm U}^\natural_{\bf a}$.  Since $\mathfrak Z^\circ(\mathcal O^\natural) {\rm U}^\natural_{\bf a} \subset \mathfrak G^\circ_{\bf a}(\mathcal O^\natural)$, we have 
$$
{\rm N}^{\circ \natural}_{\bf a} = N^\natural \cap \mathfrak G^\circ_{\bf a}(\mathcal O^\natural) = \mathfrak Z^\circ(\mathcal O^\natural).
$$ 
The key inclusion here, $N^\natural \cap \mathfrak G^\circ_{\bf a}(\mathcal O^\natural) \subseteq \mathfrak Z^\circ(\mathcal O^\natural),$ translates in our notation to $N_G(A)(F) \cap I \subseteq M(F)_1$, which can be deduced from Lemma \ref{Timo}(a).

Translating again back to our notation we get $I = I_N \cdot I_{\overline{N}} \cdot I_M$ 
which is the desired equality since $I_M$ normalizes $I_{\overline{N}}$.
\end{proof}

\subsubsection{On $M(F)^1/M(F)_1$}

\begin{lemma} \label{basic_lem}  The following basic structure theory results hold:
\begin{enumerate}
\item[(a)] In the notation of \cite{HRo}, we have $M(F)^1/M(F)_1 = \tilde{K}/K$, which injects into $G(F)^1/G(F)_1$.    Thus $M(F)_1 = M(F)^1 \cap G(F)_1$.
\item[(b)] The Weyl group $W(G,A)$ acts trivially on $M(F)^1/M(F)_1$.
\item[(c)] Let ${\bf a} \subset \mathcal B(G,L)$ denote the alcove invariant under the group ${\rm Aut}(L/F) \supset \langle \sigma \rangle $ which corresponds to the Iwahori $I \subset G(F)$.  We assume $I \subset K$.  Then the naive Iwahori $\tilde{I} := G(F)^1 \cap {\rm Fix}({\bf a}^\sigma)$ has the following properties
\begin{itemize}
\item $M(F)^1/M(F)_1 = \tilde{I}/I = \tilde{K}/K$.
\item $\tilde{I} = I \cdot M(F)^1$.
\end{itemize}
\end{enumerate}
\end{lemma}

\begin{proof}
Part (a): in the notation of \cite{HRo}, we know that $\Lambda_{M,{\rm tors}} = \tilde{K}/K$ (\cite[Prop.~11.1.4]{HRo}).  Applying this to $G = M$ we get $\Lambda_{M, \rm tors} = M(F)^1/M(F)_1$.  So $M(F)^1/M(F)_1 = \tilde{K}/K$.  By (8.0.1) and Lemma 8.0.1 in \cite{HRo}, the latter injects into $G(F)^1/G(F)_1$.  The final statement follows.

Part (b): By \cite[Lemma 5.0.1]{HRo}, $W(G,A)$ has representatives in $K \cap N_G(A)(F)$.  Thus it is enough to show that if $n \in K \cap N_G(A)(F)$ and $m \in M(F)^1$, then $n m n^{-1} m^{-1} \in M(F)_1$.  This follows from (a), since we clearly have $n m n^{-1} m^{-1} \in M(F)^1 \cap G(F)_1$.

Part (c): First note that $M(F)^1 \subset \tilde{I}$ and $M(F)_1 \subset I$.  Thus there is a commutative diagram
$$
\xymatrix{
\tilde{I}/I \ar[r] & \tilde{K}/K \\
M(F)^1/M(F)_1 \ar[u] \ar[ur] &}
$$
The oblique arrow is bijective by (a).  We claim the horizontal arrow is injective, that is, $\tilde{I} \cap K = I$.  But $\tilde{I} \cap K = G(F)^1 \cap {\rm Fix}({\bf a}^\sigma) \cap G(F)_1 \cap {\rm Fix}(v_F)$, where $v_F$ is the special vertex in $\mathcal B(G_{\rm ad},F)$ corresponding to $K$ (cf. \cite[Lem.~8.0.1]{HRo}).  Thus $\tilde{I} \cap K = G(F)_1 \cap {\rm Fix}({\bf a}^\sigma) = I$ by Remark 8.0.2 of \cite{HRo}.   It now follows that all arrows in the diagram are bijective.   This implies both statements in (c).
\end{proof}



 
\begin{Remark} \label{Iwah_fact_rem}  Let $P = MN$ be as above. We deduce from (c) and (\ref{Iwah_fact}) the Iwahori factorization for $\tilde{I}$
\begin{equation} \label{tilde_Iwah_fact}
\tilde{I} = I_N \cdot M(F)^1 \cdot I_{\overline{N}},
\end{equation}
using the fact that $M(F)^1$ normalizes $I_N$ and $I_{\overline{N}}$.
\end{Remark}

\subsubsection{Iwasawa decomposition}

Next we need to establish a suitable form of the Iwasawa decomposition. 
Let $P = MN$ be as above. 

\begin{lemma}
The inclusion $N_G(A)(F) \hookrightarrow G(F)$ induces bijections
\begin{align}
\widetilde{W} &:= N_G(A)(F)/M(F)_1 ~ \widetilde{\rightarrow} ~ M(F)_1 N(F) \backslash G(F)/I \label{1st} \\
W(G,A) &= N_G(A)(F)/M(F) ~ \widetilde{\rightarrow} ~ P(F)\backslash G(F)/I. \label{2nd}
\end{align}
\end{lemma}

\begin{proof}


In view of the decomposition $\widetilde{W} = \Omega_M^\sigma \rtimes W(G,A)$ (cf.~Lemma 3.0.1(III) of \cite{HRo} plus Lemma \ref{Timo}(b)) and the Kottwitz isomorphism $\Omega_M^\sigma ~ \widetilde{\rightarrow} ~ M(F)/M(F)_1$ (cf. Lemma 3.0.1 of \cite{HRo}), it suffices to prove (\ref{1st}).

For $x \in \mathcal B(G,F)$, let $P_x \subset G(F)$ denote the subgroup fixing $x$.  By \cite{Land},  Proposition 12.9, we have
$$
G(F) = N(F) \cdot N_G(A)(F) \cdot P_x.
$$
For sufficiently generic points $x \in {\bf a}^\sigma$, we have $P_x = \tilde{I}$, which is $M(F)^1 I$ by Lemma  \ref{basic_lem}(c).  Since $M(F)^1 \subset N_G(A)(F)$, we have $G(F) = N(F) \cdot N_G(A)(F) \cdot I$ and the map (\ref{1st}) is surjective.  

To prove injectivity, assume $n_1 = u m_0\cdot  n_2  \cdot j$ for $u \in N(F)$, $m_0 \in M(F)_1$, $n_1, n_2 \in N_G(A)(F)$, and $j \in I$.  There exists $z \in Z(M)(F)$  such that $zuz^{-1} \in I_N$ (cf.~e.g.~\cite[Lem.~6.14]{BK}). Then
$$
zn_2 = (zuz^{-1}) m_0 \cdot zn_2 \cdot j \in I zn_2 I,
$$
and so by (\ref{BT_decomp}), $zn_2 \equiv zn_2$ modulo $M(F)_1$.
\end{proof}

\begin{lemma} \label{N_vs_I_orbits} If $x, y \in \widetilde{W}$ and $M(F)_1 N(F) x I \cap IyI \neq \emptyset$, then $x \leq y$ in the Bruhat order on $\widetilde{W}$ determined by $I$.
\end{lemma}

\begin{proof}
This follows from the BN-pair relations (\ref{BN_pair}) as in the proof of the Claim in Lemma 1.6.1 of \cite{HKP}.
\end{proof}

\subsubsection{The universal unramified principal series module ${\bf M}$}

Define 
$${\bf M} = C_c(M(F)_1 N(F) \backslash G(F)/I).
$$
The subscript ``c'' means we are considering functions supported on finitely many double cosets.  Some basic facts about ${\bf M}$ were given in \cite{HKP} for the special case where $G$ is split, and here we need to state those facts in the current general situation.

Abbreviate by setting $H = \mathcal H(G(F),I)$ and $R = \mathbb C[\Lambda_M]$.  Then $f \in H$ acts on the left on ${\bf M}$ by right convolutions by $\breve{f}$, which is defined by $\breve{f}(g) = f(g^{-1})$.  The same goes for the normalized induced representation $i^G_P(\tilde{\chi})^I = {\rm Ind}_P^G(\delta_P^{1/2}\tilde{\chi})^I$, where $\tilde{\chi}$ is a character on $M(F)/M(F)_1$. Moreover, $R$ acts on the left on ${\bf M}$ by normalized left convolutions: for $r \in R$ and $\phi \in {\bf M}$, $m \in M(F)$, 
$$
(r \cdot \phi)(m) = \int_{M(F)} r(y) \delta_P^{1/2}(y) \phi(y^{-1}m) \, dy
$$
where ${\rm vol}_{dy}(M(F)_1) = 1$.  The actions of $R$ and $H$ commute, so ${\bf M}$ is an $(R,H)$-bimodule.

\begin{lemma} \label{key_M_lem} The following statements hold.
\begin{enumerate}
\item[(a)] Any character $\tilde{\chi}^{-1} : M(F)/M(F)_1 \rightarrow \mathbb C^\times$ extends to an algebra homomorphism $\tilde{\chi}^{-1}: R \rightarrow \mathbb C$, and there is an isomorphism of left $H$-modules
$$
\mathbb C \otimes_{R,\tilde{\chi}^{-1}} {\bf M} = i^G_P(\tilde{\chi})^I.
$$
\item[(b)] For $w \in W(G,A) =: W$, set $v_w = 1_{M(F)_1 N(F) w I} \in {\bf M}$.  Then ${\bf M}$ is free of rank 1 over $H$ with canonical generator $v_1$.
\item[(c)] ${\bf M}$ is free as an $R$-module, with basis $\{v_w\}_{w\in W}$.
\end{enumerate} 
\end{lemma}

\begin{proof}
The proofs for (a-b) are nearly identical to their analogues for split groups in \cite{HKP}.  Part (a) is formal.  Part (b) relies on the Bruhat-Tits decomposition (\ref{BT_decomp}), the Iwasawa decomposition (\ref{1st}), and Lemma \ref{N_vs_I_orbits}.  

Part (c) was not explicitly mentioned in \cite{HKP}.  But it can be proved using (\ref{1st}) along with the relations analogous to \cite[(1.6.1-1.6.2)]{HKP}, for which the Iwahori factorization (\ref{Iwah_fact}) is the main ingredient.
\end{proof}

\subsection{Why $(M(F)_1, 1)$ is an $\mathfrak S_M$-type}

We let $\chi$ range over the characters of $M(F)^1/M(F)_1$.  Let $\tilde{\chi}$ denote any extension to a character of the finitely generated abelian group $M(F)/M(F)_1$.  Fix one such extension $\tilde{\chi}_0$.  
Note that the inertial class $[M(F), \tilde{\chi}_0]_M$ consists of all pairs $(M(F),\tilde{\chi})$, since $M(F)$-conjugation does not introduce any new characters on $M(F)$.  Therefore we may abuse notation and denote this inertial class by $[M(F),\tilde{\chi}]_M =: \mathfrak s^M_\chi$.  Let $\mathfrak S_M := \{ [M(F),\tilde{\chi}]_M ~ | ~ \chi \, \mbox{ranges as above} \}$.  This is a finite set of inertial classes, in bijective correspondence with $M(F)^1/M(F)_1$.  

\begin{prop}
The pair $(M(F)_1, 1)$ is a Bushnell-Kutzko type for $\mathfrak S_M$.
\end{prop}

Note: This proposition simply makes precise the last paragraph of \cite[(9.2)]{BK}.

\begin{proof}
Let $\sigma$ be an irreducible smooth representation of $M(F)$.  We must show that $\sigma = \tilde{\chi}$ for some $\tilde{\chi}$ iff $\sigma|_{M(F)_1} \supset 1$.  

\noindent $(\Rightarrow)$: Obvious. 

\noindent $(\Leftarrow)$: We see that $\sigma \ni v \neq 0$ on which $M(F)_1$ acts trivially.  Since $\sigma$ is irreducible, it coincides with the smallest $M(F)$-subrepresentation containing $v$, and then since $M(F)_1 \triangleleft M(F)$, we see that $M(F)_1$ acts trivially on all of $\sigma$; further, $\sigma$ is necessarily finite-dimensional over $\mathbb C$.  Since $M(F)/M(F)_1$ is abelian, $\sigma$ contains an $M(F)$-stable line, since a commuting set of matrices can be simultaneously triangularized.  This line is all of $\sigma$ since $\sigma$ is irreducible.  Thus $\sigma$ is 1-dimensional, and so $\sigma = \tilde{\chi}$ for some $\tilde{\chi}$. 
\end{proof}

\subsection{Why $(I,1)$ is an $\mathfrak S_G$-type}

We define $\mathfrak S_G = \{ [t]_G ~ | ~ [t]_M \in \mathfrak S_M \}$.  The map $[M,\tilde{\chi}]_M \mapsto [M,\tilde{\chi}]_G$ is injective: if $[M,\tilde{\chi_1}]_G = [M, \tilde{\chi_2}]_G$, then there exists $n \in N_G(A)(F)$ such that $^n(\tilde{\chi_1}) = \tilde{\chi_2} \eta$ for some character $\eta$ on $M(F)/M(F)^1$.  Restricting to $M(F)^1$ and using $^n(\chi_1) = \chi_1$ (Lemma \ref{basic_lem}(b)), we see $\chi_1 = \chi_2$.  So $\mathfrak S_M \cong \mathfrak S_G$ via $[t]_M \mapsto [t]_G$.

We saw above that $(M(F)_1, 1)$ is an $\mathfrak S_M$-type.  
The fact that $(I,1)$ is an $\mathfrak S_G$-type follows from \cite[Thm.~8.3]{BK}, once we check the following proposition.

\begin{prop} \label{G-cover_check}
The pair $(I,1)$ is a $G$-cover for $(M(F)_1, 1)$ in the sense of \cite[Definition~8.1]{BK}.
\end{prop}

\begin{proof}
We need to check the three conditions (i-iii) of Definition 8.1.  First (i), the fact that $(I,1)$ is decomposed with respect to $(M,P)$ in the sense of \cite[(6.1)]{BK}, follows from the Iwahori factorization $I = I_N \cdot I_M \cdot I_{\overline{N}}$ discussed in Remark \ref{Iwah_fact_rem}. The equality $I\cap M(F) = M(F)_1$ gives condition (ii). 

Finally we must prove (iii): for every $F$-parabolic $P$ with Levi factor $M$, there exists an invertible element of $\mathcal H(G(F),I)$ supported on $Iz_PI$, where $z_P$ belongs to $Z(M)(F)$ and is strongly $(P,I)$-positive. The existence of elements $z_P \in Z(M)(F)$ which are strongly $(P,I)$-positive is proved in \cite[Lemma 6.14]{BK}.  Any corresponding characteristic function $1_{Iz_PI}$ is invertible in $\mathcal H(G(F),I)$, as follows from the Iwahori-Matsumoto presentation of $\mathcal H(G(F),I)$.  (This presentation itself is easy to prove using (\ref{BN_pair}).)
\end{proof}


\subsection{Structure of the Bernstein varieties} 

Let  $\mathfrak R(G)$ denote the category of smooth representations of $G(F)$, and let $\mathfrak R_{\chi}(G)$ denote the full subcategory corresponding to the inertial class $[M,\tilde{\chi}]_G$.    That is, a representation $(\pi, V) \in \mathfrak R(G)$ is an object of $\mathfrak R_\chi(G)$ if and only if for each irreducible subquotient  $\pi'$ of $\pi$, there exists an extension $\tilde{\chi}$ of $\chi$ such that $\pi'$ is a subquotient of 
${\rm Ind}_P^G(\delta_P^{1/2} \tilde{\chi})$.

We review the structure of the Bernstein varieties $\mathfrak X^G_{\chi} $ and $\mathfrak X^M_\chi$. In this discussion, for each $\chi$ we fix an extension $\tilde{\chi}$ of $\chi$ once and for all -- the structures we define will be independent of the choice of $\tilde{\chi}$, i.e. uniquely determined by $(M,\chi)$ up to a unique isomorphism. 

As a set $\mathfrak X^G_\chi$ (resp. $\mathfrak X^M_\chi$) consists of the elements $(M,\tilde{\chi} \eta)_G$ (resp. $(M, \tilde{\chi} \eta)_M$) belonging to the inertial equivalence class $[M,\tilde{\chi}]_G$ (resp. $[M,\tilde{\chi}]_M$) as $\eta$ ranges over the set  $X(M)$ of unramified $\mathbb C^\times$-valued characters 
on $M(F)$ (unramified means it factors through $M(F)/M(F)^1$).  

 The map $X(M) \rightarrow \mathfrak X^M_\chi$, $\eta \mapsto (M,\tilde{\chi}\eta)_M$, is a bijection.  Since $X(M)$ is a complex torus, this gives $\mathfrak X^M_\chi$ the structure of a complex torus.  More canonically,  $\mathfrak X^M_\chi$ is just the variety of {\em all} extensions $\tilde{\chi}$ of $\chi$, and it is a torsor under the torus $X(M)$.

Now fix $\tilde{\chi}$ again.  There is a surjective map
\begin{align*}
\mathfrak X^M_\chi ~ &\rightarrow ~ \mathfrak X^G_\chi \\
(M,\tilde{\chi}\eta)_M ~ &\mapsto ~ (M,\tilde{\chi} \eta)_G.
\end{align*}
Since $W := W(G,A)$ acts trivially on $M(F)^1/M(F)_1$ (Lemma \ref{basic_lem}), one can prove that the fibers of this map are precisely the $W$-orbits on $\mathfrak X^M_\chi$.  Thus as sets
$$
W \backslash \mathfrak X^M_\chi = \mathfrak X^G_\chi,
$$
and this gives $\mathfrak X^G_\chi$ the structure of an affine variety over $\mathbb C$.  Having chosen the isomorphism $X(M) \cong \mathfrak X^M_\chi$ as above, we can transport the $W$-action on $\mathfrak X^M_\chi$ over to an action on $X(M)$.  {\em This action depends on the choice of $\tilde{\chi}$ and is not the usual action unless $\tilde{\chi}$ is $W$-invariant}.   We obtain $\mathfrak X^G_\chi = W \backslash_{\tilde{\chi}} X(M)$, where the latter denotes the quotient with respect to this {\em unusual} action on $X(M)$.

Let $\mathbb C[\mathfrak X^G_\chi]$ denote the ring of regular functions on the variety $\mathfrak X^G_\chi$. The algebraic morphism $\mathfrak X^M_\chi \rightarrow \mathcal X^G_\chi$ induces an isomorphism of algebras
\begin{equation} \label{reg_fcns}
\mathbb C[\mathfrak X^G_\chi] ~ \widetilde{\rightarrow} ~ \mathbb C[\mathfrak X^M_\chi]^W.
\end{equation}

\subsection{Consequences of the theory of types}

Let us define convolution in $\mathcal H(G(F),I)$ using the Haar measure $dx$ on $G(F)$ which gives $I$ volume 1.  Let $\mathcal Z(G(F),I)$ denote the center of $\mathcal H(G(F),I)$.

We define for each $\chi \in (M(F)^1/M(F)_1)^\vee$ a function $e_{\chi} \in \mathcal H(G(F),I)$ by requiring it to be supported on $\tilde{I}$, and by setting $e_{\chi}(y) = {{\rm vol}_{dx}(\tilde{I})^{-1} \, \chi}(\bar{y})^{-1}$ if $y \in \tilde{I}$.  Here we regard $\chi$
 as a character on $\tilde{I}/I$ (cf. Lemma \ref{basic_lem}) and let $\bar{y} \in \tilde{I}/I$ denote the image of $y$.  If $y = n_+ \cdot m^1 \cdot n_- \in I_N\cdot M(F)^1\cdot I_{\overline{N}}$, then $e_{\chi}(y) = {\rm vol}_{dx}(\tilde{I})^{-1}\, \chi(m^1)^{-1}$.

\begin{lemma} \label{e_chi_properties}
The functions $\{ e_\chi\}_\chi$ give a complete set of central orthogonal idempotents for $\mathcal H(G(F),I)$:
\begin{enumerate}
\item[(a)] $e_\chi \in \mathcal Z(G(F),I)$;
\item[(b)] $e_\chi e_{\chi'} = \delta_{\chi, \chi'} e_\chi$, there $\delta_{\chi, \chi'}$ is the Kronecker delta function;
\item[(c)] $1_{I} = \sum_{\chi} e_\chi$.
\end{enumerate}
\end{lemma}

\begin{proof}
The proof is a straightforward exercise for the reader.  For parts (a-b), use the fact that $M(F)^1$ normalizes $I$, that $G(F) = I \cdot N_G(A)(F) \cdot I$, and that $W(G,A)$ acts trivially on $M(F)^1/M(F)_1$ (cf. Lemma \ref{basic_lem}).
\end{proof}

\begin{prop}  \label{idem_prop} The idempotents $e_\chi$ are the elements in the Bernstein center which project the category $\mathfrak R(G)$ onto the various Bernstein components $\mathfrak R_{\chi}(G)$.  
That is, there is a canonical isomorphism of algebras
$$
\mathcal H(G(F),I) = \prod_{\chi} e_\chi \mathcal H(G(F),I) e_\chi
$$
and, for any smooth representation $(\pi, V) \in \mathfrak R(G)$, the $G(F)$-module spanned by the $\chi$-isotypical vectors $V^\chi = \pi(e_\chi) V$ is the component of $V$ lying in the subcategory $\mathfrak R_\chi(G)$. Finally, 
$$
e_\chi \mathcal H(G(F),I) e_\chi = \mathcal H(G(F), \tilde{I}, \chi),
$$
the right hand side being the algebra of $I$-bi-invariant $\mathbb C$-valued functions $f \in C_c(G(F))$ such that $f(\tilde{i}_1 g \tilde{i}_2) = \chi(\tilde{i}_1)^{-1} f(g) \chi(\tilde{i}_2)^{-1}$ for all $g \in G(F)$ and $\tilde{i}_1, \tilde{i}_2 \in \tilde{I}$.
\end{prop}

The following records the standard consequences of the fact that $(I,1)$ is an $\mathfrak S_G$-type (see \cite[Thm.~4.3]{BK}). Let $\mathfrak R_I(G)$ denote the full subcategory of $\mathfrak R(G)$ whose objects are generated as $G$-modules by their $I$-invariant vectors.

\begin{prop} \label{type_prop} As subcategories of $\mathfrak R(G)$, we have the equality $\mathfrak R_I(G) = \prod_{\chi} \mathfrak R_\chi(G)$.  In particular, an irreducible representation $(\pi, V) \in \mathfrak R(G)$ belongs to $\mathfrak R_I(G)$ if and only if $\pi \in \mathfrak R_\chi(G)$ for some $\chi$. Furthermore, there is an equivalence of categories
\begin{align*}
\mathfrak R_I(G) ~ &\widetilde{\rightarrow} ~ \mathcal H(G(F),I)\mbox{-\rm Mod} \\
(\pi,V) ~ &\mapsto ~ V^I.
\end{align*}
Finally, $\mathcal Z(G(F),I)$ is isomorphic with the center of the category $\prod_{\chi} \mathfrak R_{\chi}(G)$, which according to Bernstein's theory is the ring $\prod_\chi \mathbb C[\mathfrak X^G_\chi]$.
\end{prop}

Concretely, the map $\mathcal Z(G(F),I) \rightarrow \prod_\chi \mathbb C[\mathfrak X^G_\chi]$, $z \mapsto \hat{z}$, is characterized as follows: for every $\chi$ and every $(M,\tilde{\chi})_G \in \mathfrak X^G_\chi$, $z \in \mathcal Z(G(F),I)$ acts on ${\rm Ind}_P^G(\delta^{1/2}_P \tilde{\chi})^I$ by the scalar $\hat{z}(\tilde{\chi})$. 

Let us single out what happens in the special case of $G=M$. We can identify $\mathcal H(M(F), M(F)_1) = \mathbb C[\Lambda_M]$.  Let $e^M_\chi $ denote the idempotent in $\mathcal H(M(F),M(F)_1)$ analogous to $e_\chi$, for the case $G = M$.  By Propositions \ref{idem_prop} and \ref{type_prop} for $G=M$, we have
\begin{equation} \label{M_product}
\mathcal H(M(F),M(F)_1) = \prod_\chi e^M_\chi \mathcal H(M(F),M(F)_1) e^M_\chi = \prod_{\chi} \mathbb C[\mathfrak X^M_\chi],
\end{equation}
the last equality holding since $\mathcal H(M(F),M(F)_1)$ is already commutative.  Thus, the ring 
$$e^M_\chi \mathcal H(MF),M(F)_1) e^M_\chi$$ can be regarded as the ring of regular functions on the variety $\mathfrak X^M_\chi$ of all extensions $\tilde{\chi}$ of $\chi$.

\subsection{The embedding of $\mathbb C[\Lambda_M]^W$ into $\mathcal Z(G(F),I)$}

We make use of the following special case of a general construction of Bushnell-Kutzko \cite{BK}: for any $F$-parabolic $P$ with Levi factor $M$, there is an injective algebra homomorphism
$$
t_P: \mathcal H(M(F), M(F)_1) \rightarrow \mathcal H(G(F),I)
$$
which is uniquely characterized by the property that for each $(\pi,V) \in \mathfrak R_I(G)$, $v \in V^I$, and $h \in \mathcal H(M(F),M(F)_1)$, we have the identity
\begin{equation} \label{t_P_property}
q_\pi(t_P(h) \cdot v) = h \cdot q_\pi(v).
\end{equation}
Here $q_\pi: V^I ~ \widetilde{\rightarrow} ~ V_N^{M(F)_1}$ is an isomorphism, which is induced by the canonical projection $V \rightarrow V_N$ to the (unnormalized) Jacquet module.   See \cite[Thm.~7.9]{BK}.

It turns out that it is better to work with a different normalization.  We define another injective algebra homomorphism
\begin{align*}
\theta_P: \mathcal H(M(F), M(F)_1) ~ &\rightarrow ~ \mathcal H(G(F),I) \\
h ~ &\mapsto ~ t_P(\delta_P^{-1/2}h).
\end{align*}
Then using (\ref{t_P_property}) $\theta_P$ satisfies 
\begin{equation} \label{tilde-t_P_property}
q_\pi(\theta_P(h) \cdot v) = (\delta_P^{-1/2}h) \cdot q_\pi(v).
\end{equation}

We view $\tilde{\chi}$ as a varying element of the Zariski-dense subset $S$ of the variety of all characters on the finitely-generated abelian group $M(F)/M(F)_1$, consisting of those regular characters $\tilde{\chi}$ such that $V(\tilde{\chi}) := i^G_P(\tilde{\chi}) := {\rm Ind}^G_P(\delta_P^{1/2}\tilde{\chi})$ is irreducible as an object of $\mathfrak R(G)$. 
We apply the above discussion to the representations $V := V(\tilde{\chi})$ with $\tilde{\chi} \in S$. 
By a result of Casselman \cite{Cas}, we know that as $M(F)$-modules 
$$
\displaystyle{
V_N = \bigoplus_{w \in W} \delta_P^{1/2}(\, ^w\tilde{\chi})}$$ 
and that $M(F)_1$ acts trivially on this module. Now suppose $h \in \mathbb C[\Lambda_M]^W$. Then $\delta_P^{-1/2}h$ acts 
on $V_N = V_N^{M(F)_1}$ by the scalar $h(\tilde{\chi})$ (viewing $h$ as a regular function on $\mathfrak X^M_\chi$).  It follows from (\ref{tilde-t_P_property}) that 
\begin{equation*} 
\mbox{$\theta_P(h)$ acts by the scalar $h(\tilde{\chi})$ on $i^G_P(\tilde{\chi})^I$, for $\tilde{\chi} \in S$}.
\end{equation*}
Now let $f \in H$ be arbitrary, and set $\epsilon := f * \theta_P(h) - \theta_P(h) *f \in H$.  We see that 
\begin{equation} \label{action_over_S}
\mbox{ $\epsilon$ acts by zero on $i^G_P(\tilde{\chi})^I$ for every $\tilde{\chi} \in S$.}
\end{equation}

We claim that $\epsilon = 0$.  Recall that $\epsilon \in H$ gives an $R$-linear endomorphism of ${\bf M}$, hence by Lemma \ref{key_M_lem}(c) may be represented by an $|W| \times |W|$ matrix $E$ with entries in $R$.  Now ${\rm Spec}(R) = {\rm Spec}(\mathbb C[\Lambda_M])$ is a diagonalizable group scheme over $\mathbb C$ with character group $\Lambda_M$.  Hence $R$ is a reduced finite-type $\mathbb C$-algebra and its maximal ideals are precisely the kernels of the $\mathbb C$-algebra homomorphisms $\tilde{\chi}^{-1}: R \rightarrow \mathbb C$ coming into Lemma \ref{key_M_lem}(a).  By that Lemma and (\ref{action_over_S}), we see that $E \equiv 0 \,\, ({\rm mod} \, \mathfrak m)$ for a Zariski-dense set of maximal ideals $\mathfrak m$ in ${\rm Spec}(R)$.  Since $R$ is reduced and finite-type over $\mathbb C$, this implies that the entries of $E$ are identically zero.  This proves the claim because ${\bf M}$ is free of rank 1 over $H$ (Lemma \ref{key_M_lem}(b)).

Since $f$ was arbitrary, we get $\theta_P(h) \in \mathcal Z(G(F),I)$, as desired.  We have proved the following result.

\begin{lemma} The map $\theta_P: \mathbb C[\Lambda_M] \rightarrow \mathcal H(G(F),I)$ restricts to give an embedding $\mathbb C[\Lambda_M]^W \rightarrow \mathcal Z(G(F),I)$. 
\end{lemma}

\subsection{The center of the Iwahori-Hecke algebra}

\begin{theorem} \label{my_I_thm}
The map $\theta_P$ gives an algebra isomorphism $\mathbb C[\Lambda_M]^W ~ \widetilde{\rightarrow} ~ \mathcal Z(G(F),I)$. Further, this isomorphism is independent of the choice of parabolic $P$ containing $M$ as a Levi factor.
\end{theorem}

\begin{proof}
The description of $\theta_P$ above, and the preceding discussion, show that we have a commutative diagram
$$
\xymatrix{
\mathbb C[\Lambda_M]^W \ar[r]^{\sim} \ar[d]_{\theta_P} & \prod_\chi \mathbb C[\mathfrak X^M_\chi]^W
 \\
\mathcal Z(G(F),I) \ar[r]^{\sim} & \prod_\chi \mathbb C[\mathfrak X^G_\chi] \ar[u] \, .}
$$
The left vertical arrow is bijective because the right vertical arrow is, by (\ref{reg_fcns}).
\end{proof}

\subsection{Bernstein isomorphisms and functions}

Putting together Roche's theorem \ref{Roche_thm} with Theorem \ref{my_I_thm}, we deduce a more general result that holds for any parahoric subgroup $J \supseteq I$.

\begin{theorem} \label{my_J_thm}  The composition
\begin{equation} \label{J_thm_map}
\xymatrix{
\mathbb C[\Lambda_M]^W \ar[r]^{\theta_P} &  \mathcal Z(G(F),I) \ar[r]^{-*_I 1_J} & \mathcal Z(G(F), J) }
\end{equation}
is an isomorphism.  We call this map the {\em Bernstein isomorphism}.
\end{theorem}

\begin{defn} \label{Bern_fcn_def}
Given $\mu \in \Lambda_M$, we define the {\em Bernstein function } $z_\mu \in \mathcal Z(G(F), J)$ to be the image of the symmetric monomial function $\sum_{\lambda \in W\cdot \mu} \lambda \in \mathbb C[\Lambda_M]^W$ under the Bernstein isomorphism (\ref{J_thm_map}).
\end{defn}

\subsection{Compatibility with constant terms} \label{Bern_vs_const_term_sec}

Recall $M = {\rm Cent}_G(A)$ is a minimal $F$-Levi subgroup of $G$ and $P = MN$ is a minimal $F$-parabolic subgroup with Levi factor $M$ and unipotent radical $N$.  Let $Q = LR$ be another $F$-parabolic subgroup with $F$-Levi factor $L$ and unipotent radical $R$. Assume $Q \supseteq P$; then $L \supseteq M$ and $R \subseteq N$.  Further $L$ contains a minimal $F$-parabolic subgroup $L \cap P = M \cdot (L \cap N)$, and $N = L \cap N \cdot R$. 

If $J \subset G(F)$ is a parahoric subgroup corresponding to a facet in the apartment of the Bruhat-Tits building of $G(F)$ corresponding to $A$, then $J_L := J \cap L$ is a parahoric subgroup of $L(F)$ (by \cite[Lem.~4.1.1]{HRo}).

Given $f \in \mathcal H(G(F),J)$, define $f^{(Q)} \in \mathcal H(L(F), J_L)$ by
$$
f^{(Q)}(l) = \delta_Q^{1/2}(l) \int_{R} f(lr) \, dr = \delta_Q^{-1/2}(l) \int_{R} f(rl) \, dr,
$$
where ${\rm vol}_{dr}(J \cap R) = 1$. An argument similar to Lemma 4.7.2 of \cite{H09} shows that $f \mapsto f^{(Q)}$ sends $\mathcal Z(G(F),J)$ into $\mathcal Z(L(F), J_L)$, and determines a map $c^G_L$ making the following diagram commute:
\begin{equation} \label{const_term_comp}
\xymatrix{
\mathbb C[\Lambda_M]^{W(G,A)} \ar[r]^{\sim} \ar@{^(->}[d] & \mathcal Z(G(F),J) \ar[d]_{c^G_L} \\ \mathbb C[\Lambda_M]^{W(L,A)} \ar[r]^{\sim} & \mathcal Z(L(F) , J_L).}
\end{equation}
The diagram shows that $c^G_L$ is indeed an (injective) algebra homomorphism and, as the notation suggests, is independent of the choice of parabolic subgroup $Q$ having $L$ as a Levi factor. We call $c^G_L$ the {\em constant term homomorphism}.

By its very construction, the map $\theta_M: \mathbb C[\Lambda_M] ~ \widetilde{\rightarrow} ~ \mathcal H(M(F), M(F)_1)$ has its inverse induced by the Kottwitz isomorphism $\kappa_M(F): M(F)/M(F)_1 ~ \widetilde{\rightarrow} ~ \Lambda_M$. By taking $L = M$ in (\ref{const_term_comp}), this remark allows us to write down the inverse of $\theta_P$ in general.

\begin{cor} \label{theta_inverse}
The composition $\kappa_M(F) \circ c^G_M$ takes values in $\mathbb C[\Lambda_M]^{W(G,A)}$ and is the inverse of the Bernstein isomorphism $\theta_P$.
\end{cor}

\subsection{Transfer homomorphisms} \label{App:transfer_hom_sec}

\subsubsection{Construction}
Transfer homomorphisms were defined for special maximal parahoric Hecke algebras in \cite{HRo}. By virtue of the Bernstein isomorphisms (\ref{J_thm_map}), we can now define these homomorphisms on the level of centers of arbitrary parahoric Hecke algebras.

Let us recall the general set-up from \cite[$\S11.2$]{HRo}.  Let $G^*$ be a quasi-split group over $F$. Let $F^s$ denote a separable closure of $F$, and set $\Gamma = {\rm Gal}(F^s/F)$. Recall that an inner form of $G^*$ is a pair $(G,\Psi)$ consisting of a connected reductive $F$-group $G$ and a $\Gamma$-stable $G^*_{\rm ad}(F^s)$-orbit $\Psi$ of $F^s$-isomorphisms $\psi: G \rightarrow G^*$.  The set of isomorphism classes of pairs $(G,\Psi)$ corresponds bijectively to $H^1(F, G^*_{\rm ad})$. 

Fix once and for all parahoric subgroups $J \subset G(F)$ and $J^* \subset G^*(F)$. Choose any maximal $F$-split tori $A \subset G$ and $A^* \subset G^*$ such that the facet fixed by $J$ (resp.~$J^*$) is contained in the apartment of the building $\mathcal B(G,F)$ (resp. $\mathcal B(G^*,F)$) corresponding to the torus $A$ (resp.~$A^*$). Let $M = {\rm Cent}_G(A)$ and $T^* = {\rm Cent}_{G^*}(A^*)$, a maximal $F$-torus in $G^*$. 

Now {\em choose} an $F$-parabolic subgroup $P \subset G$ having $M$ as Levi factor, and an $F$-rational Borel subgroup $B^* \subset G^*$ having $T^*$ as Levi factor. Then there exists a unique parabolic subgroup $P^* \subset G^*$ such that $P^* \supseteq B^*$ and $P^*$ is $G^*(F^s)$-conjugate to $\psi(P)$ for every $\psi \in \Psi$.  Let $M^*$ be the unique Levi factor of $P^*$ containing $T^*$. Then define
$$
\Psi_M = \{ \psi \in \Psi ~ | ~ \psi(P) = P^*, \,\, \psi(M) = M^* \}.
$$
The set $\Psi_M$ is a nonempty $\Gamma$-stable $M^*_{\rm ad}(F^s)$-orbit of $F^s$-isomorphisms $M \rightarrow M^*$, and so $(M,\Psi_M)$ is an inner form of $M^*$. Choose any $\psi_0 \in \Psi_M$.  Then since $\psi_0|A$ is $F$-rational, $\psi_0(A)$ is an $F$-split torus in $Z(M^*)$ and hence $\psi_0(A) \subseteq A^*$.

Now $\psi_0$ induces a $\Gamma$-equivariant map $Z(\widehat{M}) ~ \widetilde{\rightarrow} ~ Z(\widehat{M^*}) \hookrightarrow \widehat{T^*}$ and hence a homomorphism 
$$
t_{A^*,A}: \mathbb C[X^*(\widehat{T^*})^{\Phi^*_F}_{I^*_F}]^{W(G^*,A^*)} \longrightarrow \mathbb C[X^*(Z(\widehat{M}))^{\Phi_F}_{I_F}]^{W(G,A)},
$$
where $(\cdot)^*$ designates the Galois action on $G^*$ (for Weyl-group equivariance see \cite[$\S12.2$]{HRo}). Since $\Psi_M$ is a torsor for $M^*_{\rm ad}$ this homomorphism does not depend on the choice of $\psi_0 \in \Psi_M$. Further, it depends only on the choice of $A^*$ and $A$, and not on the choice of the parabolic subgroups $P \supset M$ and $B^* \supset T^*$ we made in constructing it.

\begin{defn} \label{t_defn}
Let $J \subset G(F)$ and $J^* \subset G^*(F)$ be any parahoric subgroups and choose compatible maximal $F$-split tori $A$ resp.~$A^*$ as above. Then we define the {\em transfer homomorphism} $t: \mathcal Z(G^*(F), J^*) \rightarrow \mathcal Z(G(F), J)$ to be the unique homomorphism making the following diagram commute
$$
\xymatrix{
\mathcal Z(G^*(F), J^*) \ar[r]^{t} \ar[d]_{\wr} & \mathcal Z(G(F), J) \ar[d]_{\wr}  \\
 \mathbb C[X^*(\widehat{T^*})^{\Phi_F^*}_{I^*_F}]^{W(G^*,A^*)} \ar[r]^{t_{A^*,A}} & \mathbb C[X^*(Z(\widehat{M}))^{\Phi_F}_{I_F}]^{W(G,A)},
}$$
where the vertical arrows are the Bernstein isomorphisms.
\end{defn}
By \cite[4.6.28]{BT2}, any two choices for $A$ (resp.~$A^*$) are $J$-(resp.~$J^*$-)conjugate. Using Corollary \ref{theta_inverse} it follows that $t$ is independent of the choice of $A$ and $A^*$ and is a completely canonical homomorphism.

\begin{Remark} \label{t_surj}
The map 
$$
t_{A^*,A}: X^*(\widehat{T^*})^{\Phi^*_F}_{I^*_F} \rightarrow X^*(Z(\widehat{M}))^{\Phi_F}_{I_F}
$$ 
is surjective. Via the Kottwitz homomorphism we may view this as the composition
\begin{equation} \label{composition}
\xymatrix{
T^*(F)/T^*(F)_1 \ar[r] & M^*(F)/M^*(F)_1 \ar[r]^{\psi^{-1}_0}_{\sim} & M(F)/M(F)_1}
\end{equation}
where the first arrow is induced by the inclusion $T^* \hookrightarrow M^*$. It is enough to observe that $M^*(F) = T^*(F) \cdot M^*(F)_1$, which in turn follows from the Iwasawa decomposition (cf.~(\ref{2nd})) for $M^*(F)$, which states that $M^*(F) = T^*(F) \cdot U^*_{M^*}(F) \cdot K_{M^*}$ for an $F$-rational Borel subgroup $B^*_{M^*} = T^* \cdot U^*_{M^*}$ and a special maximal parahoric subgroup $K_{M^*}$ in $M^*$, and from the vanishing of the Kottwitz homomorphism on $U^*_{M^*}(F) \cdot K_{M^*}$.
\end{Remark}

\subsubsection{Normalized transfer homomorphism}

The transfer homomorphism is slightly too naive, and it is necessary to normalize it in order to get a homomorphism which has the required properties.  We need to define normalized homomorphisms $\widetilde{t}_{A^*,A}$ on Weyl-group invariants, for which the following lemma is needed.

\begin{lemma} \label{Weyl_grp_lem}
Recall that $T^* = {\rm Cent}_{G^*}(A^*)$ is a maximal torus in $G^*$ defined over $F$; let $S^*$ be the $F^{\rm un}$-split component of $T^*$, a maximal $F^{\rm un}$-split torus in $G^*$ defined over $F$ and containing $A^*$. We have $T^* = {\rm Cent}_{G^*}(A^*) = {\rm Cent}_{G^*}(S^*)$. Choose a maximal $F^{\rm un}$-split torus $S \subset G$ which is defined over $F$ and which contains $A$, and set $T = {\rm Cent}_G(S)$. Choose $\psi_0 \in \Psi_M$ such that $\psi_0$ is defined over $F^{\rm un}$ and satisfies $\psi_0(S) = S^*$ and hence $\psi_0(T) = T^*$. Then the diagram
$$
\xymatrix{
W(G,A) \ar@{-->}[r]^{\psi^\natural_0}_{\sim} \ar[d]_{\wr} & W(G^*,A^*)/W(M^*,A^*) \ar[d]_{\wr} \\
\big[ W(G,S)/W(M,S)\big]^{\Phi_F} \ar[r]^{\psi_0}_{\sim} & \big[ W(G^*,S^*)/W(M^*,S^*) \big]^{\Phi^*_F}}
$$
defines a bijective map $\psi^\natural_0$. It depends on the choice of the data $P, B^*$ used to define $\Psi_M$ and $M^*$, but it is independent of the choices of $S$ and $\psi_0 \in \Psi_M$ with the stated properties.
\end{lemma}

\begin{proof}
The left vertical arrow is \cite[Lemma 6.1.2]{HRo}.  The right vertical arrow is described in \cite[Proposition 12.1.1]{HRo}. The proof of the latter justifies the lower horizontal arrow. Indeed, given $w \in W(G,A)$ we may choose a representative $n \in N_G(S)(L)^{\Phi_F}$ (cf.~\cite{HRo}). We have $\psi^{-1}_0 \circ \Phi^*_F \circ \psi_0 \circ \Phi^{-1}_F = {\rm Int}(m_\Phi)$ for some $m_\Phi \in N_{M}(S)(F^s)$. Since $n$ is $\Phi_F$-fixed, we get
$$
\Phi^*_F(\psi_0(n)) = \psi_0(n) \cdot \big[\psi_0(n)^{-1} \psi_0(m_\Phi) \psi_0(n) \psi_0(m_\Phi)^{-1}\big].
$$
As $n$ normalizes $M$ and hence $\psi_0(n)$ normalizes $M^*$, this shows that $\psi_0(n)W(M^*,S^*)$ is $\Phi^*_F$-fixed.

There exists $m^*_n \in N_{M^*}(S^*)(L)$ such that $\psi_0(n)m^*_n \in N_{G^*}(A^*)(F)$. Then $\psi_0^\natural(w)$ is the image of $\psi_0(n)m^*_n$ in $W(G^*,A^*)/W(M^*,A^*)$. The independence statement is proved using this description.
\end{proof}
Via the Kottwitz homomorphism we can view $t_{A^*,A}$ as induced by the composition (\ref{composition}). We now alter this slightly.

\begin{lemma} Given the choices of $P \supset M$ and $B^* \supset T^*$ needed to define $\Psi_M$ and given any $F^{\rm un}$-rational $\psi_0 \in \Psi_M$, we define an algebra homomorphism
\begin{align} \label{tilde_t_A_defn}
\mathbb C[T^*(F)/T^*(F)_1] &\longrightarrow \mathbb C[M(F)/M(F)_1] \\
\sum_{t^*} a_{t^*} t^* &\longmapsto \sum_{m} \big(\sum_{t^* \mapsto m} a_{t^*} \delta_{B^*}^{-1/2}(t^*) \delta_P^{1/2}(m) \big)\cdot m, \notag
\end{align}
where $t^*$ ranges over $T^*(F)/T^*(F)_1$ and $m$ ranges over $M(F)/M(F)_1$ and $t^*\mapsto m$ means that $\psi_0^{-1}(t^*) \in mM(F)_1$, \textup{(}cf.~\textup{(}\ref{composition}\textup{)}\textup{)}.

Then \textup{(}\ref{tilde_t_A_defn}\textup{)} takes $W(G^*,A^*)$-invariants to $W(G,A)$-invariants, and the resulting homomorphism
$$
\tilde{t}_{A^*,A}: \mathbb C[T^*(F)/T^*(F)_1]^{W(G^*,A^*)} \rightarrow \mathbb C[M(F)/M(F)_1]^{W(G,A)}
$$
is independent of the choices of $P$, $B^*$, and $F^{\rm un}$-rational $\psi_0 \in \Psi_M$.
\end{lemma}

\begin{proof}
To check the Weyl-group invariance, we may fix $P$ and $B^*$, and choose $S$ and $\psi_0$ as in Lemma \ref{Weyl_grp_lem}.  Suppose $\sum_{t^*} a_{t^*} t^*$ is $W(G^*,A^*)$-invariant.  We need to show that the function on $M(F)/M(F)_1$
\begin{equation} \label{m_map}
m \mapsto \sum_{t^* \mapsto m} a_{t^*} \delta_{B^*}^{-1/2}(t^*) \delta_P^{1/2}(m)
\end{equation}
is $W(G,A)$-invariant, and independent of the choice of $P$ and $B^*$.

For $w \in W(G,A)$ choose $n$ and $m^*_n$ as in the proof of Lemma \ref{Weyl_grp_lem}, and define $n'$ by $\psi_0(n') = \psi_0(n) m^*_n$.  Thus $\psi_0(n') \in N_{G^*}(A^*)(F)$ and hence it normalizes $T^*(F)$.

We claim that (\ref{m_map}) takes the same values on $mM(F)_1$ and on $^nmM(F)_1$.  First we observe that $^nmM(F)_1 = \, ^{n'}mM(F)_1$. Setting $m_n := \psi_0^{-1}(m^*_n)$, it is enough to prove $^{m_n}mM(F)_1 = mM(F)_1$. But as $\psi_0$ is $L$-rational we have $m_n \in M(L)$ and so conjugation by $m_n$ induces the identity map on $M(L)/M(L)_1$ and hence on its subset $M(F)/M(F)_1$ as well.

Now we write the value of (\ref{m_map}) on $^{n'}mM(F)_1$ as 
$$
\sum_{t^{**} \mapsto \, ^{n'}mM(F)_1} a_{t^{**}} \, \delta^{-1/2}_{B^*}(t^{**}) \delta^{1/2}_P(\,^{n'}m).
$$
Setting $t^* = \,^{\psi_0(n')^{-1}}t^{**}$ and using $a_{t^*} = a_{t^{**}}$ (which follows from $W(G^*,A^*)$-invariance), we write the above as
$$
\sum_{t^* \mapsto m} a_{t^*} \, \delta^{-1/2}_{B^*}(\,^{\psi_0(n')}t^*) \, \delta^{1/2}_P(\, ^{n'}m).
$$
The index set is stable under the $W(M^*,A^*)$-action on $T^*(F)/T^*(F)_1$. If we look at the sum over the $W(M^*,A^*)$-orbit of a single element $t^*_0$, with stabilizer group ${\rm Stab}(t^*_0)$, we get
\begin{equation} \label{this_eq}
\frac{1}{|{\rm Stab}(t^*_0)|} \cdot a_{t^*_0} \cdot \sum_{y} \delta^{-1/2}_{B^*_{M^*}}(\, ^{\psi_0(n') y}t^*_0),
\end{equation}
where $y$ ranges over $W(M^*,A^*)$. Now $n \in N_G(S)(L)^{\Phi_F} \subseteq N_G(A)(F) = N_G(M)(F)$.  Hence $\psi_0(n) m^*_n = \psi_0(n')$ normalizes $M^*$ as well as $T^*$, and thus conjugation by $\psi_0(n')$ takes $B^*_{M^*}$ to another $F$-rational Borel subgroup of $M^*$ containing $T^*$. Using this it is clear that (\ref{this_eq}) is unchanged if the superscript $\psi_0(n')$ is omitted, and this proves our claim. For the same reason (\ref{m_map}) is independent of the choice of $P$ and $B^*$, and similarly $\tilde{t}_{A^*,A}$ is independent of the choice of $P$ and $B^*$, and of the choice of $F^{\rm un}$-rational $\psi_0 \in \Psi_M$. 
\end{proof}

Now we give a normalized version of Definition \ref{t_defn}.

\begin{defn} \label{tilde_t_defn}
We define the {\em normalized transfer homomorphism} $\tilde{t}: \mathcal Z(G^*(F), J^*) \rightarrow \mathcal Z(G(F),J)$ to be the unique homomorphism making the following diagram commute
$$
\xymatrix{
\mathcal Z(G^*(F), J^*) \ar[r]^{\tilde{t}} \ar[d]_{\wr} & \mathcal Z(G(F), J) \ar[d]_{\wr}  \\
 \mathbb C[X^*(\widehat{T^*})^{\Phi_F^*}_{I^*_F}]^{W(G^*,A^*)} \ar[r]^{\tilde{t}_{A^*,A}} & \mathbb C[X^*(Z(\widehat{M}))^{\Phi_F}_{I_F}]^{W(G,A)},
}$$
where the vertical arrows are the Bernstein isomorphisms.
\end{defn}
As was the case for $t$, the homomorphism $\tilde{t}$ is independent of the choice of $A$ and $A^*$, and it is a completely canonical homomorphism.

The following shows it is compatible with constant term homomorphisms.

\subsubsection{Normalized transfer homomorphisms and constant terms} We use the notation of $\S \ref{Bern_vs_const_term_sec}$. Write $L = {\rm Cent}_G(A_L)$ for some torus $A_L \subseteq A$. Let $L^* = {\rm Cent}_{G^*}(A^*_{L^*})$ for a subtorus $A^*_{L^*} \subseteq A^*$. Without loss of generality, we may assume that our inner twist $G \rightarrow G^*$ restricts to give an inner twist $L \rightarrow L^*$.

\begin{prop} \label{t_const_term_prop}
In the above situation, the following diagram commutes:
\begin{equation} \label{t_const_term_eq}
\xymatrix{
\mathcal Z(G^*(F), J^*) \ar[d]_{c^{G^*}_{L^*}} \ar[r]^{\tilde t} & \mathcal Z(G(F), J) \ar[d]_{c^G_L} \\
\mathcal Z(L^*(F), J^*_{L^*}) \ar[r]^{\tilde{t}} & \mathcal Z(L(F), J_L).}
\end{equation}
Taking $L = M$, the diagram shows in order to compute $\tilde{t}$ it is enough to compute it in the case where $G_{\rm ad}$ is anisotropic. In that case if $z \in \mathcal Z(G^*(F), J^*)$, the function $\tilde{t}(z)$ is given by integrating $z$ over the fibers of the Kottwitz homomorphism $\kappa_{G^*}(F)$.
\end{prop}

\begin{proof}
The commutativity boils down to the fact that the quantities (\ref{this_eq}) do not depend on the ambient group $G$.
\end{proof}

\begin{Remark} The final sentence in Proposition \ref{t_const_term_prop} is the key to explicit computation of $\tilde{t}(z)$ given $z$, and is illustrated in $\S\ref{comp_w_Rap}$. This final sentence was incorrectly asserted to hold for the {\em unnormalized} transfer homomorphisms (for special maximal parahoric Hecke algebras) in Prop.~12.3.1 of \cite{HRo}.
\end{Remark}

\small
\bigskip
\obeylines
\noindent
University of Maryland
Department of Mathematics
College Park, MD 20742-4015 U.S.A.
email: tjh@math.umd.edu

\end{document}